\documentclass{article}
\usepackage{amsmath,amscd,amstext,amssymb,amsfonts}
\usepackage{amsthm}
\usepackage{graphics}
\usepackage{latexsym,empheq,mathtools}
\usepackage{bookman}
\numberwithin{equation}{section}
\usepackage{enumerate}
\usepackage{amscd}

\theoremstyle{plain}
\newtheorem{theorem}{Theorem}[section]
\newtheorem{proposition}[theorem]{Proposition}
\newtheorem{lemma}[theorem]{Lemma}
\newtheorem{definition}[theorem]{Definition}
\newtheorem{corollary}[theorem]{Corollary}

\theoremstyle{definition}
\newtheorem{remark}[theorem]{Remark}
\newtheorem{exams}[theorem]{Examples}
\newtheorem{example}[theorem]{Example}

\allowdisplaybreaks[4]

\begin{document}

\newcommand{\nd}{{\ensuremath d}}
\newcommand{\M}{{\mathcal M}_{u,p}}
\newcommand{\Mf}{{\mathcal M}_{\varphi,p}}
\newcommand{\MB}{{\mathcal N}^{s}_{u,p,q}}
\newcommand{\MA}{{\mathcal A}^{s}_{u,p,q}}
\newcommand{\MfA}{{\mathcal A}^{s}_{\varphi,p,q}}
\newcommand{\MfB}{{\mathcal N}^{s}_{\varphi,p,q}}
\newcommand{\MF}{{\mathcal E}^{s}_{u,p,q}}
\newcommand{\MfF}{{\mathcal E}^{s}_{\varphi,p,q}}
\newcommand{\MBa}{{\mathcal N}^{s_1}_{u_1,p_1,q_1}}
\newcommand{\MBb}{{\mathcal N}^{s_2}_{u_2,p_2,q_2}}
\newcommand{\MFa}{{\mathcal E}^{s_1}_{u_1,p_1,q_1}}
\newcommand{\MFb}{{\mathcal E}^{s_2}_{u_2,p_2,q_2}}
\newcommand{\MfBa}{{\mathcal N}^{s_1}_{\varphi_1,p_1,q_1}}
\newcommand{\MfBb}{{\mathcal N}^{s_2}_{\varphi_2,p_2,q_2}}
\newcommand{\MfFa}{{\mathcal E}^{s_1}_{\varphi_1,p_1,q_1}}
\newcommand{\MfFb}{{\mathcal E}^{s_2}_{\varphi_2,p_2,q_2}}

\newcommand{\wvp}{\widetilde{\varphi}}
\newcommand{\B}{\ensuremath{B^s_{p,q}}}
\newcommand{\F}{\ensuremath{F^s_{p,q}}}
\newcommand{\bt}{{B}_{p,q}^{s,\tau}}
\newcommand{\ft}{{F}_{p,q}^{s,\tau}}
\newcommand{\at}{{A}_{p,q}^{s,\tau}}
\newcommand{\btt}{{B}_{p,q}^{s,\varphi}}
\newcommand{\ftt}{{F}_{p,q}^{s,\varphi}}
\newcommand{\btta}{{B}_{p_1,q_1}^{s_1,\varphi_1}}
\newcommand{\ftta}{{F}_{p_1,q_1}^{s_1,\varphi_1}}
\newcommand{\bttb}{{B}_{p_2,q_2}^{s_2,\varphi_2}}
\newcommand{\fttb}{{F}_{p_2,q_2}^{s_2,\varphi_2}}
\newcommand{\Att}{{A}_{p,q}^{s,\varphi}}
\newcommand{\att}{{a}_{p,q}^{s,\varphi}}

\newcommand{\A}{\mathcal{A}}
\newcommand{\AF}{\A^{s}_{\varphi,p,q}}
\newcommand{\mb}{{n}^{s}_{u,p,q}}
\newcommand{\mf}{{e}^{s}_{u,p,q}}
\newcommand{\mfb}{{n}^{s}_{\varphi,p,q}}
\newcommand{\mff}{{e}^{s}_{\varphi,p,q}}
\newcommand{\mfFa}{{e}^{s_1}_{\varphi_1,p_1,q_1}}
\newcommand{\mfFb}{{e}^{s_2}_{\varphi_2,p_2,q_2}}
\newcommand{\mfBa}{{n}^{s_1}_{\varphi_1,p_1,q_1}}
\newcommand{\mfBb}{{n}^{s_2}_{\varphi_2,p_2,q_2}}
\newcommand{\btts}{{b}_{p,q}^{s,\varphi}}
\newcommand{\ftts}{{f}_{p,q}^{s,\varphi}}
\newcommand{\atts}{{a}_{p,q}^{s,\varphi}}

\newcommand{\kjm}{\chi_{j,m}}
\newcommand{\ijm}{I_{J}^j(m)}
\newcommand{\hijm}{\hat{I}_{J}^j(M)}
\newcommand{\lz}{\lambda}
\newcommand{\lzjm}{\lambda_{j,m}}
\newcommand{\mujm}{\mu_{J,M}}
\newcommand{\ajm}{a_{j,m}}
\newcommand{\qjm}{Q_{j,m}}
\newcommand{\qjjm}{Q_{J,M}}
\newcommand{\mq}{\mathcal{Q}}
\newcommand{\wq}{\widetilde{Q}}
\newcommand{\wcc}{\widetilde{C}}
\newcommand{\wgam}{\widetilde{\gamma}}

\newcommand{\dint}{\,\mathrm{d}}
\newcommand{\Dd}{\;\mathrm{D}}
\newcommand{\supp}{\mathrm{supp}\,}
\newcommand{\id}{\mathrm{id}}
\newcommand{\eb}{\hookrightarrow}

\newcommand{\nat}{\ensuremath{\mathbb{N}}}
\newcommand{\no}{\ensuremath{\nat_0}}
\newcommand{\real}{\ensuremath{{\mathbb R}}}
\newcommand{\rd}{\ensuremath{{\mathbb R}^\nd}}
\newcommand{\zd}{\ensuremath{\mathbb{Z}^\nd}}
\newcommand{\sd}{{\mathcal S}(\mathbb{R}^\nd)}
\newcommand{\sdd}{{\mathcal S}'(\mathbb{R}^\nd)}
\newcommand{\Gp}{{\mathcal G}_p}
\newcommand{\mhl}{\mathcal{M}_{\mathrm{HL}}}
\newcommand{\jjp}{j_P\vee0}
\newcommand{\fuff}{(\theta_j\widehat{f})^\vee}
\newcommand{\feff}{(\eta_j\widehat{f})^\vee}
\newcommand{\Lloc}{\ensuremath L_1^{\mathrm{loc}}}

\newcommand{\btr}{\blacktriangleright}
\newcommand{\ljm}{\lambda_{j,m}}
\newcommand{\rphi}[1]{\ensuremath{\mathsf{r}_{#1}}}
\newcommand{\whole}[1]{\ensuremath\lfloor #1 \rfloor}

\title{Embeddings of generalised Morrey smoothness spaces}    

\author{Dorothee D. Haroske\footnotemark[1], Zhen Liu\footnotemark[2], Susana D. Moura\footnotemark[3], and Leszek Skrzypczak\footnotemark[1]}

\footnotetext[1]{The first and last author were partially supported by the German Research Foundation (DFG), Grant no. Ha 2794/8-1.}

\footnotetext[2]{The second author was supported by the China Scholarship Council (CSC), Grant no. 202006350058. }

\footnotetext[3]{The third author was partially supported by the Center for Mathematics of the University of Coimbra-UIDB/00324/2020, funded by the Portuguese Government through FCT/MCTES.}

\maketitle%

\begin{abstract}
  We study embeddings between generalised Triebel-Lizorkin-Morrey spaces ${\mathcal E}^{s}_{\varphi,p,q}({\mathbb R}^d)$ and within the scales of further generalised Morrey smoothness spaces like ${\mathcal N}^{s}_{\varphi,p,q}({\mathbb R}^d)$, ${B}_{p,q}^{s,\varphi}({\mathbb R}^d)$ and ${F}_{p,q}^{s,\varphi}({\mathbb R}^d)$. The latter have been investigated in a recent paper by the first two authors (2023), while the embeddings of the scale ${\mathcal N}^{s}_{\varphi,p,q}({\mathbb R}^d)$ were mainly obtained in a paper of the first and last two authors (2022). Now we concentrate on the characterisation of the spaces ${\mathcal E}^{s}_{\varphi,p,q}({\mathbb R}^d)$.
  Our approach requires a wavelet characterisation of those spaces which we establish for the system of Daubechies' wavelets. Then we prove necessary and sufficient conditions for the embedding ${\mathcal E}^{s_1}_{\varphi_1,p_1,q_1}({\mathbb R}^d)\hookrightarrow {\mathcal E}^{s_2}_{\varphi_2,p_2,q_2}({\mathbb R}^d)$. We can also provide some almost final answer to the question when ${\mathcal E}^{s}_{\varphi,p,q}({\mathbb R}^d)$ is embedded into $C({\mathbb R}^d)$, complementing our recent findings in case of ${\mathcal N}^{s}_{\varphi,p,q}({\mathbb R}^d)$. 
\end{abstract}



\section*{Introduction}

In this paper we study smoothness function spaces built upon generalised Morrey spaces $\Mf(\rd)$, $0<p<\infty$, $\varphi \in \Gp$, that is, $\varphi:(0,\,\infty)\rightarrow(0,\,\infty)$ is nondecreasing and $t^{-d/p}\varphi(t)$ is nonincreasing in $t>0$. 
These spaces generalise the  Morrey spaces $\mathcal{M}_{u,p}(\rd)$, $0 < p \le u < \infty$, and were introduced by  Mizuhara \cite{mi} and Nakai \cite{nak94} in the beginning of the 1990's. The spaces were applied successfully   to PDEs, e.g.  to nondivergence elliptic differential problems, cf. \cite{FHS,KMR,WNTZ}, to parabolic differential equations \cite{ZJSZ} or Schr\"odinger equations \cite{KNS}.  We refer to \cite{Saw18} and the monograph \cite{sdh20b} for further information about the spaces and the historical remarks. 
 
In the recent years, smoothness function spaces built upon Morrey spaces $\mathcal{M}_{u,p}(\rd)$,  in particular Besov-Morrey spaces  $\mathcal{N}^s_{u,p,q}(\rd)$ and Triebel-Lizorkin-Morrey spaces $\mathcal{E}^s_{u,p,q}(\rd)$, $0 < p \le u < \infty$, $0 < q \le \infty$, $s \in \real$, were studied intensively. Netrusov  was the first who combined the Besov and Morrey norms  cf. \cite{Net}. He considered function spaces on domains and proved some embedding theorem, but the spaces attracted further attention mainly by possible applications to PDEs. We refer to the famous paper by Kozono and Yamazaki \cite{ky94} related to the study of the Navier-Stokes equations, as well as to the results by Mazzucato \cite{Maz}, by  Ferreira, Postigo \cite{FP} or by Yang, Fu, Sun \cite{YFS}.

In the general setting, when $\varphi\in\Gp$, the generalised Besov-Morrey spaces  $\MfB(\rd)$  and Triebel-Lizorkin-Morrey spaces $\MfF(\rd)$ were introduced and studied by Nakamura, Noi and Sawano  \cite{nns16}. Note that in case of $\varphi(t) = t^{d/u}$, $0<p\leq u<\infty$, those generalised spaces $\MfB(\rd)$ and $\MfF(\rd)$ coincide with the Besov-Morrey spaces $\mathcal{N}^s_{u,p,q}(\rd)$ and {the Triebel-Lizorkin-Morrey spaces} $\mathcal{E}^s_{u,p,q}(\rd)$, respectively.

Recently, in \cite{hms22,hms23} we studied the necessary and sufficient conditions for the embeddings $\MfBa(\rd)\hookrightarrow \MfBb(\rd)$ and $\MfB(\rd)\hookrightarrow C(\rd)$. In our recent paper \cite{hl23} we introduced and characterised the scales $\btt(\rd)$ and $\ftt(\rd)$ which can be seen as generalisations of the spaces $\bt(\rd)$ and $\ft(\rd)$ when $\varphi(t)=t^{d(\frac1p-\tau)}$, $t>0$, $\tau\in [0,\frac1p]$; we refer to the monograph \cite{ysy10} for an extensive treatment of these scales of spaces. As it is well-known that in case of $\varphi(t)=t^{d/u}$, $0<p\leq u<\infty$, the corresponding scales $\bt(\rd)$, $\ft(\rd)$, $\MfB(\rd)$ and $\MfF(\rd)$ are closely linked for suitable choices of $\tau$ and $u$, we shall also focus on their interrelation in our general approach. {To make this relationship more transparent we  slightly change the notation related to    $\btt(\rd)$ and $\ftt(\rd)$ spaces, used in \cite{hl23}. }

Our main tools are wavelet characterisations and atomic decompositions. Our new results in this context can be found in Section~\ref{sec-wavelet} below, in particular in Theorem~\ref{bwd} (for the spaces $\btt(\rd)$), Theorem~\ref{fwd} (for the spaces $\ftt(\rd)$), and Theorem~\ref{mew} (for the spaces $\MfF(\rd)$). The main clue to obtain it are the known atomic decompositions together with the general wavelet approach in \cite{hst18} -- following the same strategy as already done successfully in \cite{hms22} in case of $\MfB(\rd)$. As we already know, such wavelet characterisations are of much wider interest and applicability than presented in this paper. This will be done in the future.

Comparing the different scales we find { that $\ftt(\rd)=\MfF(\rd)$, see Theorem~\ref{rswm}(i),}
which is the perfect counterpart of the well-known coincidence $\ft(\rd)=\MF(\rd)$ in case of $\tau=\frac1p-\frac1u$, cf. \cite[Corollary~3.3]{ysy10}. In case of the generalised Besov-Morrey smoothness spaces we obtain in Theorem~\ref{rswm}(ii), that $\MfB(\rd)\hookrightarrow \btt(\rd)$, with coincidence if, and only if, $q=\infty$ {or
\begin{equation}\label{neg-int-1}
 \lim_{t\to 0^+} \varphi(t)t^{-\frac{d}{p}} <\infty \qquad\text{and}\qquad  
  \lim_{t\to \infty} \varphi(t)t^{-\frac{d}{p}} > 0.
  \end{equation}
  }
{When $q<\infty$ and {\eqref{neg-int-1} is not satisfied, i.e.,}
\begin{equation*}\label{int-1}
  \lim_{t\to 0^+} \varphi(t)t^{-\frac{d}{p}} = \infty \qquad\text{or}\qquad  
  \lim_{t\to \infty} \varphi(t)t^{-\frac{d}{p}} = 0,
\end{equation*}
then the embedding is proper.}
This result does not only cover the well-known results in case of $\bt(\rd)$ and $\MB(\rd)$, but, in addition, gives a better insight into the r\^ole played by the function $\varphi$ near $0$ and $\infty$. This was not visible before in case of $\varphi(t)=t^{d/u}$, $0<p\leq u<\infty$,  as then the behaviour of $\varphi(t)$ is everywhere the same.

A similar phenomenon appeared related to the embeddings  $\mathcal{N}^s_{p,q,q_0}(\rd) \hookrightarrow \MF(\rd)$ and $\MF(\rd)\hookrightarrow \mathcal{N}^s_{p,q,q_1}(\rd)$, respectively. In the classical case, when $\varphi(t)=t^{d/p}$, $t>0$, it is well-known that
\[ B^s_{p,\min(p,q)}(\rd) \hookrightarrow F^s_{p,q}(\rd) \hookrightarrow B^s_{p,\max(p,q)}(\rd)
\]
and the fine indices $q_0= \min(p,q)$, $q_1=\max(p,q)$ are known to be sharp, cf. \cite{SiT}. However, this changes when $\varphi(t)=t^{d/u}$, $0<p<u<\infty$, $t>0$, as then the sharp result reads as
\[
  \mathcal{N}^s_{p,q,\min(p,q)}(\rd) \hookrightarrow \MF(\rd) \hookrightarrow \mathcal{N}^s_{p,q,\infty}(\rd),
  \]
  cf. \cite{saw08}. Now we found in Theorem~\ref{Nphi-Ephi} that the link between the above two situations is again characterised by the behaviour of $\varphi(t) t^{-d/p}$ near $0$ and $\infty$: the fine index $q_1=\max(p,q)$ can only be obtained when {\eqref{neg-int-1} is satisfied}. Otherwise $q_1=\infty$ is the best possible fine index.

  Dealing with the embedding $\MfFa(\rd)\hookrightarrow \MfFb(\rd)$ we observed that the outcome reads different in the cases $p_1\geq p_2$ or $p_1<p_2$, respectively. 
  We deal with these different cases separately and compare it with the previously known results as well as for special examples. 
  {If $p_1<p_2$, then the fine parameters $q_1$ and $q_2$ do not influence the embedding  which is in good agreement with the classical situation for spaces of type $\F(\rd)$ -- in contrast to spaces $\B(\rd)$.  Moreover, in this case}
  it turns out that the (new) number $\rphi{\varphi}$, measuring the best possible $\Gp$-class for $\varphi$ locally, has some important influence on the characterisation of the embedding. This phenomenon can be observed, too, when dealing with the embedding $\MfF(\rd) \hookrightarrow C(\rd)$ at the end of the paper.

The paper is organised as follows. In  Section~\ref{prelim} we recall the main  definitions and findings -- in case of the spaces $\bt(\rd)$, $\ft(\rd)$, $\MB(\rd)$, $\MF(\rd)$ --, together with some examples. 
In Section~\ref{sec-wavelet} we present the wavelet characterisation of the spaces $\btt(\rd)$, $\ftt(\rd)$ and $\MfF(\rd)$. 
Section~\ref{rel-scales} deals with the  embeddings between the different scales of generalised Morrey smoothness spaces, while Section~\ref{sect-emb} contains our new characterisations for the embeddings of type $\MfFa(\rd)\hookrightarrow \MfFb(\rd)$, $\ftta(\rd)\hookrightarrow \fttb(\rd)$ and, finally, $\MfF(\rd)\hookrightarrow C(\rd)$. We  always illustrate and discuss our results with several concrete examples.

\section{Preliminaries}\label{prelim}
First we fix some notation. 
Let $\rd$ be $d$-dimensional Euclidean space. Let $\nat$ be the set of all natural numbers and $\no:=\nat\cup\{0\}$. Let $\zd$ be the set of all lattice points in $\rd$ having integer components. For $a\in\real$, let $a_+:=\max(a,0)$ and $\lfloor a\rfloor:=\max\{k\in\mathbb{Z}: k\leq a\}$. 
For each cube $Q\subset\rd$, we denote by $\ell(Q)$ the side length of $Q$: $\ell(Q):=|Q|^\frac{1}{d}$, where $|Q|$ denotes the volume of the cube $Q$. 
For $x\in\rd$ and $r\in(0,\infty)$, we denote by $Q(x,r)$ the compact cube centred at $x$ with side length $r$, whose sides are parallel to the axes of coordinates. We write simply $Q(r)=Q(0,r)$ when $x=0$.
Let $\mathcal{Q}$ denote the set of all dyadic cubes in $\rd$, namely, $\mathcal{Q}:=\{Q_{j,k}:=2^{-j}([0,1)^d+k):j\in\mathbb{Z}, k\in\zd\}$. For all $Q\in\mathcal{Q}$, let $j_Q:=-\log_2\ell(Q)$, and let $j_Q\vee 0:=\max(j_Q,\,0)$.
All unimportant positive constants will be denoted by $C$, occasionally the same letter $C$ is used to denote different constants in the same chain of inequalities.
By the notation $A\lesssim B$, we mean that there exists a positive constant $c$ such that $A\leq c\, B$, whereas the symbol $A\sim B$ stands for $A\lesssim B\lesssim A$. Given two (quasi-)Banach spaces $X$ and $Y$, we write $X\eb Y$ if $X\subset Y$ and the natural embedding of $X$ into $Y$ is continuous.

\subsection{ Generalised Morrey spaces}

In \cite{mor38} Morrey introduced the Morrey spaces $\M(\rd)$, $0<p\leq u<\infty$, which are defined to be the set of all functions $f\in L_p^{\rm{loc}}(\rd)$ such that
        \begin{align*}
            \|f\mid \mathcal{M}_{u,p}(\rd)\|:=
                        \sup_{{x\in\rd,r>0}} {r^{\frac{d}{u}}}
            \left({\frac{1}{|{Q(x,r)}|}}\int_{{Q(x,r)}}|f(y)|^p\,\mathrm{d}y\right)^\frac{1}{p}<\infty.
        \end{align*}
These spaces refine the scale of Lebesgue spaces $L_p(\rd)$, $0<p<\infty$, where in the particular case of $p=u$ they coincide.

In this paper, we consider generalised Morrey spaces where the parameter $u$ is replaced by a function $\varphi$ according to the following definition.

\begin{definition} \label{def-gen-Morrey}
 Let $0<p<\infty$ and $\varphi:(0,\,\infty)\rightarrow [0,\,\infty)$ be a function which does not satisfy $\varphi\equiv0$. 
        The generalised Morrey space $\Mf(\rd)$ is defined to be the set of all functions $f\in L_p^{\mathrm{loc}}(\rd)$ such that 
        \begin{align*}
             \|f \mid \Mf(\rd)\|_\star :=\sup_{x\in\rd,r>0} \varphi(r)
             \left(\frac{1}{|{Q}(x,r)|}\int_{{Q}(x,r)} |f(y)|^p \dint y \right)^{\frac{1}{p}}<\infty.
        \end{align*}
\end{definition}

\begin{remark}
The above definition goes back to \cite{nak94}. When $\varphi(t):=t^\frac{d}{u}$ for $t>0$ and $0<p\leq u<\infty$, then $\Mf(\rd)$ coincides with $\M(\rd)$.
Furthermore, when $\varphi(t):=t^\frac{d}{p}$ for $t>0$ and $0<p<\infty$, then $\Mf(\rd)$ coincides with $L_p(\rd)$.
When $\varphi(t):=t^{-\sigma}\chi_{(0,1)}(t)$ where $-\frac{d}{p}\leq \sigma<0$, then $\Mf(\rd)$ coincides with the local Morrey space $\mathcal{L}_p^\sigma(\rd)$ introduced by Triebel in \cite{tri11}, cf. also \cite[Section 1.3.4]{tri13}.
If $\sigma=-\frac{d}{p}$, then the space is a uniform Lebesgue space $\mathcal{L}_p(\rd)$.
Moreover, let $\varphi_0\equiv1$, then $\mathcal{M}_{\varphi_0,p}(\rd)$ coincides with $L_{\infty}(\rd)$ due to Lebesgue's differentiation theorem.
\end{remark}

A natural assumption for $\Mf(\rd)$ is that $\varphi$ belongs to the class $\Gp$, where $\Gp$ is the set of all non-decreasing functions $\varphi:(0,\,\infty)\rightarrow(0,\,\infty)$ such that 
\begin{align}\label{Gp-def}
t^{-\frac{d}{p}}\varphi(t)\geq s^{-\frac{d}{p}}\varphi(s),
\end{align}
for all $0<t\leq s<\infty$.

\begin{remark}
A justification for the use of the class $\Gp$ comes from \cite[Lemma 2.2]{nns16}.  {More precisely, for $p$ and $\varphi$ as in Definition~\ref{def-gen-Morrey}, it holds $\mathcal{M}_{\varphi,p}(\rd)\neq \{0\}$ if, and only if, $\displaystyle\sup_{t>0} \varphi(t)  \min (t^{-\frac{\nd}{p}},1) < \infty$. 
  Moreover, if $\displaystyle\sup_{t>0} \varphi(t)  \min (t^{-\frac{d}{p}},1) < \infty$, then there exists $\varphi^*\in\Gp$ such that $\mathcal{M}_{\varphi,p}(\rd)  = {\mathcal M}_{\varphi^*,p}(\rd) $ in the sense of equivalent (quasi-)norms. We refer to \cite{sdh20b} for the proofs. }
  
 One can easily check that $\mathcal{G}_{p_1}\subset \mathcal{G}_{p_2}$ if $0<p_2\le p_1<\infty$.  We refer the reader to \cite[Section 12.1.2]{sdh20b} for more details about the class $\Gp$. 
  
  Note that $\varphi\in\Gp$ enjoys a doubling property, i.e., $\varphi(r)\leq \varphi(2r)\leq 2^\frac{d}{p}\varphi(r)$, $0<r<\infty$, {therefore} we can define an equivalent quasi-norm in $\Mf(\rd)$ by taking the supremum over the collection of all dyadic cubes, namely,
    \begin{align}\label{Mf-norm}
        \|f \mid \Mf(\rd)\| :=\sup_{P\in \mathcal{Q}} \varphi(\ell(P))
             \left(\frac{1}{|P|}\int_{P} |f(y)|^p \dint y \right)^{\frac{1}{p}}.
    \end{align}
    In the sequel, we always assume that $\varphi\in\Gp$ and consider the quasi-norm \eqref{Mf-norm}.   This choice is natural and covers all interesting cases. {Moreover, we shall usually assume in the sequel that $\varphi(1)=1$.}
\end{remark}

Here we illustrate some examples from \cite{sdh20b}. Other examples can be found e.g. in \cite[Example 3.15]{Saw18}.

\begin{exams}\label{ex-phi}
  \begin{enumerate}[\bfseries\upshape  (i)]
\item Let $u\in\real$, $0<p<\infty$ and let $\varphi(t)=t^u$ for $t>0$. Then $\varphi$ belongs to $\Gp$ if, and only if, $0\leq u\leq \frac{d}{p}$.
\item Let $0<u,\,v<\infty$. Then
\begin{equation}\label{ex-u-v}
\varphi_{u,v}(t)=\begin{cases}
t^{\frac{d}{u}},\quad{\text{if}}\quad t\leq 1\\
t^{\frac{d}{v}},\quad{\text{if}}\quad t> 1
\end{cases}
\end{equation}
belongs to $\Gp$ with $p=\min(u,v)$. 
In particular, the function $\varphi(t)=t^{\frac{d}{u}}$ belongs to $\Gp$ whenever $0<p\leq u<\infty$ by taking $u=v$, and the functions $\varphi(t)=\max(1, t^{\nd/v})$ and $\varphi(t)=\min(t^{\nd/u}, 1)$ belong to $\Gp$ by taking $u=\infty$ and $v=\infty$, respectively. 
\item Let $0<p<\infty$, $a\leq 0$ and {let $L$  be a sufficiently large constant. Then} $\varphi(t)=t^{\frac{d}{p}}(\log(L+t))^a$ for $t>0$ 
belongs to $\Gp$.
\item Let $0<p<\infty$,  {let $u$ be a sufficiently small positive  constant and  }
 $\varphi(t)=t^u(\log(e+t))^{-1}$ for $t>0$. Then $\varphi\notin\Gp$.
\item Let
  \begin{equation}\label{ex-log}
\varphi(t)=\begin{cases} \frac{1}{\log 2} \log(1+t), & 0<t<1, \\ t, & t\geq 1.\end{cases}
  \end{equation}
Then $\varphi \in \mathcal{G}_{\nd}$.
\end{enumerate}
\end{exams}

      \subsection{ Some generalised Morrey smoothness spaces}

Let $\mathcal{S}(\rd)$ be the space of all Schwartz functions on $\rd$ endowed with the classical topology and denote by $\mathcal{S}'(\rd)$ its topological dual, namely, the space of all continuous linear functionals on $\mathcal{S}(\rd)$ endowed with the weak $\ast$-topology.  
If $f\in\sd$, then
\begin{align}\label{ft}
    \widehat{f}(x)=\mathcal{F}f(x)=(2\pi)^{-\frac{d}{2}}\int_{\rd}\mathrm{e}^{-ix\cdot\xi}f(\xi)\,\mathrm{d}\xi,\qquad x\in\rd,
\end{align}
denotes the Fourier transform of $f$. As usual, $f^\vee$ and $\mathcal{F}^{-1}f$ stands for the inverse Fourier transform, which is given by the right-hand side of \eqref{ft} with $i$ instead of $-i$. First defined on $\mathcal{S}(\rd)$, one then extends $\mathcal{F}$ and $\mathcal{F}^{-1}$ to $\mathcal{S}'(\rd)$ in the usual way.

We first define the generalised Besov-type and Triebel-Lizorkin-type spaces.
\begin{definition}
    Let $s\in\real$, $0<p<\infty$, $0<q\leq\infty$ and ${\varphi}\in\Gp$. Let $\theta_0\in\sd$ with 
\begin{align*}
   \theta_0(x)=1\quad\text{if}\quad |x|\leq 1\qquad\text{and}\qquad \theta_0(x)=0\quad\text{if}\quad |x|\geq 3/2,
\end{align*}
and let $\theta_k(x):=\theta_0(2^{-k}x)-\theta_0(2^{-k+1}x)$, $x\in\rd$ and $k\in\nat$.
    \begin{enumerate}[\bfseries\upshape  (i)]
        \item The generalised Besov-type space ${B}_{p,q}^{s,{\varphi}}(\rd)$ is defined to be the set of all $f\in\sdd$ such that
        \begin{align*}
            \|f\mid {B}_{p,q}^{s,{\varphi}}(\rd)\|:=\sup_{P\in\mq} {\frac{\varphi(\ell(P))}{|P|^\frac{1}{p}}}
            \left[\sum_{j=\jjp}^\infty 2^{j s q}\left(\int_P
            |\fuff(x)|^p \dint x\right)^\frac{q}{p}\right]^\frac{1}{q}<\infty
        \end{align*}
        with the usual modification for $q=\infty$.
        \item  
        Assume that there exists $C,\,\varepsilon>0$ such that 
        \begin{align}\label{intc}
            \frac{t^\varepsilon}{\varphi(t)}\leq \frac{C r^\varepsilon}{\varphi(r)}
            \quad\text{holds for}\quad t\geq r,
        \end{align}
        when $0<q<\infty$. 
        The generalised Triebel-Lizorkin-type space ${F}_{p,q}^{s,{\varphi}}(\rd)$ is defined to be the set of all $f\in\sdd$ such that 
        \begin{align*}
            \|f\mid {F}_{p,q}^{s,{\varphi}}(\rd)\|:=\sup_{P\in\mq} {\frac{\varphi(\ell(P))}{|P|^\frac{1}{p}}}
            \left[\int_P\left(\sum_{j=\jjp}^\infty 2^{j s q}
            |\fuff(x)|^q\right)^\frac{p}{q}\dint x\right]^\frac{1}{p}<\infty
        \end{align*}
        with the usual modification for $q=\infty$.
        \item The space ${A}_{p,q}^{s,\varphi}(\rd)$ denotes either $\btt(\rd)$ or $\ftt(\rd)${, assuming} that $\varphi$ satisfies \eqref{intc} when $q<\infty$ and $\Att(\rd)=\ftt(\rd)$.
    \end{enumerate}
\end{definition}

\begin{remark} \label{rmk-tau-spaces}
The spaces ${B}_{p,q}^{s,{\varphi}}(\rd)$ and ${F}_{p,q}^{s,{\varphi}}(\rd)$ are equivalent, respectively, to the spaces ${B}_{p,q}^{s,{\wvp}}(\rd)$ and ${F}_{p,q}^{s,{\wvp}}(\rd)$ introduced in \cite{hl23}, if one replaces $\varphi(\ell(P))^{-1}|P|^\frac{1}{p}$ by a function $\wvp(\ell(P))$ in the above definition, where $\wvp\in\Gp$ and satisfies $t^{\varepsilon-\frac{d}{p}}\wvp(t)\leq C r^{\varepsilon-\frac{d}{p}}\wvp(r)$ for $t\geq r$ when $q<\infty$ in case of the $F$-spaces. 
 In this paper, we consider the above definition of the generalised Besov-type and Triebel-Lizorkin-type spaces for convenience.

 When $\varphi(t):=t^\frac{d}{p}$ for $t>0$, then $B_{p,q}^{s,\varphi}(\rd)$ and $F_{p,q}^{s,\varphi}(\rd)$ coincide with the classical Besov space $B_{p,q}^s(\rd)$ and {the} Triebel-Lizorkin space $F_{p,q}^s(\rd)$, respectively;  see, for example, \cite{t83,t92,t06,t20} for more details about these spaces.
 When {$\varphi(t):=t^{d(\frac{1}{p}-\tau)}$ for $t>0$ and $\tau\in[ 0,\frac{1}{p}]$, then $B_{p,q}^{s,\varphi}(\rd)$ and $F_{p,q}^{s,\varphi}(\rd)$} coincide with, respectively,  the Besov-type space $B_{p,q}^{s,\tau}(\rd)$ and the Triebel-Lizorkin-type space $F_{p,q}^{s,\tau}(\rd)$ introduced in \cite{ysy10}.
\end{remark}

Now we recall {the} generalised Besov-Morrey and Triebel-Lizorkin-Morrey spaces introduced in \cite{nns16}.

\begin{definition} \label{def-E-N-spaces}
Let $s\in\mathbb{R}$, $0<p<\infty$, $0<q\leq\infty$ and $\varphi\in\Gp$. 
Let $\eta_0,\eta\in\mathcal{S}(\rd)$ be non-negative compactly supported functions satisfying
\begin{align*}
        \eta_0(x)>0,\qquad&\text{if}\qquad x\in Q(2) \\
        0\notin\supp \eta\quad\text{and}\quad\eta(x)>0,\qquad&\text{if}\qquad x\in Q(2)\setminus Q(1). 
    \end{align*}
    For $j\in\mathbb{N}$, let $\eta_j(x):=\eta(2^{-j}x)$, $x\in\rd$.
\begin{enumerate}[\bfseries\upshape  (i)]
    \item The {generalised Besov-Morrey space} $\MfB(\rd)$ is defined to be the set of all  $f\in\mathcal{S}'(\rd)$ such that
        \begin{align*}
            \|f\mid \MfB(\rd)\|:=\left(\sum_{j=0}^\infty 2^{j s q} \|\feff\mid \Mf(\rd)\|^q\right)^\frac{1}{q}<\infty
        \end{align*}
        with the usual modification for $q=\infty$.
    \item Assume that $\varphi$ satisfies \eqref{intc} when $q<\infty$. The generalised Triebel-Lizorkin-Morrey space $\MfF(\rd)$ is defined to be the set of all $f\in\mathcal{S}'(\rd)$ such that
         \begin{align*}
            \|f\mid \MfF(\rd)\|:=\left\|\left(\sum_{j=0}^\infty 2^{j s q} |\feff(\cdot)|^q\right)^\frac{1}{q}\mid\Mf(\rd)\right\|<\infty
        \end{align*}
        with the usual modification for $q=\infty$.
    \item The space $\mathcal{A}_{\varphi,p,q}^s(\rd)$ denotes either $\MfB(\rd)$ or $\MfF(\rd)$, assuming  that $\varphi$ satisfies \eqref{intc} when $q<\infty$ and $\mathcal{A}_{\varphi,p,q}^s(\rd)=\MfF(\rd)$.
    \end{enumerate}
\end{definition}

\begin{remark}\label{equnv}
   \begin{enumerate}[\bfseries\upshape  (i)]
  \item  When $\varphi(t):=t^\frac{d}{u}$ for $t>0$ and $0<p\leq u<\infty$, then 
    \begin{align*}
        \MfB(\rd)=\MB(\rd)\qquad\text{and}\qquad\MfF(\rd)=\MF(\rd)
    \end{align*}
    are the usual Besov-Morrey and Triebel-Lizorkin-Morrey spaces introduced in \cite{ky94, tx05}.
    For more details about the spaces $\MB(\rd)$ and $\MF(\rd)$, we refer the reader to the monographs \cite{ysy10, sdh20a} and the survey papers by W. Sickel \cite{sic12, sic13}. 
    We can also recover the classical Besov spaces $B_{p,q}^s(\rd)$ and Triebel-Lizorkin spaces $F_{p,q}^s(\rd)$ for any $0<p<\infty$, $0<q\leq \infty$ and $s\in\real$, since
    \begin{align*}
        B_{p,q}^s(\rd)=\mathcal{N}_{p,p,q}^s(\rd)
        \qquad\text{and}\qquad
        F_{p,q}^s(\rd)=\mathcal{E}_{p,p,q}^s(\rd).
    \end{align*}
    Moreover, let $0<p<u<\infty$, then according to \cite[Proposition 1.3]{saw08}, we have a sharp embedding
    \begin{align}\label{se-au}
        \mathcal{N}_{u,p,\min(p,q)}^s(\rd)\eb\MF(\rd)\eb\mathcal{N}_{u,p,\infty}^s(\rd).
    \end{align}
    The index $\infty$ on the right-hand side of \eqref{se-au} cannot be replaced by any finite number; 
    see \cite[Proposition 1.6]{saw08}. However, when $u=p$, we have
    \begin{align*}
        \mathcal{N}_{p,p,\min(p,q)}^s(\rd)\eb\mathcal{E}_{p,p,\min(p,q)}^s(\rd)
        \eb\mathcal{N}_{p,p,\max(p,q)}^s(\rd),
    \end{align*}
    or equivalently,
    \begin{align}\label{se-bf}
        B_{p,\min(p,q)}^s(\rd)\eb F_{p,q}^s(\rd)\eb B_{p,\max(p,q)}^s(\rd).
    \end{align}
  Moreover,
  \begin{equation}\label{E=Mf}
    \mathcal{E}^{0}_{\varphi,p,2} (\rd)=\Mf(\rd), \quad \varphi\in\Gp, \quad 1<p<\infty,
  \end{equation}
  {if} \eqref{intc} holds, cf. \cite{nns16} with reference to \cite[Proposition~5.1]{SHG-15}, see also \cite[Proposition~3.18]{YZY-BJMA15} in a more general context.  
    This generalises the Mazzucato result \cite[Proposition~4.1]{Maz} related to the case $\varphi(t)=t^\frac{d}{u}$ for $t>0$, which reads as
\[
\mathcal{E}^0_{u,p,2}(\rd)=\M(\rd),\quad 1<p\leq u<\infty.
\]
In particular, this covers the well-known classical outcome that 
\begin{equation}\label{E-Lp}
\mathcal{E}^0_{p,p,2}(\rd)=L_p(\rd)=F^0_{p,2}(\rd),\quad 1<p<\infty,
\end{equation}
always understood in the sense of equivalent norms.

\item
 If $\varphi_0\equiv1$, then   $\mathcal{M}_{\varphi_0,p}(\rd)=L_{\infty}(\rd)$  for any $0<p<\infty$. Therefore, in such a case and for any $0<q\leq\infty$ and $s\in\real$, it holds
    \begin{align*}
     \mathcal{N}_{\varphi_0,p,q}^s(\rd) = B_{\infty,q}^s(\rd)
        \quad\text{and}\quad
 \mathcal{E}_{\varphi_0,p,\infty}^s(\rd)= F_{\infty,\infty}^s(\rd)= B_{\infty,\infty}^s(\rd)= \mathcal{N}_{\varphi_0,p,\infty}^s(\rd).
    \end{align*}
Note that   $\varphi_0\equiv1$ does not satisfy condition  \eqref{intc}, hence  $\mathcal{E}_{\varphi_0,p,q}^s(\rd)$ is not defined when $0<q<\infty$. This is coherent with the well-known fact that the  definition of the Triebel-Lizorkin spaces 
$F_{p,q}^s(\rd)$ in case of $p=\infty$ require a different approach from the one in Definition~\ref{def-E-N-spaces}. We refer to Triebel's monograph \cite{t20} for more information about the spaces $F_{\infty,q}^s(\rd)$. But this setting is out of the scope of the present paper.

\item     Let $s\in\real$, $0<p<\infty$, $0<q_1, q_2\leq\infty$, $\varphi\in\Gp$ with \eqref{intc} in case of $\MfF=\MfA$ and $q<\infty$. Then  we also have the following elementary embeddings,
    \begin{align*}
        \A_{\varphi,p,q_1}^s(\rd)\eb\A_{\varphi,p,q_2}^s(\rd),\qquad q_1\leq q_2,
    \end{align*}
    and
    \begin{align*}
        \A_{\varphi,p,q_1}^{s+\varepsilon}(\rd)\eb\A_{\varphi,p,q_2}^s(\rd),\qquad \varepsilon>0,
    \end{align*}
    see \cite[Proposition 3.3]{nns16} for more details.
    \end{enumerate}
\end{remark}

\section{Wavelet decomposition}\label{sec-wavelet}
Let $L\in\nat$ and let $\psi_F,\,\psi_M\in C^L(\real)$ are real-valued compactly supported ($L^2$-normalised) functions with
\begin{align}\label{wave1}
    \int_\real \psi_F^2(t)\dint t=C,\qquad \int_\real t^\ell\psi_M(t)\dint t=0,\quad \ell<L,
\end{align}
for some constant $C>0$. 
The function $\psi_F$ is called scaling function (or father wavelet) and $\psi_M$ is called associated function
(or mother wavelet).

Let $G=(G_1,\cdots,G_d)\in G^\ast=\{F,\,M\}^{d\ast}$, where $\ast$ indicates that at least one of the
components of $G$ must be an $M$. Then we set
\begin{align}\label{wave2}
    \psi_{j,m}^G(x):=2^{\frac{jd}{2}}\prod_{r=1}^d\psi_{G_r}(2^jx_r-m_r),\qquad
    \psi_m(x):=\prod_{r=1}^d\psi_F(x_r-m_r), \qquad x\in \rd,
\end{align}
where $j\in\no$, $m\in\zd$, $G\in G^\ast$. The family $\{\psi_m,\,\psi_{j,m}^G: j\in\no,m\in\zd,G\in G^\ast\}$
is called a (Daubechies) wavelet system.

To obtain the wavelet decomposition of ${B}_{p,q}^{s,{\varphi}}(\rd)$, ${F}_{p,q}^{s,{\varphi}}(\rd)$ and $\MfF(\rd)$, 
we need the following notion of $\varkappa$-sequence spaces. We start with some notation.
Let $m\in\zd$, $j,\,J\in\no$, $r>1$ and $C>0$. Let
\begin{align*}
    I_J^j(m)=\{M\in\zd: r\,Q_{J,M}\cap C\,Q_{j,m}\neq \emptyset\}.
\end{align*}
Note that the cardinality of $I_J^j(m)$ satisfies
\begin{align}\label{ijm}
    \sharp I_J^j(m)\sim
\begin{cases}
    1,&\quad J\leq j,         \\
    2^{d(J-j)},&\quad J>j,     
\end{cases}
\end{align}
with the implicit constants independent of $j,\,J\in\no$ and $m\in\zd$.

\begin{definition}\label{vks}
Let $\varkappa>0$. Then a quasi-normed space $a(\rd)$, consisting of sequences
$$\mu=\{\mu_{jm}\in\mathbb{C}:j\in\no,m\in\zd\}$$
 with $\|\mu \mid a(\rd)\|<\infty$,
 is called a $\varkappa$-sequence space if
\begin{enumerate}[\bfseries\upshape  (i)]
\item for any $r>1$, $C>0$, and all $\mu\in a(\rd)$ any sequence $\lambda=\{\lzjm\in\mathbb{C}:j\in\no,m\in\zd\}$ with
\begin{equation}\label{vke1}
|\lzjm|\leq C\sum_{J\in\no}2^{-\varkappa|J-j|}\sum_{M\in I_J^j(m)}2^{-d(J-j)_+}|\mujm|,\quad j\in\no,\,\,m\in\zd,
\end{equation}
belongs to $a(\rd)$ and satisfies
\begin{equation}\label{vke2}
\|\lambda\mid a(\rd)\|\leq c\,\|\mu\mid a(\rd)\|
\end{equation}
for some $c>0$ which may depend on $r$, $C$, $\varkappa$ and $d\in\nat$;
\item for any cube $Q$ there is a constant $C_Q>0$ such that for all $\mu\in a(\rd)$,
\begin{equation}\label{vke3}
|\mujm|\leq C_Q 2^{J\varkappa}\|\mu\mid a(\rd)\|\qquad \text{for all}\,\,\,J\in\no\,\,\text{and}\,\,
M\in\zd\,\,\text{with}\,\,Q_{J,M}\subset Q.
\end{equation}
\end{enumerate}
\end{definition}

\begin{remark}
    The above definition was  introduced in \cite{hst18}, which is a tool to show how to construct compactly supported wavelets if one has an atomic decomposition, cf.  \cite[Theorem 5.1]{hst18}. { Besov sequence spaces $b_{p,q}^s(\rd)$ and  Triebel-Lizorkin sequence spaces $f_{p,q}^s(\rd)$ are examples of  $\varkappa$-sequence spaces, cf. \cite[Propositions 6.1 and 6.5]{hst18}.}
\end{remark} 

    We  recall the following notation for later use:
    \begin{align*}
    \hat{I}_J^j(M):=\{m\in\zd: r\,Q_{J,M}\cap C\,Q_{j,m}\neq \emptyset\}.
    \end{align*}
    The cardinality of $\hat{I}_J^j(M)$ satisfies
    \begin{align*}
    \sharp \hijm\sim
    \begin{cases}
    1,&\quad j\leq J,         \\
    2^{d(j-J)},&\quad j>J.     
    \end{cases}
    \end{align*}

\subsection{ Wavelet decomposition of ${B}_{p,q}^{s,{\varphi}}(\rd)$ and ${F}_{p,q}^{s,{\varphi}}(\rd)$}

We first recall the corresponding sequence spaces for ${B}_{p,q}^{s,{\varphi}}(\rd)$ and ${F}_{p,q}^{s,{\varphi}}(\rd)$. For $j\in\no$ and $m\in\zd$ we denote by $\chi_{j,m}$ the characteristic function of the cube $\qjm$.
\begin{definition}[\cite{hl23}]  \label{def-seq-b-f}
    Let $s\in\real$, $0<p<\infty$, $0<q\leq\infty$ and ${\varphi}\in\Gp$.
    \begin{enumerate}[\bfseries\upshape  (i)]
        \item The sequence space ${b}_{p,q}^{s,{\varphi}}(\rd)$ is defined to be the set of all sequences $\lz:=\{\lzjm\}_{j\in\no,m\in\zd}$ such that
        \begin{align*}
            \|\lz\mid {b}_{p,q}^{s,{\varphi}}(\rd)\|:=
            \sup_{P\in\mq} {\frac{\varphi(\ell(P))}{|P|^\frac{1}{p}}} 
            \left\{\sum_{j=\jjp}^\infty \left[\sum_{\substack{m\in\zd\\\qjm\subset P}} (2^{j (s-\frac{d}{p})}
            |\lzjm|)^p\right]^\frac{q}{p}\right\}^\frac{1}{q}
        \end{align*}
        is finite (with the usual modification for $q=\infty$).
        \item Assume in addition that $\varphi$ satisfies \eqref{intc} when $0<q<\infty$. The sequence space ${f}_{p,q}^{s,{\varphi}}(\rd)$ is defined to be the set of all sequences $\lz:=\{\lzjm\}_{j\in\no,m\in\zd}$ such that
        \begin{align*}
            \|\lz\mid {f}_{p,q}^{s,{\varphi}}(\rd)\|:=\sup_{P\in\mq} {\frac{\varphi(\ell(P))}{|P|^\frac{1}{p}}}   
            \left\{\int_P\left[\sum_{j=\jjp}^\infty\sum_{\substack{m\in\zd\\\qjm\subset P}}
            (2^{j s}|\lzjm|\kjm(x))^q\right]^\frac{p}{q}\,\mathrm{d}x\right\}^\frac{1}{p}
        \end{align*}
        is finite (with the usual modification for $q=\infty$).
        \item The sequence space ${a}_{p,q}^{s,{\varphi}}(\rd)$ denotes either ${b}_{p,q}^{s,{\varphi}}(\rd)$ or ${f}_{p,q}^{s,{\varphi}}(\rd)$, assuming in addition that $\varphi$ satisfies \eqref{intc} when $0<q<\infty$ and ${a}_{p,q}^{s,{\varphi}}(\rd)$ denotes ${f}_{p,q}^{s,{\varphi}}(\rd)$.
    \end{enumerate}
\end{definition}

Next we recall the atomic decomposition for ${B}_{p,q}^{s,{\varphi}}(\rd)$ and ${F}_{p,q}^{s,{\varphi}}(\rd)$ obtained in \cite{hl23}.
\begin{definition}\label{defi-atom}
    Let $c>1$, $L\in\no\cup\{-1\}$ and $K\in\no$. A $C^K$-function $\ajm:\rd\rightarrow\mathbb{C}$ is said to be a $(K,L,c)$-atom supported near $\qjm$ with $j\in\no$ and $m\in\zd$, if
    \begin{align*}
     2^{-j |\alpha|}|{\Dd}^\alpha \ajm(x)|\leq \chi_{c \qjm}(x)  
    \end{align*}
    for all $x\in\rd$ and for all $\alpha\in\no^d$ with $|\alpha|\leq K$ and
    \begin{align*}
        \int_{\rd}x^\beta \ajm(x)\,\mathrm{d} x=0
    \end{align*}
    for all $\beta\in\no^d$ with $|\beta|\leq L$ when $L\geq 0$ and $j\in\nat$.
\end{definition}

We use the notation
\begin{align*}
    \sigma_p:=d\left(\frac{1}{p}-1\right)_+\qquad\text{and}\qquad
    \sigma_{p,q}:=\max(\sigma_p,\sigma_q)
\end{align*}
in the sequel.

\begin{theorem}[\cite{hl23}]\label{atde}
    Let $s\in\real$, $0<p<\infty$, $0<q\leq\infty$ and ${\varphi}\in\Gp$. Let also $c>1$, $L\in\no\cup\{-1\}$ and $K\in\no$. Assume that
    \begin{align*}
\begin{cases}
    K\geq \lfloor 1+s\rfloor_+,\quad L\geq \max(-1,\lfloor\sigma_p-s\rfloor),&\qquad\text{if}\qquad A_{p,q}^{s,\varphi}(\mathbb{R}^d)=B_{p,q}^{s,\varphi}(\mathbb{R}^d);         \\
    K\geq \lfloor 1+s\rfloor_+,\quad L\geq \max(-1,\lfloor\sigma_{p,q}-s\rfloor),&\qquad\text{if}\qquad A_{p,q}^{s,\varphi}(\mathbb{R}^d)=F_{p,q}^{s,\varphi}(\mathbb{R}^d).     
\end{cases}
    \end{align*}
    Assume in addition that $\varphi$ satisfies \eqref{intc} when $0<q<\infty$ and $A_{p,q}^{s, {\varphi}}(\rd)=F_{p,q}^{s, {\varphi}}(\rd)$, $a_{p,q}^{s, {\varphi}}(\rd)=f_{p,q}^{s, {\varphi}}(\rd)$.
    \begin{enumerate}[\bfseries\upshape  (i)]
        \item Let $f\in A_{p,q}^{s, {\varphi}}(\rd)$. Then there exist a family 
        $\{\ajm\}_{j\in\no,m\in\zd}$ of $(K,L,c)$-atoms and a sequence $\lz=\{\lzjm\}_{j\in\no,m\in\zd}\in a_{p,q}^{s, {\varphi}}(\rd)$ such that
        \begin{align*}
            f=\sum_{j=0}^\infty\sum_{m\in\zd}\lzjm \ajm\quad\text{in}\quad\sdd
        \end{align*}
        and
        \begin{align*}
            \|\lz\mid a_{p,q}^{s, {\varphi}}(\rd)\|\lesssim\|f\mid A_{p,q}^{s, {\varphi}}(\rd)\|.
        \end{align*}
        \item Let $\{\ajm\}_{j\in\no,m\in\zd}$ be a family of $(K,L,c)$-atoms and 
        $\lz=\{\lzjm\}_{j\in\no,m\in\zd}\in a_{p,q}^{s, {\varphi}}(\rd)$. Then
        \begin{align*}
            f=\sum_{j=0}^\infty\sum_{m\in\zd}\lzjm \ajm
        \end{align*}
        converges in $\sdd$ and belongs to $A_{p,q}^{s, {\varphi}}(\rd)$. Furthermore,
        \begin{align*}
            \|f\mid A_{p,q}^{s, {\varphi}}(\rd)\|\lesssim\|\lz\mid a_{p,q}^{s, {\varphi}}(\rd)\|.
        \end{align*}
    \end{enumerate}
\end{theorem}

We will need the following modified version $\tilde{a}_{p,q}^{s,{\varphi}}(\rd)$  of $a_{p,q}^{s,{\varphi}}(\rd)$ spaces. The space $\tilde{a}_{p,q}^{s,{\varphi}}(\rd)$   collects all sequences 
\[
\lambda= \left\{ \lambda_m\in \mathbb{C}, \lambda_{j,m}^G \in \mathbb{C} : \; m\in \zd, \; j \in \mathbb{N}_0,\; G\in G^* \right\}
\]
such that
\begin{equation*}\label{norm-tilde-a}
\| \lambda\mid \tilde{a}_{p,q}^{s,{\varphi}}(\rd)\|:= \bigg\| \sum_{m\in\zd} \lambda_{m} \chi_{Q_{m}} \,\Big|\, \mathcal{M}_{\varphi,p}(\rd) \bigg\|  + \sum_{G\in G^*} \big\|
\{\lambda_{j,m}^G\}_{j\in\no, m\in\zd} \mid a_{p,q}^{s,{\varphi}} (\rd)\big\| 
\end{equation*} 
is finite.

Now we show the main results.
\begin{theorem}\label{bwd}
    Let $0<p<\infty$, $0<q\leq\infty$, $s\in\real$ and ${\varphi}\in\Gp$. For the wavelets defined in \eqref{wave2} we take
\begin{equation}\label{lsp}
L>\max(\lfloor 1+s\rfloor_+,\,\frac{3d}{p}-s).
\end{equation}
Let $f\in\mathcal{S}'(\rd)$. Then $f\in B_{p,q}^{s, {\varphi}}(\rd)$ if, and only if, it can be represented as
\begin{align*}
    f=\sum_{m\in\zd}\lambda_m\psi_m+\sum_{G\in G^\ast}\sum_{j\in\no}\sum_{m\in\zd}\lzjm^{G}2^{-\frac{jd}{2}}\psi_{j,m}^G,
    \qquad\lambda\in \tilde{b}_{p,q}^{s,{\varphi}}(\rd),
\end{align*}
and {the series} converges unconditionally in $\mathcal{S}'(\rd)$. 
The representation is unique with
\begin{align*}
    \lzjm^G=\lzjm^G(f)=2^{\frac{jd}{2}}(f,\,\psi_{j,m}^G)\quad\text{and}\quad\lambda_m=\lambda_m(f)=(f,\,\psi_m),
\end{align*}
and
\begin{align*}
    I:f\mapsto\{\lambda_m{f},\,2^{\frac{jd}{2}}(f,\,\psi_{j,m}^G)\}
\end{align*}
is a linear isomorphism of $B_{p,q}^{s, {\varphi}}(\rd)$ onto $\tilde{b}_{p,q}^{s,{\varphi}}(\rd)$.

Furthermore, $\|I(f)\mid \tilde{b}_{p,q}^{s,{\varphi}}(\rd)\|$ may be used as an equivalent quasi-norm in ${B}_{p,q}^{s,{\varphi}}(\rd)$.
\end{theorem}

\begin{proof}
    Since the space $B_{p,q}^{s, {\varphi}}(\rd)$ is an (isotropic, inhomogeneous) quasi-Banach {function} space satisfying $\sd\eb B_{p,q}^{s, {\varphi}}(\rd)\eb\sdd$ (cf. \cite[Theorem 4.10]{hl23}) 
    and the inequality $L>3d/p-s$ implies $L>\sigma_p-s$, which means that $B_{p,q}^{s, {\varphi}}(\rd)$ can be characterised in terms of an $L$-atomic decomposition with $L=K$ and coefficients in $b_{p,q}^{s, {\varphi}}(\rd)$ (cf. Theorem~\ref{atde}), 
    by Theorem 5.1 in \cite{hst18}, we only need to prove that the sequence space $b_{p,q}^{s, {\varphi}}=b_{p,q}^{s, {\varphi}}(\rd)$ is a $\varkappa$-sequence space for some $\varkappa$, 
    $0<\varkappa<L$ (cf. Definition~\ref{vks}). 
    The proof is similar to the proof of {Proposition 6.1 in \cite{hst18}}. We give some details for the reader's convenience.

    \emph{Step 1.} {Let  a sequence $\lambda$ satisfy the condition  \eqref{vke1}. } We divide {the {right}-hand side of  the estimate} \eqref{vke1} into two parts,
    \begin{align}\label{bwem}
        |\lzjm|\lesssim \sum_{J=0}^{j}2^{-(j-J)\varkappa}\sum_{M\in\ijm}|\mujm|
        +\sum_{J>j}2^{-(J-j)(\varkappa+d)}\sum_{M\in\ijm}|\mujm|.
    \end{align}
    If $0<p\leq 1$, then by the monotonicity of the $\ell^q$-norm in $q$ and \eqref{ijm}, we have 
    \begin{align}\label{bwe1}
        \sum_{m\in\zd}|\lzjm|^p&\lesssim \sum_{J=0}^j 2^{-(j-J) \varkappa p}\sum_{m\in\zd}\sum_{M\in\ijm}|\mujm|^p \notag\\
        &\quad +\sum_{J>j}2^{-(J-j)(\varkappa+d)p}\sum_{m\in\zd}\sum_{M\in\ijm}|\mujm|^p \notag\\
        &\lesssim \sum_{J=0}^j 2^{-(j-J) \varkappa p}\sum_{M\in\zd}|\mujm|^p\sum_{m\in\hijm}1 \notag\\
        &\quad +\sum_{J>j}2^{-(J-j)(\varkappa+d)p}\sum_{M\in\zd}|\mujm|^p\sum_{m\in\hijm}1 \notag\\
        &\lesssim \sum_{J=0}^j 2^{-(j-J)(\varkappa-\frac{d}{p})p}\sum_{M\in\zd}|\mujm|^p \notag\\
        &\quad +\sum_{J>j}2^{-(J-j)(\varkappa+d)p}\sum_{M\in\zd}|\mujm|^p.
    \end{align}
    If $1<p<\infty$, by applying the H\"older inequality twice and \eqref{ijm}, we obtain for some $\varepsilon>0$,
    \begin{align}\label{bwe2}
        &\sum_{J=0}^j 2^{-(j-J)\varkappa(j-J)}\sum_{M\in\ijm}|\mujm| \notag\\
        &\quad \lesssim \left[\sum_{J=0}^j 2^{-(j-J)(\varkappa-\varepsilon)p}\left(\sum_{M\in\ijm}|\mujm|\right)^p\right]^\frac{1}{p}
        \left(\sum_{J=0}^j2^{-(j-J)\varepsilon p'}\right)^\frac{1}{p'} \notag\\
        &\quad \lesssim \left[\sum_{J=0}^j 2^{-(j-J)(\varkappa-\varepsilon)p}
        \sum_{M\in\ijm}|\mujm|^p\left(\sum_{M\in\ijm} 1^{p'}\right)^\frac{p}{p'}\right]^\frac{1}{p} \notag\\
        &\quad \lesssim \left(\sum_{J=0}^j 2^{-(j-J)(\varkappa-\varepsilon)p}
        \sum_{M\in\ijm}|\mujm|^p\right)^\frac{1}{p}.
    \end{align}
    Similar to \eqref{bwe2}, we have 
    \begin{align}\label{bwe3}
        &\sum_{J>j} 2^{-(J-j)(\varkappa+d)}\sum_{M\in\ijm}|\mujm| \notag\\
        &\quad \lesssim \left[\sum_{J>j} 2^{-(J-j)(\varkappa+d-\varepsilon)p}\sum_{M\in\ijm}|\mujm|^p
        \left(\sum_{M\in\ijm}1^{p'}\right)^\frac{p}{p'}\right]^\frac{1}{p} \notag\\
        &\quad \lesssim \left(\sum_{J>j} 2^{-(J-j)(\varkappa-\varepsilon+\frac{d}{p})p}\sum_{M\in\ijm}|\mujm|^p\right)^\frac{1}{p}.
    \end{align}
    Thus, using \eqref{ijm} again, we conclude that
    \begin{align}\label{bwe4}
        \sum_{m\in\zd}|\lzjm|^p&\lesssim \sum_{J=0}^j2^{-(j-J)(\varkappa-\varepsilon-\frac{d}{p})p}\sum_{M\in\zd}|\mujm|^p \notag\\
        &\quad +\sum_{J>j}2^{-(J-j)(\varkappa-\varepsilon+\frac{d}{p})p}\sum_{M\in\zd}|\mujm|^p.
    \end{align}
    With the aid of the notation $\sigma_p=d(\frac{1}{p}-1)_+$, {we combine \eqref{bwe1} with \eqref{bwe4} and get}  
    \begin{align}\label{bwe5}
        \sum_{m\in\zd}|\lzjm|^p&\lesssim \sum_{J=0}^j2^{-(j-J)(\varkappa-\varepsilon-\frac{d}{p})p}\sum_{M\in\zd}|\mujm|^p \notag\\
        &\quad +\sum_{J>j}2^{-(J-j)(\varkappa-\varepsilon+\frac{d}{p}-\sigma_p)p}\sum_{M\in\zd}|\mujm|^p.
    \end{align}

    Assume that $0<q\leq p$, then
    \begin{align*}
        &2^{j (s-\frac{d}{p}) q}\left(\sum_{m\in\zd}|\lzjm|^p\right)^\frac{q}{p}\\
        &\quad \lesssim \sum_{J=0}^j 2^{(j-J)(s-\frac{d}{p}) q} \,2^{-(j-J)(\varkappa-\varepsilon-\frac{d}{p}) q}
        \,2^{J (s-\frac{d}{p}) q}\left(\sum_{M\in\zd}|\mujm|^p\right)^\frac{q}{p}\\
        &\quad\quad +  \sum_{J>j} 2^{(j-J)(s-\frac{d}{p}) q} \,2^{-(J-j)(\varkappa-\varepsilon+\frac{d}{p}-\sigma_p) q}
        \,2^{J (s-\frac{d}{p}) q}\left(\sum_{M\in\zd}|\mujm|^p\right)^\frac{q}{p}\\
        &\quad = \sum_{J=0}^j 2^{-(j-J)(\varkappa-\varepsilon-s) q}
        \,2^{J (s-\frac{d}{p}) q}\left(\sum_{M\in\zd}|\mujm|^p\right)^\frac{q}{p}\\
        &\quad\quad +  \sum_{J>j} 2^{-(J-j)(\varkappa-\varepsilon+s-\sigma_p) q}
        \,2^{J (s-\frac{d}{p}) q}\left(\sum_{M\in\zd}|\mujm|^p\right)^\frac{q}{p}.
    \end{align*}
    Thus, we have
    \begin{align*}
        &\frac{\varphi(\ell(P))^q}{|P|^\frac{q}{p}}\sum_{j=\jjp}^\infty 2^{j (s-\frac{d}{p}) q}\left(\sum_{\substack{m\in\zd\\\qjm\subset P}}|\lzjm|^p\right)^\frac{q}{p}\\
        &\quad \lesssim \frac{\varphi(\ell(P))^q}{|P|^\frac{q}{p}}\sum_{j=\jjp}^\infty\sum_{J=0}^j 2^{-(j-J)(\varkappa-\varepsilon-s) q}
        \,2^{J (s-\frac{d}{p}) q}\left(\sum_{\substack{M\in\zd\\Q_{J,M}\subset P}}|\mujm|^p\right)^\frac{q}{p}\\
        &\quad\quad +  \frac{\varphi(\ell(P))^q}{|P|^\frac{q}{p}}\sum_{j=\jjp}^\infty\sum_{J>j} 2^{-(J-j)(\varkappa-\varepsilon+s-\sigma_p) q}
        \,2^{J (s-\frac{d}{p}) q}\left(\sum_{\substack{M\in\zd\\Q_{J,M}\subset P}}|\mujm|^p\right)^\frac{q}{p}\\
        &\quad \lesssim \frac{\varphi(\ell(P))^q}{|P|^\frac{q}{p}}\sum_{J=\jjp}^\infty 2^{J (s-\frac{d}{p}) q} 
        \left(\sum_{\substack{M\in\zd\\Q_{J,M}\subset P}}|\mujm|^p\right)^\frac{q}{p}
        \sum_{j\geq J} 2^{-(j-J)(\varkappa-\varepsilon-s) q}\\
        &\quad\quad +  \frac{\varphi(\ell(P))^q}{|P|^\frac{q}{p}}\sum_{J=\jjp}^\infty 2^{J (s-\frac{d}{p}) q} 
        \left(\sum_{\substack{M\in\zd\\Q_{J,M}\subset P}}|\mujm|^p\right)^\frac{q}{p}
        \sum_{j< J} 2^{-(J-j)(\varkappa-\varepsilon+s-\sigma_p) q}\\
        &\quad \lesssim \|\mu\mid\btts(\rd)\|^q
    \end{align*}
    if we choose $\varepsilon>0$ such that $\varkappa-\varepsilon+s-\sigma_p>0$ and $\varkappa-\varepsilon-s>0$. 
    This is always possible if $\varkappa>\max(\sigma_p-s,\,s)$, that is, we need $L>\max(\sigma_p-s,\,s)$ here which is implied by \eqref{lsp}.
    In the case of $p<q\leq\infty$, by a similar argument as above and choosing $\varepsilon$ sufficiently small such that $2\varepsilon<\varkappa-\max(\sigma_p-s,\,s)$, we finish the proof of Definition~\ref{vks}{(i)}.

    \emph{Step 2.} Let $P$ be some dyadic cube, $J\in\no$ and $M\in\zd$ such that $Q_{J,M}\subset P$. 
    Since $2^{-J}\leq \ell(P)$ and $2^{-J}\leq 1$, $J\in\no$, then
    $\varphi(\ell(P))\geq\varphi(2^{-J})\geq 2^{-J\frac{d}{p}}$.
    Thus, we have
    \begin{align*}
        &\frac{\varphi(\ell(P))}{|P|^\frac{1}{p}}\cdot\sup_{\nu\in\no\cap[j_P,\infty)}2^{\nu (s-\frac{d}{p})}\left(\sum_{k\in\zd}|\mu_{\nu,k}|^p\right)^\frac{1}{p}\\
        &\qquad\geq \ell(P)^{-\frac{d}{p}}2^{-J\frac{d}{p}}\cdot2^{J(s-\frac{d}{p})}|\mujm|\\
        &\qquad= \ell(P)^{-\frac{d}{p}}2^{J(s-\frac{2d}{p}+\varkappa)}2^{-J \varkappa}|\mujm|\\
        &\qquad\geq \ell(P)^{-(s-\frac{3d}{p}+\varkappa)}2^{-J \varkappa}|\mujm|\\
        &\qquad=: \frac{1}{C}\,2^{-J \varkappa}|\mujm|.
    \end{align*}
    This is always possible by choosing $\varkappa$ such that $L>\varkappa>(\frac{3d}{p}-s)_+$. Hence, we have
    \begin{align*}
        |\mujm|\leq C\, 2^{J \varkappa}\|\mu\mid b_{p,\infty}^{s,{\varphi}}\|
        \leq C\, 2^{J \varkappa}\|\mu\mid b_{p,q}^{s,{\varphi}}\|.
    \end{align*}
    This {confirms condition (ii) in Definition~\ref{vks} and thus} completes the proof of Theorem~\ref{bwd}.
\end{proof}

\begin{theorem}\label{fwd}
    Let $0<p<\infty$, $0<q\leq\infty$, $s\in\real$ and ${\varphi}\in\Gp$. Assume in addition that  $\varphi$ satisfies \eqref{intc} when $q<\infty$. For the wavelets defined in \eqref{wave2} we take
\begin{equation}\label{lspq}
L>\max(\lfloor 1+s\rfloor_+,\,\frac{3d}{p}-s,\,\sigma_{p,q}-s).
\end{equation}
Let $f\in\mathcal{S}'(\rd)$. Then $f\in F_{p,q}^{s, {\varphi}}((\rd)$ if, and only if, it can be represented as
\begin{align*}
    f=\sum_{m\in\zd}\lambda_m\psi_m+\sum_{G\in G^\ast}\sum_{j\in\no}\sum_{m\in\zd}\lzjm^{G}2^{-\frac{jd}{2}}\psi_{j,m}^G,
    \qquad\lambda\in {\tilde{f}}_{p,q}^{s,{\varphi}}(\rd),
\end{align*}
{and the series} converges unconditionally in $\mathcal{S}'(\rd)$. The representation is unique with
\begin{align*}
    \lzjm^G=\lzjm^G(f)=2^{\frac{jd}{2}}(f,\,\psi_{j,m}^G)\quad\text{and}\quad\lambda_m=\lambda_m(f)=(f,\,\psi_m),
\end{align*}
and
\begin{align*}
    I:f\mapsto\{\lambda_m{f},\,2^{\frac{jd}{2}}(f,\,\psi_{j,m}^G)\}
\end{align*}
is a linear isomorphism of $F_{p,q}^{s, {\varphi}}((\rd)$ onto ${\tilde{f}}_{p,q}^{s,{\varphi}}(\rd)$.

Furthermore, $\|I(f)\mid {\tilde{f}}_{p,q}^{s,{\varphi}}(\rd)\|$ may be used as an equivalent quasi-norm in $F_{p,q}^{s, {\varphi}}((\rd)$.
\end{theorem}

{ To prove} Theorem~\ref{fwd}, we need the following lemma.

\begin{lemma}\label{hlp0}
    Let $0<p<\infty$, $0<q\leq\infty$, $s\in\real$, $0<r<\min(p,q)$ and $\varphi\in\Gp$. 
    Assume in addition that $\varphi$ satisfies \eqref{intc} when $q<\infty$. Then there exists a {positive} constant $C$ such that for any sequence $\{f_k\}_{k\in\no}$ of $L_p^{\varphi}(\rd)$-functions,
    \begin{align*}
        \left\|\left(\sum_{j=\jjp}^\infty\mhl(|f_k|^r)(\cdot)^\frac{q}{r}\right)^\frac{1}{q}\mid L_p^{\varphi}(\rd)\right\|
        \leq C \,\left\|\left(\sum_{j=\jjp}^\infty|f_k|^q\right)^\frac{1}{q}\mid L_p^{\varphi}(\rd)\right\|,
    \end{align*}
with the usual modification made when $q=\infty$,  where $L_p^{\varphi}(\rd)$ is given by
    \begin{align*}
    L_p^{\varphi}(\rd):=\left\{g\in L_p^{\mathrm{loc}}(\rd):\,
        \|g\mid L_p^{\varphi}(\rd)\|:=\sup_{P\in\mq,|P|\geq 1}\frac{\varphi(\ell(P))}{|P|^\frac{1}{p}}
        \left(\int_P|g(x)|^p\dint x\right)^\frac{1}{p}<\infty\right\}
    \end{align*}
    and $\mhl$ stands for the Hardy-Littlewood maximal function, namely,
    \begin{align*}
        \mhl(g)(x):=\sup_{Q\ni x}\frac{1}{|Q|}\int_Q|g(y)|\,\mathrm{d}y,\quad x\in\rd,
    \end{align*}
    and the supremum is taken over all cubes which contain $x$.
\end{lemma}

\begin{proof}
    Note that $q/r>1$ and $p/r>1$, then \cite[Theorem 3.6(ii)]{hl23} yields Lemma~\ref{hlp0}.
\end{proof}


\begin{proof}[Proof of Theorem~\ref{fwd}.]
    Similar to the proof of Theorem~\ref{bwd}, we prove that the theorem follows from Theorem 5.1 in \cite{hst18}. 
    The space $\ftt(\rd)$ is an (isotropic, inhomogeneous) quasi-Banach {function}  space satisfying $\sd\eb\ftt(\rd)\eb\sdd$ (cf. \cite[Theorem 4.10]{hl23}) and the inequality $L>\sigma_{p,q}-s$ implies that $\ftt(\rd)$ can be characterised in terms of an $L$-atomic decomposition with $L=K$ and coefficients in $\ftts(\rd)$ (cf. Theorem~\ref{atde}). 
    Then it is sufficient to prove that the sequence space $\ftts=\ftts(\rd)$ is a $\varkappa$-sequence space for some $\varkappa$, $0<\varkappa<L$ (cf. Definition~\ref{vks}). 
    The proof is similar to the proof of \cite[Proposition 6.5]{hst18}. For the reader's convenience, we give some details.

    We divide \eqref{vke1} into two parts,
\begin{equation}\label{fwem}
\begin{split}
|\lzjm|&\lesssim \sum_{J=0}^{j}2^{-\varkappa(j-J)}\sum_{M\in\ijm}|\mujm|
+\sum_{J>j}2^{-(J-j)(\varkappa+n)}\sum_{M\in\ijm}|\mujm|\\&=:\Lambda_{j,m}+\Gamma_{j,m}.
\end{split}
\end{equation}

To estimate $\Lambda_{j,m}$, we assume $q<\infty$ and $\varepsilon>0$, then
$$
2^{jsq}\Lambda_{j,m}^q\kjm(x)\lesssim \sum_{J=0}^j2^{-(j-J)(\varkappa-s-\varepsilon)q}
\sum_{M\in\ijm}2^{Jsq}|\mujm|^q\kjm(x).
$$
Summation over $m\in\zd$ gives
\begin{align*}
\sum_{m\in\zd}2^{jsq}\Lambda_{j,m}^q\kjm(x)\lesssim \sum_{J=0}^j2^{-(j-J)(\varkappa-s-\varepsilon)q}\sum_{M\in\zd}2^{Jsq}|\mujm|^q\sum_{m\in\zd:M\in\ijm}\kjm(x).
\end{align*}
Analogously to the proof of Theorem 1.15 in \cite{tri08}, we argue that for fixed $x\in\rd$, $j\in\no\cap[j_P,\infty)$, $J\in\no$ and $M\in\zd$ the summation $\sum\kjm(x)$ over those $m\in\zd$ with $M\in\ijm$ is comparable with $\chi_{J,M}(x)$ which can be estimated from above by its maximal function. Hence we obtain for any $\varrho>0$,
\begin{align}\label{fwe1}
\sum_{m\in\zd}2^{jsq}\Lambda_{j,m}^q\kjm(x)\lesssim \sum_{J=0}^j2^{-(j-J)(\varkappa-s-\varepsilon)q}\sum_{M\in\zd}\mhl(2^{Js\varrho}|\mu_{J,M}|^\varrho\chi_{J,M}(\cdot))(x)^{\frac{q}{\varrho}}.
\end{align}
As for $\Gamma_{j,m}$, responsible for the terms with $J>j$, we have
\begin{equation*}
2^{js}\Gamma_{j,m}\kjm(x)\lesssim \kjm(x)\sum_{J>j}2^{-(J-j)(\varkappa+s+d)}\sum_{M\in\ijm}2^{Js}|\mujm|.
\end{equation*}
Assume $0<\varrho<1$ and $x\in\rd$ with $\kjm(x)=1$, then we can estimate the last sum by
\begin{align*}
\left(\sum_{M\in\ijm}2^{Js}|\mu_{J,M}|\right)^\varrho&\leq \sum_{M\in\ijm}2^{Js\varrho}|\mu_{J,M}|^\varrho\cdot
\frac{1}{|Q_{J,M}|}\int_{\rd}\chi_{J,M}(y) \dint y\\
&\lesssim 2^{Jd}2^{-jd}2^{jd}\int_{\rd}\sum_{M\in\ijm}2^{Js\varrho}|\mu_{J,M}|^\varrho\chi_{J,M}(y) \dint y\\
&\lesssim 2^{(J-j)d}\mhl\left(\sum_{M\in\ijm}2^{Js\varrho}|\mu_{J,M}|^\varrho\chi_{J,M}(\cdot)\right)(x).
\end{align*}
We insert { the last inequality into the last but one}. Assuming again $q<\infty$ one obtains for any fixed $\varepsilon>0$ that
\begin{align*}
2^{jsq}\Gamma_{j,m}^q\kjm(x)\lesssim \sum_{J>j}2^{-(J-j)(\varkappa+s+d-\frac{d}{\varrho}-\varepsilon)q}
\mhl\left(\sum_{M\in\ijm}2^{Js\varrho}|\mujm|^\varrho\chi_{J,M}(\cdot)\right)(x)^{\frac{q}{\varrho}}.
\end{align*}
First we consider the summation over $m\in\zd$. With $J=j+t$ we have
\begin{align}\label{fwe2}
\begin{split}
\sum_{M\in\zd}2^{jsq}\Gamma_{j,m}^q\kjm(x)&\lesssim\sum_{t=1}^{\infty}2^{-t(\varkappa+s+d-\frac{d}{\varrho}-\varepsilon)q}\\
&\,\,\,\,\,\,\times
\sum_{m\in\zd}\mhl
\left(\sum_{M\in I_{j+t}^{j}(m)}2^{(j+t)s\varrho}|\mu_{j+t,M}|^{\varrho}
\chi_{j+t,M}(\cdot)\right)(x)^{\frac{q}{\varrho}}.
\end{split}
\end{align}
Let us consider the sum
$$
\sum_{M\in I_{j+t}^{j}(m)}\left(2^{(j+t)s}|\mu_{j+t,M}|
\chi_{j+t,M}(x)\right)^\varrho
$$
for fixed $j\in\no\cap[j_P,\infty)$, $t\in\mathbb{N}$, $m\in\zd$ and $x\in\rd$. 
The cube $\qjm$ is divided in the cubes $Q_{j+t, M}$. Note that the ``small" cubes $Q_{j+t,M}$ are all disjoint. So we have
\begin{align*}
\sum_{M\in I_{j+t}^{j}(m)}\left(2^{(j+t)s}|\mu_{j+t,M}|
\chi_{j+t,M}(x)\right)^\varrho=\left(\sum_{M\in I_{j+t}^{j}(m)}2^{(j+t)s}|\mu_{j+t,M}|
\chi_{j+t,M}(x)\right)^\varrho=:g_{j,m}^t(x)^\varrho.
\end{align*}
Thus we get with \eqref{fwem}, \eqref{fwe1}, \eqref{fwe2} and summation over $j$ that
\begin{align*}
&\sum_{j=\jjp}^\infty\sum_{m\in\zd}2^{jsq}|\lzjm|^q\kjm(x)\\
&\lesssim \sum_{J=\jjp}^\infty\sum_{M\in\zd}\mhl\left(2^{Js\varrho}|\mujm|^\varrho\chi_{J,M}(\cdot)\right)(x)^{\frac{q}{\varrho}}
\sum_{j\geq J}2^{-(j-J)(\varkappa-s-\varepsilon)q}\\
&\,\,\,\,\,\,+\sum_{t=1}^{\infty}2^{-t(\varkappa+s+d-\frac{d}{\varrho}-\varepsilon)q}\sum_{j=\jjp}^{\infty}\sum_{m\in\zd}
\mhl\left(g_{j,m}^t(\cdot)^\varrho\right)(x)^{\frac{q}{\varrho}},
\end{align*}
with $0<\varepsilon<\varkappa-s$. Finally we have
\begin{align*}
\|\lz\mid f_{p,q}^{s,\varphi}(\rd)\|&\lesssim \left\|\left(\sum_{J=\jjp}^\infty\sum_{M\in\zd}\mhl
\left(2^{Js\varrho}|\mujm|^\varrho\chi_{J,M}\right)(\cdot)^{\frac{q}{\varrho}}\right)^{\frac{1}{q}}
\mid L_p^{\varphi}(\rd)\right\|\\
&\,\,\,\,\,\,+\left\|\left(\sum_{t=1}^\infty2^{-t(\varkappa+s+d-\frac{d}{\varrho}-\varepsilon)q}
\sum_{j=\jjp}^\infty\sum_{m\in\zd}\mhl\left((g_{j,m}^t)^\varrho\right)(\cdot)^{\frac{q}{\varrho}}\right)^{\frac{1}{q}}
\mid L_p^{\varphi}(\rd)\right\|.
\end{align*}
With an additional use of H\"older's inequality (for $0<q<1$) and an additional $\varepsilon$ we can take the sum over $t$ out of the $L_p^{\varphi}$-norm,
\begin{align*}
\|\lz\mid f_{p,q}^{s,\varphi}(\rd)\|&\lesssim \left\|\left(\sum_{J=\jjp}^\infty\sum_{M\in\zd}\mhl
\left(2^{Js\varrho}|\mujm|^\varrho\chi_{J,M}\right)(\cdot)^{\frac{q}{\varrho}}\right)^{\frac{1}{q}}
\mid L_p^{\varphi}(\rd)\right\|\\
&\,\,\,\,\,\,+\sum_{t=1}^\infty2^{-t(\varkappa+s+d-\frac{d}{\varrho}-2\varepsilon)}\left\|\left(
\sum_{j=\jjp}^\infty\sum_{m\in\zd}\mhl\left((g_{j,m}^t)^\varrho\right)(\cdot)^{\frac{q}{\varrho}}\right)^{\frac{1}{q}}
\mid L_p^{\varphi}(\rd)\right\|.
\end{align*}
Since $0<\varrho<\min(1,p,q)$ we can use Lemma~\ref{hlp0},
\begin{align*}
\|\lz\mid f_{p,q}^{s,\varphi}(\rd)\|&\lesssim \left\|\left(\sum_{J=\jjp}^\infty\sum_{M\in\zd}2^{Jsq}
|\mujm|^q\chi_{J,M}(\cdot)\right)^{\frac{1}{q}}\mid L_p^{\varphi}(\rd)\right\|+\sum_{t=1}^\infty
2^{-t(\varkappa+s+d-\frac{d}{\varrho}-2\varepsilon)}\\
&\,\,\,\,\,\,\times
\left\|\left(\sum_{j=\jjp}^\infty\sum_{m\in\zd}\left(\sum_{M\in I_{j+t}^j(m)}2^{(j+t)s}|\mu_{j+t,M}|\chi_{j+t,M}(\cdot)\right)^q\right)^{\frac{1}{q}}\mid L_p^{\varphi}(\rd)\right\|\\
&=\left\|\left(\sum_{J=\jjp}^\infty\sum_{M\in\zd}2^{Jsq}
|\mujm|^q\chi_{J,M}(\cdot)\right)^{\frac{1}{q}}\mid L_p^{\varphi}(\rd)\right\|+\sum_{t=1}^\infty
2^{-t(\varkappa+s+d-\frac{d}{\varrho}-2\varepsilon)}\\
&\,\,\,\,\,\,\times
\left\|\left(\sum_{j=\jjp}^\infty\sum_{m\in\zd}\sum_{M\in I_{j+t}^j(m)}2^{(j+t)sq}|\mu_{j+t,M}|^q\chi_{j+t,M}(\cdot)\right)^{\frac{1}{q}}\mid L_p^{\varphi}(\rd)\right\|\\
&\lesssim \|\mu\mid f_{p,q}^{s,\varphi}(\rd)\|\\
&\,\,\,\,\,\,+ \sum_{t=1}^\infty
2^{-t(\varkappa+s+d-\frac{d}{\varrho}-2\varepsilon)}
\left\|\left(\sum_{J=\jjp}^\infty\sum_{M\in\zd}2^{Jsq}|\mujm|^q\chi_{J,M}(\cdot)\right)^{\frac{1}{q}}
\mid L_p^{\varphi}(\rd)\right\|\\
&\lesssim \|\mu\mid f_{p,q}^{s,\varphi}(\rd)\|,
\end{align*}
where $0<2\varepsilon<\varkappa+s+d-\frac{d}{\varrho}$ and $0<\varrho<\min(1,p,q)$. Recall that we have
additionally $0<\varepsilon<\varkappa-s$. Then we choose
\begin{equation}\label{fwe3}
\varkappa>\max(\whole{1+s}_+,\sigma_{p,q}-s).
\end{equation}
Finally, it is easy to show (ii) in Definition~\ref{vks}. It holds $b_{p,q}^{s,\varphi}\eb f_{p,\infty}^{s,\varphi}$.
So we have here the same condition for $\varkappa$ as in the $b$-case,
\begin{equation}\label{fwe4}
\varkappa>\frac{3d}{p}-s.
\end{equation}
Then \eqref{fwe3} and \eqref{fwe4} lead to \eqref{lspq}. With some modifications and a similar proof we obtain the same result for $q=\infty$. Hence we finish the proof of Theorem~\ref{fwd}.
\end{proof}

\subsection{ Wavelet decomposition of $\MfF(\rd)$}

{
We first recall the corresponding sequence spaces for $\MfB(\rd)$  and $\MfF(\rd)$ .

\begin{definition}[\cite{nns16}] \label{def-seq-n-e}
Let $s\in\real$, $0<p<\infty$, $0<q\leq\infty$ and ${\varphi}\in\Gp$. 
    \begin{enumerate}[\bfseries\upshape  (i)]
        \item The sequence space $n^{s}_{\varphi,p,q}(\rd)$ is the set  of all doubly indexed sequences $\lz=\{\lzjm\}_{j\in\no,m\in\zd}$ for which
\begin{align*}
 \|\lambda\mid n^{s}_{\varphi,p,q}(\rd)\|:=\Biggl(\sum_{j=0}^{\infty}2^{jsq}   \bigg\| \sum_{m\in\zd} \lambda_{j,m} \chi_{Q_{j,m}} \,\Big|\, {\Mf}(\rd) \bigg\|^q \Biggr)^{\frac{1}{q}} 
\end{align*}
is finite (with the usual modification if $q=\infty$).        
\item Assume in addition that $\varphi$ satisfies \eqref{intc} when $0<q<\infty$. The sequence space $\mff(\rd)$ is defined to be the set of all doubly indexed sequences $\lz=\{\lzjm\}_{j\in\no,m\in\zd}$ for which
       \begin{align*}
    \|\lambda\mid\mff(\rd)\|:=\Bigg\|\Biggl(\sum_{j=0}^{\infty}2^{jsq}\sum_{m\in\zd}\left|\lzjm\kjm(\cdot)
\right|^q\Biggr)^{\frac{1}{q}} \Big| \Mf(\rd)\Bigg\|
\end{align*}
        is finite (with the usual modification for $q=\infty$).
        \item The sequence space ${a}_{\varphi,p,q}^{s}(\rd)$ denotes either ${n}_{\varphi,p,q}^{s}(\rd)$  or ${e}_{\varphi,p,q}^{s}(\rd)$, assuming in addition that $\varphi$ satisfies \eqref{intc} when $0<q<\infty$ and 
        ${a}_{\varphi,p,q}^{s}(\rd)$ denotes ${e}_{\varphi,p,q}^{s}(\rd)$.
\item The space ${\tilde{a}}_{\varphi,p,q}^{s}(\rd)$    collects all sequences 
\[
\lambda= \left\{ \lambda_m\in \mathbb{C}, \lambda_{j,m}^G \in \mathbb{C} : \; m\in \zd, \; j \in \mathbb{N}_0,\; G\in G^* \right\}
\]
such that
\begin{equation*}
\| \lambda\mid {\tilde{a}}_{\varphi,p,q}^{s}(\rd)\|:= \bigg\| \sum_{m\in\zd} \lambda_{m} \chi_{Q_{m}} \,\Big|\, \mathcal{M}_{\varphi, p}(\rd) \bigg\|  + \sum_{G\in G^*} \big\|
\{\lambda_{j,m}^G\}_{j\in\no, m\in\zd} \mid {a}_{\varphi,p,q}^{s}(\rd)\big\| 
\end{equation*} 
is finite.        
    \end{enumerate}
\end{definition}
}

Now we recall the wavelet decomposition of $\MfB(\rd)$ proved in \cite[Theorem 3.1]{hms22}.

\begin{theorem}\label{nwdp}
Let $0<p<\infty$, $0<q\leq\infty$, $s\in\real$ and $\varphi\in\Gp$. For the wavelets defined in \eqref{wave2} we take
\begin{equation*}
L>\max(\lfloor 1+s\rfloor_+,\,\frac{d}{p}-s).
\end{equation*}
Let $f\in\mathcal{S}'(\rd)$. Then $f\in\MfB(\rd)$ if, and only if, it can be represented as
\begin{align*}
    f=\sum_{m\in\zd}\lambda_m\psi_m+\sum_{G\in G^\ast}\sum_{j\in\no}\sum_{m\in\zd}\lzjm^{G}2^{-\frac{jd}{2}}\psi_{j,m}^G,
    \qquad\lambda\in {\tilde{n}}_{\varphi,p,q}^s(\rd),
\end{align*}
{convergence being} unconditional in $\mathcal{S}'(\rd)$. The representation is unique with
\begin{align*}
    \lzjm^G=\lzjm^G(f)=2^{\frac{jd}{2}}(f,\,\psi_{j,m}^G)\quad\text{and}\quad\lambda_m=\lambda_m(f)=(f,\,\psi_m),
\end{align*}
and
\begin{align*}
    I:f\mapsto\{\lambda_m{f},\,2^{\frac{jd}{2}}(f,\,\psi_{j,m}^G)\}
\end{align*}
is a linear isomorphism of $\MfB(\rd)$ onto ${\tilde{n}}_{\varphi,p,q}^s(\rd)$.

Furthermore, $\|I(f)\mid {\tilde{n}}_{\varphi,p,q}^s(\rd)\|$ may be used as an equivalent quasi-norm in $\MfB(\rd)$.
\end{theorem}

We recall {below}  the atomic decomposition for $\MfF(\rd)$, we refer to Definition~\ref{defi-atom} for the definition of atoms.

\begin{theorem}[{\cite[Theorem 4.5]{nns16}}]\label{mfad}
Let $0<p<\infty$, $0<q\leq\infty$, $s\in\mathbb{R}$ and $\varphi\in\Gp$.
Let also $c>1$, $L\in\{-1\}\cup\no$ and $K\in\no$. Assume
\begin{align*}
    K\geq\lfloor 1+s\rfloor_+,\qquad L\geq \max(-1, \lfloor\sigma_{p,q}-s\rfloor)
\end{align*}
and $\varphi$ satisfies \eqref{intc} when $q<\infty$.
\begin{enumerate}[\bfseries\upshape  (i)]
\item Let $\lambda=\{\lzjm\}_{j\in\no,m\in\zd}\in e_{\varphi,p,q}^s(\rd)$ and $\{a_{j,m}\}_{j\in\no,m\in\zd}$
be a family of $(K,L,c)$-atoms. Then
\begin{align*}
    f=\sum_{j\in\no}\sum_{m\in\zd}\lzjm a_{j,m}
\end{align*}
converges in $\mathcal{S}'(\rd)$. Furthermore,
\begin{align*}
    \|f\mid \mathcal{E}_{\varphi,p,q}^s(\rd)\|\lesssim\|\lambda\mid e_{\varphi,p,q}^s(\rd)\|.
\end{align*}
\item Let $f\in \mathcal{E}_{\varphi,p,q}^s(\rd)$. Then there exists a family $\{a_{j,m}\}_{j\in\no,m\in\zd}$ of $(K,L,c)$-atoms and a sequence $\lambda=\{\lzjm\}_{j\in\no,m\in\zd}\in e_{\varphi,p,q}^s(\rd)$ such that
\begin{align*}
    f=\sum_{j\in\no}\sum_{m\in\zd}\lzjm a_{j,m}\qquad\text{in}\qquad\mathcal{S}'(\rd)
\end{align*}
    and
    \begin{align*}
        \|\lambda\mid e_{\varphi,p,q}^s(\rd)\|\lesssim\|f\mid \mathcal{E}_{\varphi,p,q}^s(\rd)\|.
    \end{align*}
    
\end{enumerate}
\end{theorem}

{To show} the main theorem, we need the following vector-valued maximal inequality due to \cite{nns16}.

\begin{lemma}\label{mhlm}
Let $0<p<\infty$, $0<q\leq\infty$, $0<r<\min(p,q)$ and $\varphi\in\Gp$. Assume \eqref{intc} {when $q<\infty$}.
Then for any sequence $\{g_k\}_{k=0}^{\infty}$ of $\Mf(\rd)$-functions,
\begin{equation*}
\left\|\left(\sum_{k=0}^{\infty}\mhl(|g_k|^r)(\cdot)^{q/r}\right)^{\frac{1}{q}}\mid\Mf(\rd)
\right\|\lesssim\left\|\left(\sum_{k=0}^{\infty}|g_k|^q\right)^{\frac{1}{q}}\mid\Mf(\rd)\right\|,
\end{equation*}
with the usual modification made for $q=\infty$.
\end{lemma}

\begin{proof}
We show the proof for reader's convenience. Assume first $q<\infty$. By $1<\frac{q}{r}\leq\infty$, $1<\frac{p}{r}\leq\infty$
and Theorem 2.9 in \cite{nns16},
\begin{align*}
\left\|\left(\sum_{k=0}^{\infty}\mhl(|g_k|^r)(\cdot)^{q/r}\right)^{\frac{1}{q}}\mid\Mf(\rd)\right\|
&=\left\|\left(\sum_{k=0}^{\infty}\mhl(|g_k|^r)(\cdot)^{q/r}\right)^{\frac{r}{q}}\mid
\mathcal{M}_{\varphi,\frac{p}{r}}(\rd)\right\|^{\frac{1}{r}}\\
&\lesssim\left\|\left(\sum_{k=0}^{\infty}(|g_k|^q)\right)^{\frac{r}{q}}\mid
\mathcal{M}_{\varphi,\frac{p}{r}}(\rd)\right\|^{\frac{1}{r}}\\
&=\left\|\left(\sum_{k=0}^{\infty}|g_k|^q\right)^{\frac{1}{q}}\mid\Mf(\rd)\right\|.
\end{align*}
The modification of the argument in case of $q=\infty$ is obvious and left to the reader. Thus we finish the proof.
\end{proof}


\begin{theorem}\label{mew}
Let $0<p<\infty$, $0<q\leq\infty$, $s\in\real$ and $\varphi\in\Gp$. Assume in addition that $\varphi$ satisfies \eqref{intc} when $q<\infty$.
Let
\begin{equation*}
L>\max(\lfloor 1+s\rfloor_+,\,\frac{d}{p}-s,\,\sigma_{p,q}-s).
\end{equation*}
Let $f\in\mathcal{S}'(\rd)$. Then $f\in\MfF(\rd)$ if, and only if, it can be represented as
\begin{align*}
    f=\sum_{m\in\zd}\lambda_m\psi_m+\sum_{G\in G^\ast}\sum_{j\in\no}\sum_{m\in\zd}\lzjm^{G}2^{-\frac{jd}{2}}\psi_{j,m}^G,
    \qquad\lambda\in \tilde{e}_{\varphi,p,q}^s(\rd),
\end{align*}
{convergence being} unconditional in $\mathcal{S}'(\rd)$. The representation is unique with
\begin{align*}
    \lzjm^G=\lzjm^G(f)=2^{\frac{jd}{2}}(f,\,\psi_{j,m}^G)\qquad\text{and}\qquad\lambda_m=\lambda_m(f)=(f,\,\psi_m),
\end{align*}
and
\begin{align*}
    I:f\mapsto\{\lambda_m{f},\,2^{\frac{jd}{2}}(f,\,\psi_{j,m}^G)\}
\end{align*}
is a linear isomorphism of $\MfF(\rd)$ onto ${\tilde{e}}_{\varphi,p,q}^s(\rd)$.

Furthermore, $\|I(f)\mid {\tilde{e}}_{\varphi,p,q}^s(\rd)\|$ may be used as an equivalent quasi-norm in $\MfF(\rd)$.
\end{theorem}

\begin{proof}
The proof follows {once more} from \cite[Theorem 5.1]{hst18}. { If $\varphi$ satisfies \eqref{intc} when $q<\infty$, then the space $\MfF(\rd)$ is an isotropic, inhomogeneous, quasi-Banach function space, in particular  
$$
\mathcal{S}(\rd)\eb\MfF(\rd)\eb\mathcal{S}'(\rd),
$$
 see \cite{nns16} for more details.} Note that the inequality $L>\sigma_{p,q}-s$ implies {that} $\MfF(\rd)$ can be characterised
in terms of an $L$-atomic decomposition with $L=K$ and coefficients in $e^s_{\varphi,p,q}(\rd)$.
We now only need to show that the sequence space $e^s_{\varphi,p,q}(\rd)$ is a $\varkappa$-sequence space for some
$\varkappa$, $0<\varkappa<L$. To this end, we just make some simple modifications to the proof of \cite[Proposition 6.5]{hst18} and apply Lemma~\ref{mhlm}. Hence we omit details and finish the proof.
\end{proof}

{
\begin{remark}
	One should also mention that Theorem \ref{mew} follows  from Theorem \ref{fwd} and  Theorem \ref{rswm} that is recalled below. It can be easily checked that the sequence spaces $\tilde{e}_{\varphi,p,q}^s(\rd)$ and $\tilde{f}_{p,q}^{s,\varphi} (\rd)$ are isomorphic if $\varphi$ satisfies the condition \eqref{intc} in the case $q<\infty$.    
	\end{remark}
}

\section[Different scales of generalised Morrey smoothness spaces]{Relations between the different scales of generalised Morrey smoothness spaces}\label{rel-scales}

  In the previous sections we introduced the spaces $\Att(\rd)$ and $\MfA(\rd)$ as generalisations of the well-known scales $\at(\rd)$ and $\MA(\rd)$, respectively, recall Remarks~\ref{rmk-tau-spaces} and \ref{equnv}. For those latter spaces it is well-known that for $s\in\real$, $0<p\leq u<\infty$, $0<q\leq\infty$, $\tau\geq 0$,
	\begin{equation}
	\MB(\rd) \hookrightarrow  \bt(\rd) \qquad \text{with}\qquad \tau={1}/{p}- {1}/{u}, 
	\label{N-BT-emb}
	\end{equation}
    cf. \cite[Corollary~3.3]{ysy10}, and the above embedding is proper, if $\tau>0$ and $q<\infty$, while for $\tau=0$ or $q=\infty$, both spaces coincide,
	\begin{equation}
	\mathcal{N}^{s}_{u,p,\infty}(\rd)  =  B^{s,\frac{1}{p}-\frac{1}{u}}_{p,\infty}(\rd).
	\label{N-BT-equal}
	\end{equation}
	As for the $F$-spaces, if $0\le \tau <{1}/{p}$,
	then
	\begin{equation}\label{fte}
	\ft(\rd)\, = \, \MF(\rd)\quad\text{with }\quad \tau =
	{1}/{p}-{1}/{u}\, ,\quad 0 < p\le u < \infty, 
	\end{equation}
	cf. \cite[Corollary~3.3]{ysy10}.

Now we would like to study the counterparts of the above results in the more general setting of $\Att(\rd)$ and $\MfA(\rd)$.

\begin{theorem}\label{rswm}
    Let $s\in\real$, {$0<p<\infty$, $0<q\leq\infty$, and $\varphi\in\Gp$.}
    \begin{enumerate}[\bfseries\upshape  (i)]
        \item Assume that $\varphi$ satisfies \eqref{intc} when $q<\infty$. Then 
        $F_{p,q}^{s, {\varphi}}(\rd)=\MfF(\rd)$ with equivalent norms.
        \item It holds
        $$ \MfB(\rd) \hookrightarrow B_{p,q}^{s, \varphi}(\rd). $$
        When $q=\infty$, both spaces coincide with each other, namely,
        $$
        \mathcal{N}^s_{\varphi, p, \infty}(\rd) = B_{p,\infty}^{s, \varphi}(\rd),
        $$
        and when $q<\infty$ and 
        \begin{equation}\label{ls1}
        \lim_{t\rightarrow 0^+} \varphi(t)t^{-\frac{d}{p}}=\infty \quad \text{or} \quad  \lim_{t\rightarrow +\infty} \varphi(t)t^{-\frac{d}{p}}=0,
        \end{equation}
        then $\MfB(\rd)$ is a proper subspace of $B_{p,q}^{s, \varphi}(\rd)$.
    \end{enumerate}
\end{theorem}

\begin{proof} 
{(i)} has been obtained in \cite[Proposition A.2]{nns16}. We now prove {(ii)} {by} borrowing some ideas from the proof of \cite[Theorem 1.1(i)]{syy10}.  
{Due to the wavelet decomposition, cf. Theorems~\ref{bwd} and \ref{nwdp}, it is enough to prove it for the related} sequence spaces ${\tilde{b}}_{p,q}^{s, \varphi}(\rd)$ and ${\tilde{n}}_{\varphi,p,q}^{s}(\rd)$, which is still equivalent to prove  it for the sequence spaces ${b}_{p,q}^{s, \varphi}(\rd)$ and ${n}_{\varphi,p,q}^{s}(\rd)$. Recall the definition of the latter spaces given in Definitions~\ref{def-seq-b-f} and \ref{def-seq-n-e}.

The embedding $n_{\varphi,p,q}^{s}(\rd)\hookrightarrow b_{p,q}^{s,\varphi}(\rd)$ and the identity $n_{\varphi,p,\infty}^{s}(\rd) =b_{p,\infty}^{s,\varphi}(\rd)$ can be easily checked.
 
Now let $q<\infty$ and assume first that $\lim_{t\rightarrow 0^+} \varphi(t)t^{-\frac{d}{p}}=\infty$. Then we can consider a strictly increasing sequence of natural numbers $(k_{\nu})_{\nu\in \nat_0}$ so that $k_0=1$ and 
\begin{equation} \label{cond-knu-1}
\varphi(2^{-k_{\nu+1}})2^{k_{\nu+1}\frac{d}{p}}> 2\, \varphi(2^{-k_{\nu}})2^{k_{\nu}\frac{d}{p}} , \quad \nu\in\nat_0.
\end{equation}
Let $R_j:=\underbrace{ [2^{-j},2^{1-j})\times\cdots\times[2^{-j},2^{1-j}) }_{d\ \text{times}}$ for $j\in\no$ and set
 \begin{align}\label{example1}
    \lzjm:=
        \begin{cases}
          2^{-j s}\varphi(2^{-j})^{-1},&\quad \qjm=R_j\text{  for}\;  j=k_{\nu} \text{  for some } \nu \in\no,  \\
          0,&\text{ \,\, otherwise.}
        \end{cases}
    \end{align}   
In the calculation of $ \|\lz\mid b_{p,q}^{s,\varphi}(\rd)\|$,  the inner summation is not zero only for cubes $P$  in one of the following cases: (a) $P=R_{k_{\nu}}$ for some $\nu\in\nat_0$; (b) $P$ contains all $R_{k_{\nu}}$  for $\nu\geq \nu_0$ and some $\nu_0\in\nat_0$. The case (a) can be covered by case (b). Under hypothesis (b), we have
 \begin{align*}
& \frac{\varphi(\ell(P))}{|P|^\frac{1}{p}}
        \biggl\{\sum_{j=\jjp}^\infty   2^{(s-\frac{d}{p})jq}  \biggl(  \sum_{\substack{m\in\zd\\ \qjm\subset P}}  |\lzjm|^p   \biggr)^{\frac{q}{p}}\biggr\}^\frac{1}{q} \\
&\qquad   \leq  \frac{\varphi(2^{-k_{\nu_0}+1})}{(2^{-k_{\nu_0}+1})^{\frac{d}{p}}}  \left\{\sum_{\nu=\nu_0}^{\infty} 2^{k_{\nu}(s-\frac{d}{p})q} \bigl(  2^{-k_{\nu} s}\varphi(2^{-k_{\nu}})^{-1} \bigr)^q\right\}^{\frac{1}{q}} \\
& \qquad  \leq  \frac{\varphi(2^{-k_{\nu_0}})}{2^{-\frac{d}{p} k_{\nu_0}}}  \left\{\sum_{\nu=\nu_0}^{\infty}   2^{-(\nu-\nu_0)q} 2^{-\frac{d}{p} k_{\nu_0}q}   \varphi(2^{-k_{\nu_0}})^{-q} \right\}^{\frac{1}{q}} \leq C.
\end{align*}

Here we used in the first inequality the assumption on $P$ in (b) which leads to $\ell(P)\geq 2\times 2^{-k_{\nu_0}}$ together with  \eqref{Gp-def}, while in the second inequality we applied 
again    \eqref{Gp-def} and \eqref{cond-knu-1}. 
Therefore,  $\lz \in b_{p,q}^{s,{\varphi}}(\rd)$. On the other hand,
     \begin{align*}
        \|\lz\mid n_{\varphi,p,q}^{s}(\rd)\|& =
\Biggl(\sum_{\nu=0}^\infty  2^{sk_{\nu}q}
        \sup_{l\leq k_{\nu}} \varphi( 2^{-l})^q 2^{(l-k_{\nu})\frac{d}{p}q}   2^{-s k_{\nu} q}  \varphi( 2^{-k_\nu})^{-q}   \Biggr)^\frac{1}{q}=\infty.
\end{align*}

	Assume now that $\lim_{t\rightarrow +\infty} \varphi(t)t^{-\frac{\nd}{p}}=0$ and $q<\infty$. Let $(k_{\ell})_{\ell\in \nat}$ be a  strictly increasing sequence of natural numbers and 	
	 consider cubes  $P_\ell= Q_{0, m_{k_\ell}}$, $m_{k_\ell}=(2^{k_\ell},\ldots, 0)$,  if $\ell\in\nat$ and $m_{{k_0}}=(0,\ldots,0)$. Moreover we set
	\begin{align}\label{example2}
		\lzjm:=
		\begin{cases}
			2^{-j s}, & \quad \text{if\,\,} \;\qjm  \subset P_j\\  
			0,& \quad\text{otherwise.}
		\end{cases}
	\end{align}
	Then, for a dyadic cube $P$, it holds 
	\begin{align}
		\biggl(  \sum_{\substack{m\in\zd\\ \qjm\subset P}}  |\lzjm|^p   \biggr)^{\frac{1}{p}} \sim  
		\begin{cases}\label{example2a}
			0, & \qquad\text{if}\quad P\cap P_\ell=\emptyset\,\,\text{for any}\,\, \ell,\\
			2^{-j(s-\frac{d}{p})} |P|^{1/p}, & \qquad\text{if}\quad 2^{-j}\le \ell(P)\,\,\text{and}\,\, P\subset P_j\; , \\ 
			2^{-j(s-\frac{d}{p})},& \qquad\text{if} \quad P_j\subset P.  \end{cases}
	\end{align}
          Now we calculate $\| \lambda \mid n^s_{\varphi,p,q}(\rd)\|$. If $0\leq l\leq j$, we may choose $P=Q_{l,k}$ and thus obtain by the second line of \eqref{example2a} that
\begin{align*}
		\sup_{\substack{0\leq l\leq j \\ k\in\zd}} \varphi( 2^{-l}) 2^{(l-j)\frac{d}{p}}
		\Biggl(\sum_{\substack{m\in\zd\\\qjm\subset Q_{l, k}}}
  |\lzjm|^p\Biggr)^\frac{1}{p}  \ge & \sup_{\substack{0\leq l\leq j \\ k\in\zd}} \varphi( 2^{-l}) 2^{(l-j)\frac{d}{p}} 2^{-j(s-\frac{d}{p})} 2^{-l\frac{d}{p}} \\
   = & \sup_{0\leq l\leq j} \varphi( 2^{-l}) 2^{-js} = 2^{-js},
	\end{align*}
using $\varphi(1)=1$, while in case of $l\leq 0$ and $P=Q_{l,k}$ we apply the last line of      \eqref{example2a} to conclude that       
\begin{align*}
		\sup_{\substack{l\leq 0 \\ k\in\zd}} \varphi( 2^{-l}) 2^{(l-j)\frac{d}{p}}
		\Biggl(\sum_{\substack{m\in\zd\\\qjm\subset Q_{l, k}}}
  |\lzjm|^p\Biggr)^\frac{1}{p}  \ge \ & \ \sup_{l\leq 0} \ \varphi( 2^{-l}) 2^{(l-j)\frac{d}{p}} 2^{-j(s-\frac{d}{p})}  \\
   = & \ 2^{-js} \sup_{l\leq 0} \varphi( 2^{-l}) 2^{l\frac{d}{p}} = 2^{-js},
	\end{align*}
        by the assumption on $\varphi$. Hence, together we arrive at 
       \begin{align*}
		\sup_{\substack{l\leq j \\ k\in\zd}} \varphi( 2^{-l}) 2^{(l-j)\frac{d}{p}}
		\Biggl(\sum_{\substack{m\in\zd\\\qjm\subset Q_{l, k}}}
		|\lzjm|^p\Biggr)^\frac{1}{p}  \ge 2^{-js}, 
	\end{align*}
which results in $\|\lz\mid n_{\varphi,p,q}^{s}(\rd)\|=\infty$ if $q<\infty$.
	
Now assume further that the sequence $(k_{\ell})_{\ell\in \nat}$ satisfies
	\begin{equation}\label{example_2_0}
		\varphi(2^{k_{\ell}})2^{-k_{\ell}\frac{d}{p}}< \ell^{-\frac{1}{q}} , \quad \ell \in\nat,
	\end{equation}
this is possible since $\displaystyle{\lim_{t\rightarrow +\infty} \varphi(t)t^{-\frac{\nd}{p}}=0}$. 
Then we have $\|\lz\mid b^{s,\varphi}_{p,q}(\rd)\|<\infty$,  which can be seen as follows: If $\ell(P)\le 1$, then the sum $\sum_{j=\jjp}^\infty \dots$ consists of at most one element {and, by the second line of \eqref{example2a}, we have}
	\begin{align*}
		& \frac{\varphi(\ell(P))}{|P|^\frac{1}{p}}
		\biggl\{\sum_{j=\jjp}^\infty   2^{(s-\frac{d}{p})jq}  \biggl(  \sum_{\substack{m\in\zd\\ \qjm\subset P}}  |\lzjm|^p   \biggr)^{\frac{q}{p}}\biggr\}^\frac{1}{q}=  \frac{\varphi(\ell(P))}{|P|^\frac{1}{p}} |P|^{1/p}\le \varphi(\ell(P)) \le 1.    
	\end{align*}	
	On the other hand, if {$\ell(P) > 1$ and}   the sum $\sum_{j=\jjp}^\infty \dots$ consists of $\nu$ summands, $\nu>1$, then the cube $P$ contains $\nu$ cubes $Q_{0,k_\ell}$ and  $\ell(P)> 2^{k_\nu }$. Moreover, 
	\begin{align*}
		& \frac{\varphi(\ell(P))}{|P|^\frac{1}{p}}
		\biggl\{\sum_{j=\jjp}^\infty   2^{(s-\frac{d}{p})jq}  \biggl(  \sum_{\substack{m\in\zd\\ \qjm\subset P}}  |\lzjm|^p   \biggr)^{\frac{q}{p}}\biggr\}^\frac{1}{q} 
		\leq \frac{\varphi(\ell(P))}{|P|^\frac{1}{p}} \nu^{1/q} 
		\le \varphi(2^{k_\nu})2^{-k_\nu\frac{d}{p}} \nu^{1/q} \le 1,
	\end{align*} 
where we {used the fact that $\varphi \in \Gp$ and} benefited from our assumption \eqref{example_2_0}.       
\end{proof}

\begin{remark}\label{Rem-N=B=phi}
  In the particular case when $\varphi(t):=t^{\frac{d}{u}}$, $0<p\leq u$, we recover the results obtained in \cite[Theorem~1.1]{syy10}, as recalled in \eqref{N-BT-emb}, \eqref{N-BT-equal} and \eqref{fte}.
  If the function $\varphi$ does not satisfy the condition \eqref{ls1}, then
\begin{align*}
  0< \lim_{t\rightarrow 0^+} \varphi(t)t^{-\frac{d}{p}}<\infty \qquad\text{and}\qquad   0<\lim_{t\rightarrow +\infty} \varphi(t)t^{-\frac{d}{p}}<\infty,
\end{align*}
since the function $\varphi(t)t^{-\frac{d}{p}}$ is decreasing. But this means that both spaces $\MfB(\rd)$ and  $B_{p,q}^{s, \varphi}(\rd)$ coincide with a classical Besov space, 
\begin{align*}
  \MfB(\rd) = B_{p,q}^{s, \varphi}(\rd) = \B(\rd), \qquad \varphi(t) \sim t^{\frac{d}{p}}, \ t>0.
\end{align*}
So the relation between $\MfA(\rd)$ and $\Att(\rd)$ is completely described by Theorem~\ref{rswm} and the above observation.

Moreover, the condition \eqref{ls1} now confirms that both, the behaviour of $\varphi$ near $0$ (locally) and near $\infty$ (globally) influence the embedding behaviour of $\MfB(\rd) \hookrightarrow \btt(\rd)$, we shall observe this phenomenon later again. In the special setting $\varphi(t)=t^{\frac{d}{u}}$, $t>0$, this was invisible inasmuch as local and global behaviour of $\varphi$ was characterised by the same parameter $u\geq p$. Checking our Examples~\ref{ex-phi} we find that \eqref{ls1} is satisfied for $\varphi_{u,v}$ given by Example~\ref{ex-phi}(ii) whenever $\min(u,v)=p<\max(u,v)\leq\infty$. Example~\ref{ex-phi}(iii) does not satisfy the first (local) part of \eqref{ls1}, but the second (global) one when $a<0$.
\end{remark}

Next we study the counterparts of the classical (sharp) embeddings \eqref{se-bf} in the setting of Morrey smoothness spaces. While in case of spaces $\at(\rd)$ this extends directly to
  \begin{equation}\label{elem-tau}
B^{s,\tau}_{p,\min(p,q)}(\rd)\, \hookrightarrow \, \ft(\rd)\, \hookrightarrow \, B^{s,\tau}_{p,\max(p,q)}(\rd),
\end{equation}
cf. \cite[Proposition~2.1]{ysy10}, we already discussed the different behaviour in the scale $\MA(\rd)$ in \eqref{se-au} in Remark~\ref{equnv}. In \cite[Proposition~4.9]{hl23} we observed that the generalised version of \eqref{elem-tau} remains true,
  \begin{equation}\label{elem-gen}
B^{s,\varphi}_{p,\min(p,q)}(\rd)\, \hookrightarrow \, \ftt(\rd)\, \hookrightarrow \, B^{s,\varphi}_{p,\max(p,q)}(\rd).
\end{equation}
Now we concentrate on the counterpart of \eqref{se-au}.

\begin{theorem}  \label{Nphi-Ephi}
    Let $0<p<\infty$, $0<q,q_0\leq \infty$, $s\in\real$ and $\varphi\in\Gp$. Assume in addition that $\varphi$ satisfies \eqref{intc} for $\MfF(\rd)$ when $q<\infty$. 
    \begin{enumerate}[\bfseries\upshape  (i)]
    \item Then 
    \begin{align}\label{Nphi-Ephi'}
    \mathcal{N}_{\varphi,p,\min(p,q)}^s(\rd)\eb\MfF(\rd).
    \end{align}
    \item Assume that
    \begin{align}\label{een-1'}
      \lim_{t\rightarrow 0^+} t^{-\frac{d}{p}}\varphi(t)<\infty\quad\text{and}\quad
      \lim_{t\rightarrow +\infty}t^{-\frac{d}{p}}\varphi(t)>0.
      \end{align}
      Then
      \begin{align}\label{een''}
    \MfF(\rd)\eb\mathcal{N}_{\varphi,p,q_0}^s(\rd)
\end{align}
if, and only if, $q_0\geq  \max(p,q)$.
\item
  Otherwise, if \eqref{een-1'} is not satisfied, then \eqref{een''} holds if, and only if, $q_0=\infty$.
    \end{enumerate}
\end{theorem}

\begin{proof} (i) is deduced by Minkowski's inequality and the monotonicity of $\ell_q$ spaces. As for (ii), it is enough to prove the assertions for the corresponding sequence spaces due to the wavelet characterisation. 
Parallel to Remark~\ref{Rem-N=B=phi}, \eqref{een-1'} implies that $\MfF(\rd)=\F(\rd)$ and $\MfB(\rd)=\B(\rd)$ such that the result coincides with the right-hand side of \eqref{se-bf}, the well-known classical result, {cf. \cite[Theorem~3.1.1]{SiT} for the necessity of $q_0\geq \max(p,q)$ in this case}. \smallskip

Now we deal with (iii). Since  ${e}^s_{\varphi,p,q}(\rd) \eb {n}^s_{\varphi,p,\infty}(\rd)$, the sufficiency is clear and we are left to prove the necessity. %

  Assume first that $\lim_{t\rightarrow 0^+} t^{-\frac{d}{p}}\varphi(t)=\infty$ and consider the sequence  $\lambda=\{\lzjm\}_{j\in\nat_0,m\in\zd}$ defined in \eqref{example1}. As we have seen in the proof of Theorem~\ref{rswm}, $\lambda \not\in {n}^s_{\varphi,p,q_0}(\rd)$ for any $q_0<\infty$. On the other hand, we have
\begin{align*}
   \|\lambda\mid {e}^s_{\varphi,p,q}(\rd)\| &=\bigg\|\bigg(\sum_{\nu=0}^{\infty}\varphi(2^{-k_{\nu}})^{-q} \chi_{R_{k_{\nu}}}(\cdot)\bigg)^{\frac{1}{q}}\mid \Mf(\rd)\bigg\|\\
   & = \sup_{P\in\mq}\frac{\varphi(\ell(P))}{|P|^\frac{1}{p}} \left(\int_P  \bigg(\sum_{\nu=0}^{\infty}\varphi(2^{-k_{\nu}})^{-q} \chi_{R_{k_{\nu}}}(x)\bigg)^{\frac{p}{q}} \dint x \right)^{\frac{1}{p}} \\
    & = \sup_{P\in\mq}\frac{\varphi(\ell(P))}{|P|^\frac{1}{p}} \left(\sum_{\nu=0}^{\infty}\varphi(2^{-k_{\nu}})^{-p} |P\cap R_{k_{\nu}}| \right)^{\frac{1}{p}}. 
\end{align*}
The inner summation vanishes except for cubes $P$  in one of the following three cases: (a) $P\subsetneq R_{k_{\nu}}$ for some $\nu\in\nat_0$; (b) $P=R_{k_{\nu}}$ for some $\nu\in\nat_0$; (c) $P$ contains all $R_{k_{\nu}}$  for $\nu\geq \nu_0$ and some $\nu_0\in\nat_0$.  The case (a) does not give the supremum and case (b) can be incorporated in case (c). In this latter case we have
\begin{align*}
 &\frac{\varphi(\ell(P))}{|P|^\frac{1}{p}} \left(\sum_{\nu=0}^{\infty}\varphi(2^{-k_{\nu}})^{-p} |P\cap R_{k_{\nu}}| \right)^{\frac{1}{p}}  \leq  \frac{\varphi(\ell(P))}{|P|^\frac{1}{p}} \left(\sum_{\nu=\nu_0}^{\infty}\varphi(2^{-k_{\nu}})^{-p}  2^{-k_{\nu}d} \right)^{\frac{1}{p}} \\
&\qquad \leq  \frac{\varphi(2^{-k_{\nu_0}})}{ 2^{-\frac{d}{p}k_{\nu_0}}} \left(\sum_{\nu=\nu_0}^{\infty}\varphi(2^{-k_{\nu}})^{-p}  2^{-dk_{\nu}} \right)^{\frac{1}{p}} 
 \leq   \left(\sum_{\nu=\nu_0}^{\infty}2^{-(\nu-\nu_0)p} \right)^{\frac{1}{p}} \leq C.
\end{align*}
Therefore $\lambda \in {e}^s_{\varphi,p,q}(\rd)$, which disproves (the sequence spaces version of) \eqref{een''} whenever $q_0<\infty$.

Assume now that $\lim_{t\rightarrow +\infty} t^{-\frac{d}{p}}\varphi(t)=0$. We show that the sequence {defined in} \eqref{example2},
{where now the strictly increasing sequence of natural numbers $(k_{\ell})_{\ell\in \nat}$   satisfies
\begin{equation}\label{example_2_0'}
\varphi(2^{k_{\ell}})2^{-k_{\ell}\frac{d}{p}}< \ell^{-\frac{1}{p}} , \quad \ell \in\nat,
\end{equation}}%
belongs to $ {e}^s_{\varphi,p,q}(\rd)$. We already know that it does not belong to $ {n}^s_{\varphi,p,q_0}(\rd)$ if $q_0<\infty$. So this will prove the necessity of the condition.  
	
	Let $\ell(P)\le 1$. If  $P$ is not contained in any $P_j$ cube, then the integral over $P$ used in the definition of the norm of $ {e}^s_{\varphi,p,q}(\rd)$  is zero. 
	If $P\subset P_{j_0}$ for some $j_0$, then the summation over $j$ contains only one summand and   
	\begin{align*}
		\frac{\varphi(\ell(P))}{|P|^\frac{1}{p}} 
		\left(\int_P  \bigg(\sum_{j=0}^{\infty}2^{jsq} \sum_{m\in\zd}
	|\lzjm|^q\kjm(x)\bigg)^{\frac{p}{q}} \dint x \right)^{\frac{1}{p}} = \varphi(\ell(P)) \le 1. 
\end{align*}

Let $\ell(P)>1$. If $P$ contains $\nu$ cubes $P_j$,  then $\ell(P)>2^{k_\nu}$. Assume $P_{j_1},\ldots , P_{j_\nu}$ are these cubes. We have     
\begin{align*}
	&\frac{\varphi(\ell(P))}{|P|^\frac{1}{p}} 
	\left(\int_P  \bigg(\sum_{j=0}^{\infty}2^{jsq} \sum_{m\in\zd}
|\lzjm|^q\kjm(x)\bigg)^{\frac{p}{q}} \dint x \right)^{\frac{1}{p}} \\ 
&\qquad\leq\frac{\varphi(\ell(P))}{|P|^\frac{1}{p}}  \left(\int_P  \bigg(\sum_{k=1}^{\nu} \chi_{P_{j_k}}(x)\bigg)^{\frac{p}{q}} \dint x \right)^{\frac{1}{p}}
\leq  	\frac{\varphi(2^{k_\nu})}{2^{k_\nu\frac{d}{p}} }\nu^{{1/p}}  \leq 1,
\end{align*}
{where we used the fact that $\varphi \in \Gp$ and \eqref{example_2_0'}}.  This implies $\lambda \in {e}^s_{\varphi,p,q}(\rd)$ and finally completes the proof.
\end{proof}


\section{Embeddings between generalised Morrey smoothness spaces}\label{sect-emb}

Now we study embeddings within the different scales of smoothness spaces, where only the case $\MfB(\rd)$ was investigated in \cite{hms22,hms23} already. We briefly recall it as it will be used later for our argument, but will also  serve for comparison.

\begin{theorem}[{\cite[Theorem 5.1]{hms22}}]\label{main-n}
Let $s_i\in\real$, $0<p_i<\infty$, $0<q_i<\infty$ and $\varphi_i\in\mathcal{G}_{p_i}$, $i=1,2$. 
We assume without loss of generality that $\varphi_1(1)=\varphi_2(1)=1$. Let $\varrho=\min(1,\frac{p_1}{p_2})$ and $\alpha_j=\sup\limits_{\nu\leq j}\frac{\varphi_2(2^{-\nu})}{\varphi_1(2^{-\nu})^\varrho}$, $j\in\no$.

There is a continuous embedding
\begin{align}\label{emfb1}
    \MfBa(\rd) \hookrightarrow \MfBb(\rd)
\end{align}
if, and only if,
\begin{align}\label{emfb2}
    \sup_{\nu\leq 0}\frac{\varphi_2(2^{-\nu})}{\varphi_1(2^{-\nu})^\varrho}<\infty,
\end{align}
and
\begin{align}\label{emfb3}
    \left\{2^{j(s_2-s_1)}\alpha_j\frac{\varphi_1(2^{-j})^\varrho}{\varphi_1(2^{-j})}\right\}_j\in \ell_{q^\ast},\qquad\text{where}\quad\frac{1}{q^\ast}=\left(\frac{1}{q_2}-\frac{1}{q_1}\right)_+.
\end{align}

The embedding \eqref{emfb1} is never compact.
\end{theorem}

\subsection{ Embeddings within the scale $\MfF(\rd)$ }

\subsubsection{Some preparation}

As we want to study embeddings within the scale $\MfF(\rd)$, we also recall the special case $\MF(\rd)$ for convenience.

\begin{theorem}[{\cite[Theorem 3.1]{hs14}}]\label{at2}
    Let $s_i\in\real$, $0<q_i\leq\infty$, $0<p_i<u_i<\infty$, $i=1,2$. Let $\varrho=\min(1,\frac{p_1}{p_2})$. 
    
    There is a continuous embedding
    \begin{align}\label{eeu1}
        \MFa(\rd)\eb\MFb(\rd)
    \end{align}
    if, and only if,
    \begin{align}\label{eeu2}
        \frac{u_1}{u_2}\leq\varrho
    \end{align}
    and 
    \begin{align}
        &s_1-\frac{d}{u_1}>s_2-\frac{d}{u_2},\label{eeu3}\\
         or \qquad &s_1-\frac{d}{u_1}=s_2-\frac{d}{u_2}\quad and \quad u_1<u_2,\label{eeu4}\\
         or \qquad &s_1=s_2,\quad u_1=u_2\quad and \quad q_1\leq q_2.\label{eeu5}
    \end{align}
   
    The embedding \eqref{eeu1} is never compact.
\end{theorem}



Before we present our main results, we  recall two embeddings. From Theorem~\ref{Nphi-Ephi}, we conclude that 
\begin{align}\label{elme1}
    \mathcal{N}_{\varphi,p,\min(p,q)}^s(\rd)\eb\MfF(\rd)\eb\mathcal{N}_{\varphi,p,\infty}^s(\rd).
\end{align}
{ Moreover, it is shown in  \cite[Corollary 30]{sdh20b} that, if $0<p_i<\infty$, $\varphi_i\in\mathcal{G}_{p_i}$, $i=1,2$ and  
 {$\inf_{t>0}\varphi_1(t)=0$}, then}
\begin{align}\label{elme2}
    \mathcal{M}_{\varphi_1,p_1}(\rd)\eb\mathcal{M}_{\varphi_2,p_2}(\rd)
    \qquad\text{if, and only if,}\qquad p_1\geq p_2\quad\text{and}\quad \varphi_1\gtrsim \varphi_2.
\end{align}

\subsubsection{ The case $p_1\geq p_2$}
We begin with the case $p_1\geq p_2$ which seems the easier one, see also Theorem~\ref{main-n}, where in this setting $\varrho=1$.

 \begin{theorem}\label{E-E-case1}
    Let $0<p_2\leq p_1<\infty$, $s_i\in\real$, $0<q_i\leq\infty${, $\varphi_i\in \mathcal{G}_{p_i}$ and $\varphi_i$ satisfies \eqref{intc} when $q_i<\infty$, for $i=1,2$}. We assume without loss of generality that $\varphi_1(1)=\varphi_2(1)=1$. 
    Then
    	\begin{align}\label{emff_1}
			\MfFa(\rd)\eb\MfFb(\rd)
		\end{align}
		if, and only if,
			\begin{align}
			\sup_{\nu\in\no}\frac{\varphi_2(2^{\nu})}{\varphi_1(2^{\nu})}<\infty, \label{embee0_1} 
		    \end{align}
and	    

		\begin{align}\label{embee00_1}
                  \left\{2^{j(s_2-s_1)}\frac{\varphi_2(2^{-j})}{\varphi_1(2^{-j})}\right\}_{j\in\nat} \in \ell_\infty,
                  \end{align}
with                  
\begin{align}                  & s_1 >  s_2,  \qquad \text{or} \label{cond-s-1}\\
& s_1=s_2\quad \text{and}\quad q_1\le q_2. \label{cond-s-2}
		\end{align}
The embedding \eqref{emff_1} is never compact.  
\end{theorem}

\begin{proof}
		\emph{Step 1.}  We prove the sufficiency. Due to Theorem ~\ref{rswm}(i) and the wavelet decomposition, cf. Theorem~\ref{fwd},  it is enough to prove that $f^{s_1,\varphi_1}_{p_1,q_1}(\rd) \hookrightarrow f^{s_2,\varphi_2}_{p_2,q_2}(\rd) $. 
		If $s_1>s_2$  and $q_1>q_2$, then, using twice the H\"older inequality, we obtain for $\nu\in\mathbb{Z}$ and $k\in\zd$, 
		 \begin{align*}
			&\varphi_2(2^{-\nu})\left\{\frac{1}{|Q_{\nu,k}|}\int_{Q_{\nu,k}}\left[\sum_{j={\nu\vee 0}}^\infty
			2^{j s_2 q_2}\left( \sum_{m\in\zd}|\lzjm|\kjm(x) \right)^{q_2}\right]^\frac{p_2}{q_2}\dint x
			\right\}^\frac{1}{p_2}\\
			&\quad\leq\varphi_2(2^{-\nu})\left\{\frac{1}{|Q_{\nu,k}|}\int_{Q_{\nu,k}}\left[\sum_{j={\nu\vee 0}}^\infty
			2^{j s_1 q_1} \left( \sum_{m\in\zd}| \lzjm|\kjm(x) \right)^{q_1} \right]^\frac{p_2}{q_1}
			\underbrace{\left[\sum_{j=\nu\vee 0}^\infty2^{j(s_2-s_1)\frac{q_1q_2}{q_1-q_2}}\right]^\frac{p_2(q_1-q_2)}{q_1q_2}}_{\lesssim {2^{(s_2-s_1)(\nu \vee 0)}}} \dint x 	\right\}^\frac{1}{p_2}\\
			&\quad\lesssim 2^{(\nu\vee 0)(s_2-s_1)} \frac{\varphi_2(2^{-\nu})}{\varphi_1(2^{-\nu})}
		\varphi_1(2^{-\nu})\left\{\frac{1}{|Q_{\nu,k}|}\int_{Q_{\nu,k}}\left[ \sum_{j={\nu\vee 0}}^\infty 2^{j s_1 q_1} \left(\sum_{m\in\zd}|\lzjm|\kjm(x)\right)^{q_1}\right]^\frac{p_1}{q_1}\dint x
			\right\}^\frac{1}{p_1}\\
			&\quad\lesssim\sup_{\nu\in \mathbb{Z}}2^{(\nu\vee 0)(s_2-s_1)}  \frac{\varphi_2(2^{-\nu})}{\varphi_1(2^{-\nu})}
			\|\lz\mid f_{p_1,q_1}^{s_1,\varphi_1}(\rd)\| \le C \|\lz\mid f_{p_1,q_1}^{s_1,\varphi_1}(\rd)\| .
		\end{align*}
		{In the last inequality we have used \eqref{embee0_1},  \eqref{embee00_1} and $s_1\geq s_2$.} 
                The similar arguments hold if $q_1\le q_2$ {and $s_1\ge s_2$}.

\emph{Step 2.} {Now we prove the necessity.}  It follows from {Theorem~\ref{main-n}} that the condition \eqref{embee0_1} is necessary since 
\begin{align}\label{neen1}
    \mathcal{N}_{\varphi_1,p_1,\min(p_1,q_1)}^{s_1}(\rd)\eb\MfFa(\rd)
  \eb\MfFb(\rd)\eb\mathcal{N}_{\varphi_2,p_2,\infty}^{s_2}(\rd).
\end{align}

		We prove the necessity of the assumption $s_1\ge s_2$ . We may assume that $q_2=\infty$.  
		
		Assume  $s_1< s_2$.  
		We consider a family of sequences  $\lambda^{(k)}= (\lambda_{j,m}^{(k)})$, $k\in \nat_0$, defined by  
		\begin{equation*}
			\lambda_{j,m}^{(k)}=\begin{cases}
				2^{-js_1} &\quad \text{if} \qquad   j=k\quad \text{and}\quad Q_{j,m}\subset Q_{0,0}, \\
				0 & \quad \text{otherwise} .
			\end{cases}
		\end{equation*} 
		{We calculate $\|\lambda^{(k)}| e^{{s_1}}_{\varphi_1,p_1,q_1}\|$.} Let $P$ be a dyadic cube. The integral over $P$ is not zero if $P\cap Q_{0,0}\not= \emptyset$.  We may assume $\ell(P)\le 1$ since  bigger cubes do not influence the supremum. For such $P$ we have 
		\begin{align*}
			\frac{\varphi_1(\ell(P))}{|P|^{1/p_1}}
			\left( \int_P \big(\sum_{j=0}^\infty 2^{js_1q_1}\sum_m  |\lambda^{(k)}_{j,m}|^{q_1} \kjm(x)\big)^{\frac{p_1}{q_1}} \dint x \right)^{\frac{1}{p_1}} = 
			\frac{\varphi_1(\ell(P))}{|P|^{1/p_1}}
			|P|^{1/p_1} \le 1 . 
		\end{align*} 
		So $\|\lambda^{(k)}| e^{s_1}_{\varphi_1,p_1,q_1}(\rd)\|\le 1$ for any $k\in\no$. 
		On the other hand, 
		\begin{align*}
			\frac{\varphi_2(\ell(P))}{|P|^{1/p_2}}
			\left( \int_P \big(\sup_{j,m} 2^{js_2} |\lambda^{(k)}_{j,m}| \kjm(x)\big)^{p_2} \dint x \right)^{\frac{1}{p_2}} = 
			\frac{\varphi_2(\ell(P))}{|P|^{1/p_2}}
			2^{k(s_2-s_1)}|P|^{1/p_2}.
		\end{align*}
		Taking $P=Q_{0,0}$ we get $\lim_{k\rightarrow+\infty}\|\lambda^{(k)}| e^{s_2}_{\varphi_2,p_2,\infty}\|=\lim_{k\rightarrow+\infty}2^{k(s_2-s_1)} =\infty$.

		To prove the necessity of \eqref{embee00_1} it is sufficient to consider  sequences $\lambda^{(j)}=(\lambda^{(j)}_{\nu,m})$, $j\in\mathbb{N}$,  defined by
		\[ 
		\lambda^{(j)}_{\nu,m} = \begin{cases}
			1 & \quad \text{if}\quad \nu=j \; \text{and}\; m=0,\\
			0 & \quad \text{otherwise} .
		\end{cases}
	\]
	Then, if the embedding {\eqref{emff_1}} holds, we have the following inequality
	\begin{align*}
		2^{js_2}\varphi_2(2^{-j})= \|\lambda^{(j)} |e^{{s_2}}_{\varphi_2,p_2,q_2}(\rd)\| \le C 
		\|\lambda^{(j)} |e^{s_1}_{\varphi_1,p_1,q_1}(\rd)\| = C 2^{js_1}\varphi_1(2^{-j}).
	\end{align*}
This yields \eqref{embee00_1}.
        
Now let $s_1=s_2=s$. We would like to disprove that there is an embedding \eqref{emff_1} when $q_1> q_2$. Here we use an adapted idea from the proof of \cite[Theorem 3.1]{hs14} and construct {a} function $f\in\mathcal{E}_{\varphi_1,p_1,q_1}^{s}(\rd)$ with $f\notin\mathcal{E}_{\varphi_2,p_2,q_2}^{s}(\rd)$. 
Since \eqref{embee0_1} and \eqref{embee00_1} imply $\varphi_2(t)\leq c\, \varphi_1(t)$ 
and thus by \eqref{elme2}, 
 $\mathcal{E}_{\varphi_1,p_1,q_1}^{s}(\rd)\eb\mathcal{E}_{\varphi_2,{p_2},q_1}^{s}(\rd)$, it will be sufficient to construct $f\in\mathcal{E}_{\varphi_2,{p_2},q_1}^{s}(\rd)\setminus\mathcal{E}_{\varphi_2,p_2,q_2}^{s}(\rd)$.

Hence, it is sufficient to construct $f\in\mathcal{E}_{\varphi,p,q_1}^s(\rd)\setminus\mathcal{E}_{\varphi,p,q_2}^s(\rd)$, $\varphi\in\Gp$ and $q_1> q_2$.
{Assume $q_1> q_2$, and then} there exists a sequence
$\{\lambda_j\}_{j\in\no}$ such that
\begin{align*}
    \{2^{js}\lz_j\}_j\in\ell_{q_1}\backslash\ell_{q_2}.
\end{align*}
We define
\begin{align*}
    \lz_{j,m}=\left\{
\begin{array}{lcl}
\lz_j,      &     & {\text{if}\quad Q_{j,m}\subset Q_{0,0},}\\
0,    &     & {\text{otherwise}.}
\end{array} \right.
\end{align*}
{For convenience, let us assume $q_1<\infty$, otherwise the modification of the argument is obvious.} 
Then for any $x\in Q_{0,0}$, we have
\begin{align*}
    \sum_{j\in\no,m\in\zd}2^{jsq_1}|\lz_{j,m}\kjm(x)|^{q_1}=
\sum_{j\in\no}|2^{js}\lz_j|^{q_1}=c^{q_1}<\infty,
\end{align*}
whereas this sum equals zero if $x\notin Q_{0,0}$. Let
\begin{align*}
    f=\sum_{j\in\no,m\in\zd}\ljm\psi_{j,m}^{G_1}.
\end{align*}
Then,
\begin{align*}
&\|f\mid\mathcal{E}_{\varphi,p,q_1}^s(\rd)\|\\
&\quad\lesssim\left\|\left\{\ljm\right\}_{j\in\no,m\in\zd}\mid e_{\varphi,p,q_1}^s(\rd)\right\|\\
&\quad=
\sup_{\nu\in\no, k\in\zd}\varphi(\ell(Q_{\nu,k}))|Q_{\nu,k}|^{-\frac{1}{p}}\left[\int_{Q_{\nu,k}}\left(\sum_{j\in\no,m\in\zd}2^{j s q_1}
|\ljm|^{q_1}\kjm(x)^{q_1}\right)^{\frac{p}{q_1}}\dint x\right]^{\frac{1}{p}}.
\end{align*}
If $Q_{\nu,k}\cap Q_{0,0}=\emptyset$, then the last integral equals to zero, while in case of $Q_{\nu,k}\subseteq Q_{0,0}$, i.e. $\nu\in\no$, {we have}
\begin{align*}
&\varphi(2^{-\nu})|Q_{\nu,k}|^{-\frac{1}{p}}\left[\int_{Q_{\nu,k}}
\left(\sum_{j\in\no,m\in\zd}2^{j s q_1}
|\lz_{j,m}\kjm(x)|^{q_1}\right)^{\frac{p}{q_1}}\dint x\right]^{\frac{1}{p}}\\
&\qquad
\leq\varphi(2^{-\nu})|Q_{\nu,k}|^{-\frac{1}{p}}\left(\int_{Q_{\nu,k}}c^p\dint x\right)^{\frac{1}{p}}\leq C\, \varphi(1)=C.
\end{align*}
If $Q_{0,0}\subset Q_{\nu,k}$, that is $\nu\leq 0$, we obtain
\begin{align*}
&\varphi(2^{-\nu})|Q_{\nu,k}|^{-\frac{1}{p}}\left[\int_{Q_{\nu,k}}
\left(\sum_{j\in\no,m\in\zd}2^{j s q_1}
|\lz_{j,m}\kjm(x)|^{q_1}\right)^{\frac{p}{q_1}}\dint x\right]^{\frac{1}{p}}\\
&\qquad
\leq\varphi(2^{-\nu})|Q_{\nu,k}|^{-\frac{1}{p}}\left(\int_{Q_{0,0}}c^p\dint x\right)^{\frac{1}{p}}\leq C\, \varphi(2^{-j})2^{j\frac{d}{p}}\leq \wcc,
\end{align*}
due to the properties of $\varphi\in\Gp$, which imply $\varphi(t)t^{-\frac{d}{p}}\leq c$ for $t>1$.
Thus, taking the supremum over all $\nu\in\mathbb{Z}$ and $k\in\zd$,
\begin{align*}
    \|f\mid\mathcal{E}_{\varphi,p,q_1}^s(\rd)\|\leq c<\infty,
\end{align*}
that is, $f\in\mathcal{E}_{\varphi,p,q_1}^s(\rd)$.

On the other hand, for any $x\in Q_{0,0}$, we have
\begin{align*}
    \sum_{j\in\no,m\in\zd}2^{jsq_2}|\lz_{j,m}|^{q_2}\kjm(x)^{q_2}=
\sum_{j\in\no}|2^{js}\lz_j|^{q_2}=\infty
\end{align*}
and thus $f\notin\mathcal{E}_{\varphi,p,q_2}^s(\rd)$.\\

\emph{Step 3.} Finally, we prove that the embedding \eqref{emff_1} is never compact. Note that \eqref{neen1} immediately results in the non-compactness of \eqref{emff_1} in all cases, since otherwise there would appear a contradiction to the non-compactness of the outer embedding \eqref{neen1} in all cases, according to \cite[Theorem 5.1]{hms22}. This finishes the proof.
\end{proof}

\begin{example}\label{rem-class}
Let $\varphi_i(t):=t^\frac{d}{u_i}$, $0<p_i\leq u_i<\infty$, $i=1,2$. {Assume $p_1\geq p_2$, hence $\varrho=1$.} Then \eqref{emff_1} corresponds to \eqref{eeu1}. Condition \eqref{embee0_1} reads as
\begin{align*}
    \sup_{t\geq 1}t^{d(\frac{1}{u_2}-\frac{1}{u_1})}<\infty
\end{align*}
which is equivalent to \eqref{eeu2} with $\varrho =1$.  
Moreover, we have
\begin{align*}
    2^{j(s_2-s_1)}\frac{\varphi_2(2^{-j})}{\varphi_1(2^{-j})}=2^{-j(s_1-s_2-\frac{d}{u_1}+\frac{d}{u_2})}. 
\end{align*}
Hence \eqref{embee00_1} (together with \eqref{embee0_1}) corresponds to $s_1-s_2\geq \frac{d}{u_1}-\frac{d}{u_2}\geq 0$. If $s_1-s_2>\frac{d}{u_1}-\frac{d}{u_2}\geq 0$, then \eqref{cond-s-1} coincides with \eqref{eeu3} in that case, {while $s_1-s_2=\frac{d}{u_1}-\frac{d}{u_2}> 0$ refers to \eqref{eeu4}}. 
Finally, let $s_1=s_2$ and thus $u_1=u_2$. {Then \eqref{cond-s-2} is just \eqref{eeu5}}.
 \end{example}

\begin{example}
  We return to the identity \eqref{E=Mf} in combination with Theorem~\ref{E-E-case1}. 
Let $1<p_2\leq p_1<\infty$, and $\varphi_i\in \mathcal{G}_{p_i}$, for $i=1,2$, satisfying \eqref{intc}. We assume without loss of generality that $\varphi_1(1)=\varphi_2(1)=1$. 
    Then
    \begin{align}
      \mathcal{M}_{\varphi_1,p_1}(\rd) \eb \mathcal{M}_{\varphi_2,p_2}(\rd)
\quad \text{if, and only if,}\quad \varphi_2(t) \lesssim\varphi_1(t), \quad t>0.
	\end{align}
This follows immediately by Theorem~\ref{E-E-case1} with $s_1=s_2=0$, $q_1=q_2=2$, and coincides with \eqref{elme2} in this case $p_1\geq p_2$.
\end{example}

  We can also formulate Theorem~\ref{E-E-case1} in a version which appears even closer to Theorem~\ref{main-n}.

  \begin{corollary}\label{Cor-E-alpha}
    Let $0<p_2\leq p_1<\infty$, $s_i\in\real$, $0<q_i\leq\infty$, $\varphi_i\in \mathcal{G}_{p_i}$ and $\varphi_i$ satisfies \eqref{intc} when $q_i<\infty$, for $i=1,2$. We assume without loss of generality that $\varphi_1(1)=\varphi_2(1)=1$. 
    Let $\alpha_j=\sup\limits_{\nu\leq j}\frac{\varphi_2(2^{-\nu})}{\varphi_1(2^{-\nu})}$, $j\in\no$.

    Then \eqref{emff_1} holds if, and only if, \eqref{embee0_1}  is satisfied and
		\begin{align}\label{embe-a-1}
             \left\{2^{j(s_2-s_1)}\alpha_j\right\}_{j\in\nat} \in \ell_\infty,
                  \end{align}
with $q_1\leq q_2$ in case of $s_1=s_2$.                 
  \end{corollary}

  \begin{proof}
Note that $\alpha_j\geq \frac{\varphi_2(2^{-j})}{\varphi_1(2^{-j})}$, $j\in\nat$, so \eqref{embe-a-1} implies \eqref{embee00_1}. On the other hand, $\alpha_j\geq 1$, $j\in\nat$, so \eqref{embe-a-1} also implies \eqref{cond-s-1}, \eqref{cond-s-2} in view of our assumption. This concludes the argument for the sufficiency part in view of Theorem~\ref{E-E-case1}. Conversely, assume that \eqref{emff_1} holds, then by \eqref{neen1} we have the embedding 
$ \mathcal{N}_{\varphi_1,p_1,\min(p_1,q_1)}^{s_1}(\rd) \hookrightarrow  \mathcal{N}_{\varphi_2,p_2,\infty}^{s_2}(\rd)$,
which by Theorem~\ref{main-n} results in \eqref{embee0_1} and \eqref{embe-a-1}, since $q^\ast=\infty$ and $\varrho=1$ in our setting. The argument for $q_1\leq q_2$ in case of $s_1=s_2$ is the same as in the proof of Theorem~\ref{E-E-case1}.
\end{proof}

\subsubsection{ The case $p_1<p_2$}

The more tricky case is  $p_1<p_2$. { Note that the number $\varrho$ appearing in Theorem~\ref{main-n} satisfies $\varrho = \frac{p_1}{p_2}<1$ in this case.} Before we can give our results we would like to introduce some further (local) characteristic of a function $\varphi\in \mathcal{G}=\bigcup_{p>0}\Gp = \lim_{p \to 0} \Gp$. More precisely, let us consider the (local) class
\[
  \mathcal{G}^{\mathrm{loc}} = \bigcup_{p>0}\Gp^{\mathrm{loc}} = \lim_{p \to 0} \Gp^{\mathrm{loc}},
  \]
where $\Gp^{\mathrm{loc}}$ contains all non-decreasing functions $\varphi:(0,\,\infty)\rightarrow(0,\,\infty)$ such that 
\begin{align}\label{Gp-def-local}
t^{-\frac{d}{p}}\varphi(t)\geq s^{-\frac{d}{p}}\varphi(s), \quad 0<t\leq s<1.
\end{align}
Clearly, $\Gp \subset \Gp^{\mathrm{loc}}$, $p>0$, recall \eqref{Gp-def}. Now let $\varphi \in \mathcal{G}^{\mathrm{loc}}$. We introduce
	\begin{align}
		\rphi{\varphi} = \sup\{p:\; \varphi\in\Gp^{\mathrm{loc}}\},\quad \varphi \in \mathcal{G}^{\mathrm{loc}}.
	\end{align}	
	Then 	$0<\rphi{\varphi} <\infty$ or $\rphi{\varphi} =\infty$.  { We first find that $\rphi{\varphi} =\infty$ implies $\varphi= \mathrm{ const}$ on $(0,1)$. }

          \begin{lemma}\label{Gop}
            Let $\varphi \in \mathcal{G}^{\mathrm{loc}}$.
 \begin{enumerate}[\bfseries\upshape  (i)]
 \item
 If $\rphi{\varphi}<\infty$, then $\varphi\in \mathcal{G}_{\rphi{\varphi}}^{\mathrm{loc}}$, that is, $\rphi{\varphi}=\max\{p:\; \varphi\in\Gp^{\mathrm{loc}}\}$. 	
		\item
                  Moreover, $\rphi{\varphi}=\infty$ if, and only if, $\varphi = \mathrm{const}$ on $(0,1)$.
                  \end{enumerate}
	\end{lemma}

        \begin{proof}
{Note that $\varphi\in\mathcal{G}^{\mathrm{loc}}$ implies that $\varphi:(0,\infty)\to(0,\infty)$ is non-decreasing.} Assume first
  $\rphi{\varphi}<\infty$   and $0<t<s<1$. Then 
	\begin{align*}
		\Big(\frac{s}{t}\Big)^{-\frac{\nd}{p}} \le \frac{\varphi(t)}{\varphi(s)} 
	\end{align*} 
for any $p<\rphi{\varphi}$. So 
	\begin{align*}
	\Big(\frac{s}{t}\Big)^{-\frac{\nd}{\rphi{\varphi}}}= \sup_{p<\rphi{\varphi} } \Big(\frac{s}{t}\Big)^{-\frac{\nd}{p}}  \le \frac{\varphi(t)}{\varphi(s)},  \quad 0<t<s<1,
        \end{align*}
        which proves $\varphi\in \mathcal{G}_{\rphi{\varphi}}^{\mathrm{loc}}$.
        
        {Now we deal with (ii). Obviously, any  $\varphi$ constant on $(0,1)$ belongs to all $\Gp^{\mathrm{loc}}$ which implies $\rphi{\varphi}=\infty$. Conversely, if  $\rphi{\varphi}=\infty$, then $\varphi\in \Gp^{\mathrm{loc}}$ for all $p>0$. Hence          for $0<t\le s<1$} we have 
\begin{align*}
	\varphi(t)\le \varphi(s)= \lim_{p\rightarrow \infty} s^{-\frac{d}{p}}\varphi(s)\le 
	\lim_{p\rightarrow \infty} t^{-\frac{d}{p}}\varphi(t)=\varphi(t), 
\end{align*}
so $\varphi\equiv \mathrm{const}$ on $(0,1)$.
\end{proof}

\begin{remark}
  Clearly one can also consider the global counterpart, that is, $\rphi{\varphi}^\ast = \sup\{p:\; \varphi\in\Gp\}$ and obtain a parallel result to Lemma~\ref{Gop}. Obviously $\Gp \subset \Gp^{\mathrm{loc}}$ implies $\rphi{\varphi}^\ast \leq \rphi{\varphi}$. But as we only need the local version in our later considerations, we restricted ourselves to this setting. Note that in general $\varphi\in\Gp$ implies $p\leq\rphi{\varphi}^\ast\leq \rphi{\varphi}$.
  \end{remark}

We shall assume in the sequel that $\rphi{\varphi} < \infty$.

\begin{exams}
  We consider our examples \eqref{ex-u-v} and \eqref{ex-log}.
  \begin{enumerate}[\bfseries\upshape  (i)]
  \item
    Let $0<u,\,v\leq \infty$. Then for 
$$
\varphi_{u,v}(t)=\begin{cases}
t^{\frac{d}{u}},\quad{\text{if}}\quad t\leq 1,\\
t^{\frac{d}{v}},\quad{\text{if}}\quad t> 1,
\end{cases}
$$
we find that $\rphi{\varphi_{u,v}}=u $.
\item
Then for
$$
\varphi(t)=\begin{cases} \frac{1}{\log 2} \log(1+t), & 0<t<1, \\ t, & t\geq 1.\end{cases}
  $$
we find that $\rphi{\varphi}=\nd$.
\end{enumerate}
\end{exams}

Next we give some Gagliardo-Nirenberg inequalities adapted to our setting. 
The idea of the proof follows from \cite{sic13}.

\begin{proposition}\label{GN}
	Let $-\infty<s_1<s_0<\infty$, {$0<\theta<1$}, $0<q_i\leq\infty$, $0<p_i<\infty$, $\varphi_i\in\mathcal{G}_{p_i}$, and $\varphi_i$ satisfy \eqref{intc} if $q_i<\infty$,  $i=0,1,2$. 
	\begin{enumerate}[\bfseries\upshape  (i)]
		\item Let $0<p_0,p_1<\infty$. Assume that
		\begin{align}\label{GN2}
			\frac{1}{p_2} = \frac{1-\theta}{p_1}+\frac{\theta}{p_0},\qquad
			s_2 = s_1 (1-\theta) +  s_0 \theta \qquad
			\text{and}\quad
			\varphi_2(2^{j}) =\varphi_{1}(2^{j})^{1-\theta}\varphi_{0}(2^{j})^{\theta} , \,\,\, j\in {\mathbb{Z}}.
		\end{align}        
		Then for any $f\in \mathcal{S}'(\real^d)$ and any $q_2$ 
		\[ 
		\|f|F^{s_2,\varphi_2}_{p_2,q_2}(\real^d)\|\le \|f|F^{s_1,\varphi_1}_{p_1,q_1}(\real^d)\|^{1-\theta} \|f|F^{s_0,\varphi_0}_{p_0,q_0}(\real^d)\|^{\theta}.
		\]
		
		\item Let $0<p_1<\infty$. Assume that
		\[ \frac{1}{p_2}= \frac{1-\theta}{p_1},\qquad s_2=(1-\theta)s_1+ \theta s_0,\qquad\varphi_2=\varphi_{1}^{1-\theta}\varphi_{0}^{\theta} .\] 
		Then for any $f\in \mathcal{S}'(\real^d)$ and any $q_2$ 
		\[ 
		\|f|F^{s_2,\varphi_2}_{p_2,q_2}(\real^d)\|\le \|f|F^{s_1,\varphi_1}_{p_1,q_1}(\real^d)\|^{1-\theta} \|f|F^{s_0,\varphi_0}_{\infty,\infty}(\real^d)\|^{\theta}.
		\] 
	\end{enumerate}
\end{proposition}
\begin{proof}
	The  proof is similar to \cite[Proposition 3.1]{sic13}, we give the details for reader's convenience.
	
	We use the following inequality  introduced by Oru, cf. \cite{oru98}; see also \cite{bm01}. 
	Let $s_2=(1-\theta)s_1+\theta s_0$ and $0<q\leq \infty$. For any sequence $\{a_j\}_j$ of complex numbers,
	\begin{align*}
		\|\{2^{j s_2}a_j\}_j\mid\ell_q\|\leq c\,\|\{2^{j s_1}a_j\}_j\mid\ell_\infty\|^{1-\theta}
		\|\{2^{j s_0}a_j\}_j\mid\ell_\infty\|^{\theta},
	\end{align*}
	where $c=C_{(s_0,s_1,\theta,q)}$, but does not depend on $\{a_j\}_j$. This can be applied to
	$a_j=|\fuff(x)|$. Then, if $p_0<\infty$, conditions \eqref{GN2}, H\"older's inequality and the monotonicity with respect to $q_0,q_1$ yield
	\begin{align*}
		\|f\mid F_{p_2,q_2}^{s_2,\varphi_2}(\rd)\|
		&\lesssim \sup_{P\in\mq}\varphi_2(\ell(P))\left[\frac{1}{|P|}\int_P
		\sup_{j\geq\jjp}2^{j s_1}|\fuff(x)|^{p_1}\dint x\right]^\frac{1-\theta}{p_1}\\
		&\qquad\times \left[\frac{1}{|P|}\int_P
		\sup_{j\geq\jjp}2^{j s_0}|\fuff(x)|^{p_0}\dint x\right]^\frac{\theta}{p_0}\\
		&\lesssim \sup_{P\in\mq}\varphi_1(\ell(P))^{1-\theta}\left[\frac{1}{|P|}\int_P
		\sup_{j\geq\jjp}2^{j s_1}|\fuff(x)|^{p_1}\dint x\right]^\frac{1-\theta}{p_1}\\
		&\qquad\times \sup_{P\in\mq}\varphi_0(\ell(P))^{\theta}\left[\frac{1}{|P|}\int_P
		\sup_{j\geq\jjp}2^{j s_0}|\fuff(x)|^{p_0}\dint x\right]^\frac{\theta}{p_0}\\
		&\lesssim \|f\mid F_{p_1,\infty}^{s_1,\varphi_1}(\rd)\|^{1-\theta}
		\|f\mid F_{p_0,\infty}^{s_0,\varphi_0}(\rd)\|^{\theta}\\
		&\lesssim \|f\mid F_{p_1,q_1}^{s_1,\varphi_1}(\rd)\|^{1-\theta}
		\|f\mid F_{p_0,q_0}^{s_0,\varphi_0}(\rd)\|^{\theta}.
	\end{align*}
	In the case of $p_0=\infty$, we have
	\begin{align*}
		\|f\mid F_{p_2,q_2}^{s_2,\varphi_2}(\rd)\|
		&\lesssim \|f\mid F_{p_1,\infty}^{s_1,\varphi_1}(\rd)\|^{1-\theta}
		\sup_{P\in\mq} \varphi_0(\ell(P))^{\theta}
		\sup_{x\in P} \left(\sup_{j\geq\jjp}|2^{j s_0}\fuff(x)|\right)^\theta\\
		&\lesssim \|f\mid F_{p_1,q_1}^{s_1,\varphi_1}(\rd)\|^{1-\theta}
		\|f\mid F_{\infty,\infty}^{s_0,\varphi_0}(\rd)\|^{\theta}.
	\end{align*}
	Thus, we complete the proof.
\end{proof}


\begin{lemma}\label{inftyinfty} Let $s\in \real$, $0<p<\infty$, $0<q\le \infty$ and $\varphi\in \Gp$ with $\rphi{\varphi}<\infty$.  Assume in addition that \eqref{intc}  holds when $q<\infty$. Then
	\begin{equation}
		\MfF(\rd) \eb F^{s-\frac{\nd}{\rphi{\varphi}}}_{\infty,\infty}(\rd).	
	\end{equation}
\end{lemma}
	
	\begin{proof} Let $f\in\MfF(\rd)$. The wavelet decomposition theorem, in particular Theorem~\ref{mew}, implies
\begin{align*}
	f=\sum_{m\in\zd}\lambda_m\psi_m+\sum_{G\in G^\ast}\sum_{j\in\no}\sum_{m\in\zd}\ljm^{G}2^{-\frac{jd}{2}}\psi_{j,m}^G,
	\qquad\lambda\in {{\tilde{e}}_{\varphi,p,q}(\rd)}.
\end{align*}		
For any $j\in \nat_0$ and  $m\in \mathbb{Z}^d$ 
we have  
\begin{align*}
   2^{j(s-\frac{d}{\rphi{\varphi}})}|\lambda_{j,m}| \le &2^{-j\frac{d}{\rphi{\varphi}}} \frac{1}{|Q_{j,m}|} \int_{Q_{j,m}} 2^{js}|\lambda_{j,m}| \chi_{j,m}(x)  \dint x \le  
    \varphi(2^{-j}) \frac{1}{|Q_{j,m}|} \int_{Q_{j,m}} 2^{js}|\lambda_{j,m}| \chi_{j,m}(x)  \dint x \\
    \le &	\|\lambda\mid\mff(\rd)\|
\end{align*}
since the function  $t^{-\frac{d}{\rphi{\varphi}}}\varphi(t)$ is decreasing {in $(0,1]$}. 
Taking the supremum  over $j$ and $m$ we get the result. 
	\end{proof}

We now give the counterpart of Theorem ~\ref{E-E-case1} {and} Corollary~\ref{Cor-E-alpha}.

        \begin{theorem}\label{E-E-case2}
    Let $0<p_1 < p_2<\infty$, $s_i\in\real$, $0<q_i\leq\infty$, $\varphi_i\in \mathcal{G}_{p_i}$ and $\varphi_i$ satisfies \eqref{intc} when $q_i<\infty$, for $i=1,2$. We assume without loss of generality that $\varphi_1(1)=\varphi_2(1)=1$. 
   Let $\varrho = \frac{p_1}{p_2}<1$ and $\alpha_j=\sup\limits_{\nu\leq j}\frac{\varphi_2(2^{-\nu})}{\varphi_1(2^{-\nu})^\varrho}$, $j\in\no$.
    \begin{enumerate}[\bfseries\upshape  (i)]
\item   
  Then \eqref{emff_1} holds if
\begin{align}
			\sup_{\nu\in\no }\frac{\varphi_2(2^{\nu})}{\varphi_1(2^{\nu})^\varrho}<\infty, \label{embee0_2} 
		    \end{align}
is satisfied and
		\begin{align}\label{embe-a-2}
             \left\{2^{j(s_2-s_1+\frac{\nd}{\rphi{\varphi_1}}(1-\varrho))}\alpha_j\right\}_{j\in\nat} \in \ell_\infty.
                  \end{align}
                \item
                  If \eqref{emff_1} holds, then \eqref{embee0_2}  is satisfied and
              \begin{align}\label{embe-a-3}
             \left\{2^{j(s_2-s_1)}\alpha_j \frac{\varphi_1(2^{-j})^\varrho}{\varphi_1(2^{-j})}\right\}_{j\in\nat} \in \ell_\infty.
                  \end{align}    
  \item
                  Assume that
                  \begin{equation}\label{embe-a-4}
                    \lim_{t\to 0} \varphi_1(t) t^{-\frac{\nd}{\rphi{\varphi_1}}}\leq c<\infty,
                    \end{equation}
                    then \eqref{embe-a-2} and \eqref{embe-a-3} are equivalent, that is, the embedding \eqref{emff_1} holds if, and only if, \eqref{embee0_2} and \eqref{embe-a-2} are satisfied.
  \item
                 The embedding \eqref{emff_1} is never compact.
\end{enumerate}
 \end{theorem}

 \begin{proof}
   {\em Step 1}. Using the embeddings \eqref{neen1} together with Theorem~\ref{main-n} immediately yields (ii), while part (iv) is completely parallel to the proof of Theorem~\ref{E-E-case1}.\\

   {\em Step 2}. Next we show that \eqref{embe-a-2} implies
   \begin{align}
     \label{ineq.26}
&     s_2\le s_1 -\frac{\nd}{\rphi{\varphi_1}} (1-\varrho), \\			
& 	\sup_{j\in\nat} 2^{j(s_2-s_1 +\frac{\nd}{\rphi{\varphi_1}}(1-\varrho))} \frac{\varphi_2(2^{-j})}{\varphi_1(2^{-j})^\varrho}<\infty,
          \label{ineq.27}
     \end{align}
     and then \eqref{embee0_2} together with \eqref{ineq.26} and \eqref{ineq.27} are sufficient for \eqref{emff_1}. 

     Note that $\alpha_j\geq 1$ and $\alpha_j\geq \frac{\varphi_2(2^{-j})}{\varphi_1(2^{-j})^\varrho}$, then \eqref{embe-a-2} immediately implies \eqref{ineq.26} and \eqref{ineq.27}. 
     {First} we assume that $\varphi_2=\varphi_1^\varrho$. In that case  $\rphi{\varphi_1}= \varrho \rphi{\varphi_2}$. If \eqref{ineq.27} holds, then $s_2-s_1\le \frac{\nd}{\rphi{\varphi_1}}(\varrho-1)= \frac{\nd}{\rphi{\varphi_2}} -\frac{\nd}{\rphi{\varphi_1}}$. Thus, {Proposition~\ref{GN}(ii), with $1-\theta=\varrho$ and $s_0= s_1-\frac{\nd}{\rphi{\varphi_1}}$,  and Lemma~\ref{inftyinfty} }  imply   
          \[ 
         \|f|F^{s_2,\varphi_2}_{p_2,q_2}(\rd)\|\le \ \|f|F^{s_1,\varphi_1}_{p_1,q_1}(\rd)\|.
         \] 
         
         In the general case we take $s_3= s_1 -\frac{\nd}{\rphi{\varphi_1}} (1-\varrho)$ and use the following factorisation 
         \[F^{s_1,\varphi_1}_{p_1,q_1}(\rd) \hookrightarrow F^{s_3,\varphi_1^\varrho}_{p_2,q_2}(\rd) \hookrightarrow F^{s_2,\varphi_2}_{p_2,q_2}(\rd) , \]
         where the continuity of both embeddings is guaranteed by \eqref{ineq.26} and \eqref{ineq.27} {if we  apply  Theorem \ref{E-E-case1} and Theorem~\ref{rswm}(i). }  \\

         {\em Step 3}. It remains to show that in case of \eqref{embe-a-4} the conditions \eqref{embe-a-2} and \eqref{embe-a-3} are equivalent. Note that \eqref{embe-a-4} yields that $0<c_1\leq \varphi_1(t) t^{-\frac{\nd}{\rphi{\varphi_1}}} \leq c_2<\infty$, $0<t<1$, thus
         \[
           2^{-j \frac{\nd}{\rphi{\varphi_1}}(\varrho-1)} \sim \varphi_1(2^{-j})^{\varrho-1}, \quad j\in\nat,
           \]
proves the equivalence of    \eqref{embe-a-2} and \eqref{embe-a-3}.   
                  \end{proof}

\begin{remark}
Note that in Step~2 of the proof of Theorem~\ref{E-E-case2} we used that \eqref{ineq.26} and \eqref{ineq.27} together with \eqref{embee0_2} are sufficient for the embedding \eqref{emff_1}. Let us mention that in general these two conditions do not imply each other. Take, for instance, our example function $\varphi_{u,v}$ given by \eqref{ex-u-v}. Then \eqref{embee0_2} requires $\varrho\geq \frac{v_1}{v_2}$, and for $\varrho\leq \frac{u_1}{u_2}$, then 
\eqref{ineq.26} implies \eqref{ineq.27}, while in case of $\varrho\geq \frac{u_1}{u_2}$,  \eqref{ineq.27} implies \eqref{ineq.26}. In the special case of $u_i=v_i$, $i=1,2$, referring to the Morrey spaces $\MF$ or the classical spaces $\F$, we are thus always in the situation that \eqref{ineq.27} implies \eqref{ineq.26}. 
\end{remark}

\begin{example}
  Note that our example function $\varphi_{\rphi{\varphi},v}$ given by \eqref{ex-u-v} with $u=\rphi{\varphi}$ obviously satisfies condition \eqref{embe-a-4}, but also example \eqref{ex-log} does so,
  \[
    \lim_{t\to 0} t^{-1} \varphi(t) = \frac{1}{\log 2} \lim_{t\to 0} \frac{\log (1+t)}{t} = \frac{1}{\log 2},
  \]
  recall $\rphi{\varphi}=\nd$.
\end{example}

\begin{example}
    Let $\varphi_i(t):=t^\frac{d}{u_i}$, $0<p_i\leq u_i<\infty$, $i=1,2$. Similar to Remark~\ref{rem-class}, condition \eqref{embee0_2} is equivalent to \eqref{eeu2} with $\varrho< 1$. 
    Moreover, let $\varrho=\frac{p_1}{p_2}<1$, then
    \begin{align*}
        \alpha_j=\sup_{\nu\leq j}2^{-\nu d (\frac{1}{u_2}-\frac{\varrho}{u_1})}
        =2^{-j d (\frac{1}{u_2}-\frac{\varrho}{u_1})},
    \end{align*}
    and then
    \begin{align*}
        2^{j(s_2-s_1+\frac{\nd}{\rphi{\varphi_1}}(1-\varrho))}\alpha_j
        =2^{j(s_2-s_1)}\alpha_j \frac{\varphi_1(2^{-j})^\varrho}{\varphi_1(2^{-j})}
        =2^{-j(s_1-s_2-\frac{d}{u_1}+\frac{d}{u_2})},
    \end{align*}
    since $\rphi{\varphi_1}=u_1$. 
    Assume that $s_1-s_2=d(\frac{1}{u_1}-\frac{1}{u_2})$. Note that $\varrho<1$ means $u_1<u_2$ which corresponds to \eqref{eeu4}. Thus \eqref{embe-a-2} and \eqref{embe-a-3} together with $\varrho<1$ are the counterpart of \eqref{eeu3} and \eqref{eeu4} in this settings.
\end{example}

\begin{remark}\label{R-M-M-emb-2}
We deal again with the special setting \eqref{E=Mf}, now in case of $1<p_1<p_2<\infty$, $\varphi_i\in \mathcal{G}_{p_i}$, for $i=1,2$, satisfying \eqref{intc}. We assume without loss of generality that $\varphi_1(1)=\varphi_2(1)=1$. Note first that the sufficient condition for the embedding $\mathcal{M}_{\varphi_1,p_1}(\rd) \eb \mathcal{M}_{\varphi_2,p_2}(\rd) $, as obtained in Theorem~\ref{E-E-case2}(i) cannot be applied, since \eqref{embe-a-2}  is never satisfied in view of $s_1=s_2=0$, $\rphi{\varphi_1}<\infty$, $\varrho<1$ and $\alpha_j\geq 1$, $j\in\nat$. Conversely, if $\mathcal{M}_{\varphi_1,p_1}(\rd) \eb \mathcal{M}_{\varphi_2,p_2}(\rd) $ holds, then by Theorem~\ref{E-E-case2}(ii), \eqref{embee0_2} and \eqref{embe-a-3} (with $s_1=s_2=0$) have to be satisfied. The latter implies $\varphi_1(t)\geq c>0$, $0<t<1$, in view of $\varrho<1$ and $\alpha_j\geq 1$, $j\in\nat$. In particular, $\inf_{t>0} \varphi_1(t)>0$ which implies $\mathcal{M}_{\varphi_1,p_1}(\rd) \eb L_\infty(\rd)$. Note that this does not lead to a contradiction with \eqref{elme2}, since we do not have $\inf_{t>0} \varphi_1(t)=0$ as required there. Hence, it would be possible, for instance, that $\varphi_2(t)\leq c'$, $t>0$, such that $L_\infty(\rd)\eb \mathcal{M}_{\varphi_2,p_2}(\rd)$ and we have thus the chain of embeddings $\mathcal{M}_{\varphi_1,p_1}(\rd) \eb L_\infty(\rd)\eb  \mathcal{M}_{\varphi_2,p_2}(\rd) $ even in the case $p_1<p_2$ (for appropriate $\varphi_i$, $i=1,2$). This also emphasises the importance of the additional assumption $\inf_{t>0} \varphi_1(t)=0$ in \eqref{elme2}. Likewise it indicates the gap between necessary and sufficient condition in Theorem~\ref{E-E-case2}. Moreover, according to Theorem~\ref{E-E-case2}(iii), if \eqref{embe-a-4} is satisfied, then  $\mathcal{M}_{\varphi_1,p_1}(\rd) \eb \mathcal{M}_{\varphi_2,p_2}(\rd) $ if, and only if, \eqref{embee0_2} and \eqref{embe-a-2} are satisfied, that is, never in our setting -- as explained above (no matter what the $\varphi_i$ satisfy otherwise). This is in good coincidence with \eqref{elme2} again.
\end{remark}

For convenience we formulate a last corollary, covering both cases $p_1\geq p_2$ and $p_1<p_2$.
 \begin{corollary}\label{CorF-F}
 Let $0<p_i<\infty$, $s_i\in\real$, $0<q_i\leq\infty$, $\varphi_i\in \mathcal{G}_{p_i}$ and $\varphi_i$ satisfies \eqref{intc} when $q_i<\infty$, for $i=1,2$. We assume without loss of generality that $\varphi_1(1)=\varphi_2(1)=1$. 
 Let $\varrho=\min(1,\frac{p_1}{p_2})$ and $\alpha_j=\sup\limits_{\nu\leq j}\frac{\varphi_2(2^{-\nu})}{\varphi_1(2^{-\nu})^\varrho}$, $j\in\no$.    Then 
    	\begin{align*}
			\MfFa(\rd)\eb\MfFb(\rd)
		\end{align*}
		if, and only if,
			\begin{align*}
			\sup_{\nu\in\no}\frac{\varphi_2(2^{\nu})}{\varphi_1(2^{\nu})^{\varrho}}<\infty, 
		    \end{align*}
and	    
               \begin{align*}
   \left\{2^{j(s_2-s_1+\frac{\nd}{\rphi{\varphi_1}}(1-\varrho))}\alpha_j\right\}_{j\in\nat} \in \ell_\infty,
                  \end{align*}     
where in case of $p_1<p_2$ we assume, in addition, \eqref{embe-a-4}, and in the case $p_1 \geq p_2$ we assume $q_1\leq q_2$ when $s_1=s_2$.
\end{corollary}

\subsection{ { Some particular embeddings} }
{ In this section we are interested mainly in embeddings that have the  classical Triebel-Lizorkin space as a source or target space. But we start with a statement  that is a counterpart of Corollary \ref{CorF-F}  for  the scale $\ftt(\rd)$. }

\begin{corollary}\label{main-fw}
    Let $0<p_i<\infty$, $s_i\in\real$, $0<q_i\leq \infty$ and $\varphi_i\in\mathcal{G}_{p_i}$ satisfying \eqref{intc} when $q_i<\infty$, $i=1,2$.  
    We assume without loss of generality that $\varphi_1(1)=\varphi_2(1)=1$. Let $\alpha_j=\sup_{\nu\leq j}\frac{\varphi_2(2^{-\nu})}{\varphi_1(2^{-\nu})^\varrho}$, $j\in\no$. 
    \begin{enumerate}[\bfseries\upshape  (i)]
      \item
Let $0<p_2\leq p_1<\infty$. 
        Then we have the continuous embedding 
    \begin{align}\label{efwvp}
        \ftta(\rd)\eb\fttb(\rd)
    \end{align}
    if, and only if, \eqref{embee0_2} and \eqref{embe-a-1} are satisfied with $q_1\leq q_2$ in case of $s_1=s_2$.
  \item
    Let $0<p_1<p_2<\infty$, $\varrho=\frac{p_1}{p_2}<1$. Then \eqref{efwvp} holds if \eqref{embee0_2} and \eqref{embe-a-2} are satisfied. Conversely,  \eqref{efwvp} implies that  \eqref{embee0_2} and \eqref{embe-a-3} are satisfied. 
    If ${\varphi_1}$ satisfies \eqref{embe-a-4}, then \eqref{efwvp} holds if, and only if, \eqref{embee0_2} and \eqref{embe-a-2} are satisfied.
   
\item
  The embedding  \eqref{efwvp} is never compact.
\end{enumerate}
  \end{corollary}

  \begin{proof}
    This is a direct consequence of Corollary~\ref{Cor-E-alpha} and   Theorem~\ref{E-E-case2} together with Theorem~\ref{rswm}(i).
  \end{proof}

\begin{corollary}\label{Cor-Ef-F}
    Let $s_i\in\real$, $0<p_i<\infty$ and $0<q_i \leq \infty$, $i=1,2$. Assume that $\varphi\in\mathcal{G}_{p_1}$ and $\varphi$ satisfies \eqref{intc} when $q_1<\infty$. We consider the embedding
    \begin{align}\label{E-F}
        {F^{s_1,\varphi}_{p_1,q_1}(\rd) =  \mathcal{E}^{s_1}_{\varphi, p_1,q_1}(\rd)} \eb F_{p_2,q_2}^{s_2}(\rd).
    \end{align}
    \begin{enumerate}[\bfseries\upshape  (i)]
    \item If $p_1\geq p_2$, then \eqref{E-F} is continuous if, and only if,
    \begin{align}\label{E-F-1}
        \varphi(t)\gtrsim t^\frac{d}{p_2},\quad t\geq 1,
    \end{align}
    and
    \begin{align*}
        \left\{2^{j(s_2-s_1-\frac{d}{p_2})}\varphi(2^{-j})^{-1}\right\}_{j\in\nat}\in\ell_\infty
    \end{align*}
    with
    \begin{align*}
        s_1>s_2\qquad\text{or}\qquad s_1=s_2\quad\text{and}\quad q_1\leq q_2.
    \end{align*}
    \item Assume that $p_1<p_2$.
    \begin{enumerate}[\bfseries\upshape  (a)]
    \item Then \eqref{E-F} is continuous if
    \begin{align}\label{e-f-c1}
        \varphi(t)\sim t^\frac{d}{p_1},\quad t\geq 1,
    \end{align}
    and 
    \begin{align}\label{e-f-c2}
        s_2\leq s_1-\frac{d}{\rphi{\varphi}}\left(1-\frac{p_1}{p_2}\right)
    \end{align}
    are satisfied.
    \item If \eqref{E-F} is continuous, then \eqref{e-f-c1} is satisfied and 
    \begin{align}\label{e-f-c3}
        \left\{2^{j(s_2-s_1)}\varphi(2^{-j})^{\frac{p_1}{p_2}-1}\right\}_{j\in\nat}\in\ell_\infty.
    \end{align}
  \item Assume that
\begin{equation}\label{embe-a-4'}
  \lim_{t\to 0} \varphi(t) t^{-\frac{d}{\rphi{\varphi}}} \leq c<\infty
  \end{equation}
holds, then \eqref{e-f-c2} and \eqref{e-f-c3} are equivalent, that is, the embedding \eqref{E-F} is continuous if, and only if, \eqref{e-f-c1} and \eqref{e-f-c2} are satisfied.
    \end{enumerate}
    \end{enumerate}
\end{corollary}

\begin{proof}
  {Let $\varphi_2(t)=t^\frac{d}{p_2}$ and  $\varphi_1(t)=\varphi(t)$, $t>0$. Then}
    (i) is directly {concluded} from  Theorem~\ref{E-E-case1}.  We now deal with (ii) {and apply}  Theorem~\ref{E-E-case2}. Then \eqref{embee0_2} is equivalent to $\varphi(t)\gtrsim t^\frac{d}{p_1}$, $t\geq 1$. On the other hand, $\varphi\in\mathcal{G}_{p_1}$ implies $\varphi(t)\leq t^\frac{d}{p_1}$, $t\geq 1$. Thus, we have \eqref{e-f-c1}. Note that $\alpha_j\sim 1$, since
    \begin{align*}
        1\leq \alpha_j=\sup_{\nu\leq j}\varphi(2^{-\nu})^{-\frac{p_1}{p_2}}2^{-\nu\frac{d}{p_2}}
        \sim\max_{\nu=0,\cdots,j}\varphi(2^{-\nu})^{-\frac{p_1}{p_2}}2^{-\nu\frac{d}{p_2}}
        \leq\max_{\nu=0,\cdots,j}2^{\nu\frac{d}{p_2}}2^{-\nu\frac{d}{p_2}}=1
    \end{align*}
    by $\varphi(2^{-\nu})\geq 2^{-\nu\frac{d}{p_1}}$, $\nu\in\no$, due to $\varphi\in\mathcal{G}_{p_1}$. Hence, \eqref{e-f-c2} and \eqref{e-f-c3} are obtained. 
\end{proof}

{
For convenience, we formulate the special cases given by \eqref{E=Mf} and \eqref{E-Lp} in the source and target space, respectively.

\begin{corollary}\label{Cor-Ef-F-spec}
  Let $s\in\real$, $0<p_i<\infty$, $i=1,2$, and $0<q \leq\infty$. Assume that $\varphi\in\mathcal{G}_{p_1}$ and $\varphi$ satisfies \eqref{intc} when $q<\infty$.
    \begin{enumerate}[\bfseries\upshape  (i)]
    \item If $p_1\geq p_2$ and $1 < p_2<\infty$, then 
\begin{align}\label{E-Lr}
        F^{s,\varphi}_{p_1,q}(\rd) =  \mathcal{E}^{s}_{\varphi, p_1,q}(\rd) \eb L_{p_2}(\rd),
    \end{align}
if, and only if, \eqref{E-F-1} is satisfied and
 \begin{align*}
        \left\{2^{-j(s+\frac{d}{p_2})}\varphi(2^{-j})^{-1}\right\}_{j\in\nat}\in\ell_\infty
    \end{align*}
    with
    \begin{align*}
      s>0\qquad\text{or}\qquad s=0\quad\text{and}\quad q\leq 2.
    \end{align*}
    \item If $p_1\geq p_2$ and $1 < p_1<\infty$, then 
\begin{align}\label{Mf-F}
        \mathcal{M}_{\varphi,p_1}(\rd) \eb F^{s}_{p_2,q}(\rd),
    \end{align}
if, and only if, \eqref{E-F-1} is satisfied and
 \begin{align*}
        \left\{2^{j(s-\frac{d}{p_2})}\varphi(2^{-j})^{-1}\right\}_{j\in\nat}\in\ell_\infty
    \end{align*}
    with
    \begin{align*}
      s<0\qquad\text{or}\qquad s=0\quad\text{and}\quad q\geq 2.
    \end{align*}
  \item Assume that $p_1<p_2$, $1<p_2<\infty$, and \eqref{embe-a-4'} holds. 
Then
        \eqref{E-Lr} is continuous if, and only if, \eqref{e-f-c1} and $s_1\geq \frac{d}{\rphi{\varphi}}\left(1-\frac{p_1}{p_2}\right)$ are satisfied.
        \item
    Assume that $p_1<p_2$, $1<p_1<\infty$, and \eqref{embe-a-4'} are satisfied. Then \eqref{Mf-F} is continuous  if, and only if, \eqref{e-f-c1} and $s_2\leq -\frac{d}{\rphi{\varphi}}\left(1-\frac{p_1}{p_2}\right)$ are satisfied.
    \end{enumerate}
\end{corollary}

\begin{proof}
We apply Corollary~\ref{Cor-Ef-F}. 
Here the result for \eqref{E-Lr} is a consequence of \eqref{E-Lp}, that is, when $s_2=0$ and $q_2=2$. Again the extension to \eqref{E-Lr} relies on \eqref{E-Lp} (and obvious counterparts for parts (a) and (b) in Corollary~\ref{Cor-Ef-F}(ii) could be given). The argument for \eqref{Mf-F} is parallel, now based on \eqref{E=Mf}, choosing then $s_1=0$ and $q_1=2$ in Corollary~\ref{Cor-Ef-F}.
\end{proof}

\begin{remark}
Let $1<p_2\leq p_1<\infty$. Then parts (i) and (ii) of the above corollary imply, in particular, that $\mathcal{M}_{\varphi,p_1}(\rd)\eb L_{p_2}(\rd)$ if, and only, $\varphi(t) \gtrsim t^{d/p_2}$, $t>0$, in good coincidence with \eqref{elme2}. Again, parts (iii) and (iv) amount to the finding that there is no embedding $\mathcal{M}_{\varphi,p_1}(\rd)\eb L_{p_2}(\rd)$ in case of $1<p_1<p_2<\infty$ if \eqref{embe-a-4'} and $\rphi{\varphi}<\infty$, recall also our discussion in Remark~\ref{R-M-M-emb-2}. 
\end{remark}

}

\begin{corollary}\label{Cor-F-Ef}
    Let $s_i\in\real$, $0<p_i<\infty$ and $0<q_i\leq\infty$, $i=1,2$. Assume that $\varphi\in\mathcal{G}_{p_2}$ and $\varphi$ satisfies \eqref{intc} when $q_2<\infty$. We consider the embedding
    \begin{align}\label{F-E}
        F_{p_1,q_1}^{s_1}(\rd)\eb \mathcal{E}^{s_2}_{\varphi, p_2,q_2}(\rd)=F^{s_2,\varphi}_{p_2,q_2}(\rd).
    \end{align}
    \begin{enumerate}[\bfseries\upshape  (i)]
    \item If $p_1\geq p_2$, then \eqref{F-E} is continuous if, and only if,
    \begin{align*}
        \sup_{t\geq 1}t^{-\frac{d}{p_1}}\varphi(t)<\infty
    \end{align*}
    and
    \begin{align*}
        \left\{2^{j(s_2-s_1+\frac{d}{p_1})}\varphi(2^{-j})\right\}_{j\in\nat}\in\ell_\infty
    \end{align*}
    with
    \begin{align*}
        s_1>s_2\qquad\text{or}\qquad s_1=s_2\quad\text{and}\quad q_1\leq q_2.
    \end{align*}
    \item If $p_1<p_2$, then \eqref{F-E} is continuous if, and only if,
    \begin{align}\label{f-e-c}
        \left\{2^{j(s_2-s_1+\frac{d}{p_1})}\varphi(2^{-j})\right\}_{j\in\nat}\in\ell_\infty.
    \end{align}
    \end{enumerate}
\end{corollary}

\begin{proof}
    Let $\varphi_1(t)=t^\frac{d}{p_1}$, $t>0$, then $\MfFa(\rd)= F^{s_1}_{p_1,q_1}(\rd)$. Now (i) is a consequence of Theorem~\ref{E-E-case1}, while (ii) follows from  Theorem~\ref{E-E-case2}: note that \eqref{embee0_2} reads as
    \begin{align*}
        \sup_{\nu\leq 0}2^{\nu\frac{d}{p_2}}\varphi(2^{-\nu})\leq C,
    \end{align*}
    but this is always true due to $\varphi\in\mathcal{G}_{p_2}$. Similarly, we have
    \begin{align*}
        \alpha_j\sim\max_{\nu=0,\cdots,j}2^{\nu\frac{d}{p_2}}\varphi(2^{-\nu})=2^{j\frac{d}{p_2}}\varphi(2^{-j}),
    \end{align*}
    such that \eqref{embe-a-3} can be rewritten as \eqref{f-e-c}.     
    Moreover, \eqref{embe-a-4} always holds in this setting since $\rphi{\varphi_1}=p_1$. Hence, we finish the proof.
\end{proof}

\begin{remark}  One could explicate the counterpart of Corollary~\ref{Cor-F-Ef} in the sense of Corollary~\ref{Cor-Ef-F-spec}, using $L_{p_1}(\rd) = F^0_{p_1,2}(\rd)$, $1<p_1<\infty$, and $\mathcal{M}_{\varphi,p_2}(\rd)=\mathcal{E}^0_{\varphi,p_2,2}(\rd)$, $1<p_2<\infty$, in \eqref{F-E}, respectively. We leave it for the reader.
 \end{remark}

 \begin{corollary}
   Let $0<p_i<\infty$, $s_i\in\real$, $0<q_i\leq\infty$, $i=1,2$, and $\varphi\in \mathcal{G}_{\max(p_1,p_2)}$. We assume without loss of generality that $\varphi(1)=1$. Assume in addition that \eqref{intc}  holds when $q_i<\infty$.
   \begin{enumerate}[\bfseries\upshape  (i)]
\item     Let $p_1\geq p_2$, $\varphi\in \mathcal{G}_{p_1}$.    Then
    	\begin{align}\label{vf1=vf2}
			\mathcal{E}^{s_1}_{\varphi, p_1,q_1}(\rd)\eb  \mathcal{E}^{s_2}_{\varphi, p_2,q_2}(\rd)
		\end{align}
		if, and only if, 
		\begin{align*}                  & s_1 >  s_2,  \qquad \text{or} \\
& s_1=s_2\quad \text{and}\quad q_1\le q_2. 
		\end{align*}
\item
                Let $p_1 < p_2$, $\varphi\in \mathcal{G}_{p_2}$, and $\varrho = \frac{p_1}{p_2}<1$.
    \begin{enumerate}[\bfseries\upshape  (a)]
\item   
  Then \eqref{vf1=vf2} holds if
		\begin{align}\label{embe-a-2'}
             \left\{2^{j(s_2-s_1+\frac{\nd}{\rphi{\varphi}}(1-\varrho))}\varphi(2^{-j})^{1-\varrho}\right\}_{j\in\nat} \in \ell_\infty.
                  \end{align}
 \item If  \eqref{vf1=vf2} holds, then $s_1\geq s_2$.
   \item
                  Assume that \eqref{embe-a-4'} is satisfied. Then the embedding \eqref{vf1=vf2} holds if, and only if,  $s_1\geq s_2$.
    \end{enumerate}
    \end{enumerate}
 \end{corollary}

We add an immediate consequence of our above findings in terms of the spaces $\btt(\rd)$ and discuss it below.

  \begin{corollary}\label{Cor-B-B-1}
Let $0<p_i<\infty$, $s_i\in\real$, $0<q_i\leq \infty$ and $\varphi_i\in\mathcal{G}_{p_i}$, $i=1,2$.  
    We assume without loss of generality that $\varphi_1(1)=\varphi_2(1)=1$. Let $\alpha_j=\sup_{\nu\leq j}\frac{\varphi_2(2^{-\nu})}{\varphi_1(2^{-\nu})^\varrho}$, $j\in\no$, and $\varrho=\min\left(1,\frac{p_1}{p_2}\right)$. 
    \begin{enumerate}[\bfseries\upshape  (i)]
      \item
        Then we have the continuous embedding 
    \begin{align}\label{bfbf}
        \btta(\rd)\eb\bttb(\rd)
    \end{align}
    if \eqref{embee0_2} and
    \begin{equation}
      \label{suff-B-B}
      \left\{ 2^{j(s_2-s_1) }\alpha_j \varphi_1(2^{-j})^{\varrho-1}\right\}_{j\in\nat} \in \ell_{q_2}.
    \end{equation}
  \item
    If \eqref{bfbf} is satisfied, then \eqref{embee0_2} is satisfied and
     \begin{equation}
      \label{nec-B-B}
      \left\{ 2^{j(s_2-s_1)} \alpha_j \varphi_1(2^{-j})^{\varrho-1}\right\}_{j\in\nat} \in \ell_{\infty}.
    \end{equation}
    In particular, when $q_2=\infty$, then \eqref{bfbf} holds if, and only if, 
    \eqref{embee0_2} and \eqref{nec-B-B} are satisfied.
  \item
  The embedding  \eqref{bfbf} is never compact.
\end{enumerate}
  \end{corollary}

  \begin{proof}
    We combine Theorems~\ref{rswm} and \ref{main-n} to get the sufficiency {of \eqref{embee0_2} and} \eqref{suff-B-B} via
    \begin{equation*}
      \btta(\rd) \eb B^{s_1, \varphi_1}_{p_1,\infty}(\rd) = \mathcal{N}^{s_1}_{\varphi_1, p_1,\infty}(\rd) \eb \MfBb(\rd) \eb \bttb(\rd),
  \end{equation*}
  where \eqref{emfb3} coincides with \eqref{suff-B-B}, while the necessity comes from
  \begin{equation*}
    \MfBa(\rd) \eb \btta(\rd) \eb \bttb(\rd) \eb B^{s_2,\varphi_2}_{p_2,\infty}(\rd)\eb \mathcal{N}^{s_2}_{\varphi_2,p_2,\infty}(\rd)
    \end{equation*}
and the coincidence of \eqref{nec-B-B} with \eqref{emfb3} in this case. The non-compactness is concluded in the same way by Theorem~\ref{main-n}.  
  \end{proof}

  \begin{remark}
Obviously the result above is not yet sharp and complete. Apart from the cases when $p_i=q_i$, $i=1,2$, and thus Corollary~\ref{main-fw} can be rewritten for $\btt$ instead of $\ftt$, only some more special cases seem to be known in literature, cf. \cite[Thm.~4.14]{hl23} related to the cases $s_1-s_2 = \nd\left(\frac{1}{p_1} - \frac{1}{p_2}\right)>0$,  $\varphi_1(t) t^{-\frac{\nd}{p_1}} =\varphi_2(t) t^{-\frac{\nd}{p_2}}$, $t>0$, and \cite[Prop.~4.20]{hl23} related to the setting $s_1=s_2$, $p_1\geq p_2$,  $q_1\leq q_2$, $\varphi_1 =\varphi_2$. A more detailed discussion, like in       \cite[Theorem~2.5]{YHSY} dealing with the special case of spaces of type $B^{s, \tau}_{p, q}(\rd)$, recall Remark~\ref{rmk-tau-spaces}, seems not yet available. We postpone this question to a later occasion.
  \end{remark}


\subsection{Embeddings into $L_\infty(\rd)$}

We first recall some known results.

\begin{theorem}\label{eb-lif}
    \begin{enumerate}[\bfseries\upshape  (i)]
        \item Let $s\in\real$, $0<p\leq \infty$ and $0<q\leq \infty$. Then 
        \begin{align*}
            F_{p,q}^s(\rd)\eb L_\infty(\rd)
        \end{align*}
        if, and only if,
        \begin{align*}
            s>\frac{d}{p} \qquad\quad \text{or}\qquad\quad s=\frac{d}{p}\quad\text{and}\quad 0<p\leq 1.
        \end{align*}
        \item Let $s\in\real$, $0<p< u<\infty$ and $0<q\leq \infty$. Then
        \begin{align*}
            \MF(\rd)\eb L_\infty(\rd)
        \end{align*}
        if, and only if,
        \begin{align*}
            s>\frac{d}{u}.
        \end{align*}
        \item Let $s\in\real$, $0<p<\infty$, $0<q\leq \infty$ and $\varphi\in\Gp$. 
        Let $q'$ be given by $\frac{1}{q'}=\left(1-\frac{1}{q}\right)_+$, as usual. Then
        \begin{align*}
            \MfB(\rd)\eb L_\infty(\rd)
        \end{align*}
        if, and only if,
        \begin{align*}
            \left\{2^{-js}\varphi(2^{-j})^{-1}\right\}_{j\in\no}\in\ell_{q'}.
        \end{align*}
    \end{enumerate}
\end{theorem}

\begin{remark}
\begin{enumerate}[\bfseries\upshape  (i)]
    \item One may find (i) in \cite{t20} and references therein. (ii) has been proved in \cite[Proposition 3.8]{hs14}. (iii) is a very recent result in \cite[Theorem 3.4]{hms23} with forerunner in \cite[Corollary~5.10]{hms22}. 
    \item In (i) and (ii), $L_\infty(\rd)$ can be replaced by $C(\rd)$, where $C(\rd)$ stands for the Banach space consisting of bounded uniformly continuous functions. As for (iii), this is also true, if one replaces $L_\infty(\rd)$ in the proof of \cite[Theorem 3.4]{hms23} by $C(\rd)$.
    \end{enumerate}
\end{remark}

We have the following Proposition in \cite[Lemma~4.12]{hl23} (note the different definition in view of $\varphi$, recall Remark~\ref{rmk-tau-spaces}).

\begin{proposition}[{\cite[Lemma~4.12, Prop.~4.17]{hl23}}]\label{af-f}
  Let $s\in\real$, $0<p<\infty$, $0<q\leq\infty$ and $\varphi\in\Gp$. Assume in addition that $\varphi$ satisfies \eqref{intc} when $q<\infty$ and $A_{p,q}^{s,\varphi}(\rd)$ denotes $F_{p,q}^{s,\varphi}(\rd)$. Assume that $\{2^{-js} \varphi(2^{-j})^{-1}\}_{j\in\nat} \in \ell_1$. Then 
    \begin{align}\label{awvp-f}
        A_{p,q}^{s,\varphi}(\rd)\eb C(\rd).
    \end{align}
    In particular, when $s>\frac{\nd}{p}$, then \eqref{awvp-f} is satisfied.
\end{proposition}

In case of $\MfF(\rd)=\ftt(\rd)$ this could also be obtained using Theorem~\ref{Nphi-Ephi} and Theorem~\ref{eb-lif}(iii) with $q=\infty$.

\begin{lemma}\label{E-C-nec}
  Let $s\in\real$, $0<p<\infty$, $0<q\leq\infty$ and $\varphi\in\Gp$. Assume in addition that $\varphi$ satisfies \eqref{intc} when $q<\infty$. Then
  \begin{equation}\label{mf-lif}
\MfF(\rd) \eb C(\rd)
  \end{equation}
implies $\{2^{-js} \varphi(2^{-j})^{-1}\}_{j\in\nat} \in \ell_{p'}$.
  \end{lemma}

  \begin{proof}
Here we use ideas of the proofs of \cite[Prop. 3.8]{hs14} and \cite[Thm.~3.4]{hms23}. 
    Let first $0<p\leq 1$, that is $p'=\infty$. By Remark~\ref{equnv}, $\mathcal{E}_{\varphi_0,p,\infty}^0(\rd)=F_{\infty,\infty}^0(\rd)$ when $\varphi_0\equiv1$. Thus \eqref{mf-lif} implies that
    \begin{align*}
        \MfF(\rd)\eb C(\rd)\eb F_{\infty,\infty}^0(\rd)=\mathcal{E}_{\varphi_0,p,\infty}^0(\rd).
    \end{align*}
    Thus, $\left\{2^{-js}\varphi(2^{-j})^{-1}\right\}_{j\in\no}\in\ell_{\infty}$ is necessary for the embedding \eqref{mf-lif} by Theorem~\ref{E-E-case1}.

Now assume $p>1$, that is $p'<\infty$. Let $\gamma=\{\gamma_j\}_{j\in\no}\in\ell_{p}$, $\|\gamma\mid\ell_p\|\leq 1$. 
    For an appropriately chosen sequence $\{m_j\}_{j\in\no}\in\mathbb{Z}$, we consider a shrinking sequence of cubes $Q_{0,m_0}\supsetneq Q_{1,m_1}\supsetneq Q_{2,m_2}\supsetneq\cdots$ with $x_0\in\bigcap_j Q_{j,m_j}$. 
    We define
    \begin{align*}
        \lz_{j,m_j}(\gamma)=\gamma_j 2^{-j s}\varphi(2^{-j})^{-1},\quad j\in\no,
    \end{align*}
    and
    \begin{align*}
        \lzjm(\gamma)=\left\{
        \begin{array}{lcl}
        \lz_{j,m_j}(\gamma),      &     & {\text{if}\quad j\in\no\quad \text{and}\quad m=m_j,}\\
        0,    &     & {\text{otherwise}.}
        \end{array} \right.
    \end{align*}
    By Theorem~\ref{nwdp}, we put 
    \begin{align*}
        f_\gamma=\sum_{j=0}^\infty\sum_{m\in\zd}\lzjm(\gamma) 2^{-\frac{j\nd}{2}} \psi_{j,m}^G,
    \end{align*}
    for some fixed multi-index $G$. We claim that 
    \begin{align*}
        \|f_\gamma\mid\MfF(\rd)\|\lesssim\|\gamma\mid\ell_p\|.
    \end{align*}
    In fact, 
    \begin{align*}
        &\varphi(2^{-\nu})2^{\nu\frac{d}{p}}\left[\int_{Q_{\nu,k}}\left(
        \sum_{j=0}^\infty\sum_{m\in\zd}(2^{j s}|\lzjm|\chi_{j,m}(x))^q\right)^\frac{p}{q}\dint x\right]^\frac{1}{p}\\
        &\qquad =\varphi(2^{-\nu})2^{\nu\frac{d}{p}}\left[\int_{Q_{\nu,k}}\left(
        \sum_{j=0}^\infty2^{j s q}|\lz_{j,m_j}|\chi_{j,m_j}(x)^q\right)^\frac{p}{q}\dint x\right]^\frac{1}{p}\\
        &\qquad =\varphi(2^{-\nu})2^{\nu\frac{d}{p}}\left[\int_{Q_{\nu,k}}\left(
          \sum_{j=0}^\nu 2^{j s q}2^{-j s q}\varphi(2^{-j})^{-q}|\gamma_j|^q\chi_{j,m_j}(x)^q + \right.\right.\\
      &\qquad\qquad\qquad\qquad \left.\left. +\sum_{j=\nu+1}^\infty 2^{j s q}2^{-j s q}\varphi(2^{-j})^{-q}|\gamma_j|^q\chi_{j,m_j}(x)^q\right)^\frac{p}{q}\dint x\right]^\frac{1}{p}\\
        &\qquad \lesssim I_1 + I_2.
          \end{align*}
          Using the condition \eqref{Gp-def} with $t=2^{-j}$, $s=2^{-\nu}$, $j\geq \nu$, we can further estimate $I_2$ by
      
          \begin{align*}
        I_2 & \leq \varphi(2^{-\nu})2^{\nu\frac{\nd}{p}}\left[\int_{Q_{\nu,k}}\left(
        \sum_{j=\nu}^\infty\varphi(2^{-\nu})^{-q}2^{(j-\nu)\frac{dq}{p}}|\gamma_j|^q\chi_{j,m_j}(x)^q\right)^\frac{p}{q}\dint x\right]^\frac{1}{p}\\
        &\qquad = \left[\int_{Q_{\nu,m_\nu}}\left(
        \sum_{j=\nu}^\infty2^{j\frac{dq}{p}}|\gamma_j|^q\chi_{j,m_j}(x)^q\right)^\frac{p}{q}\dint x\right]^\frac{1}{p}\\
        &\qquad \lesssim \left(\sum_{\ell=\nu}^\infty 2^{-\ell d}\left(\sum_{j=\nu}^\ell 2^{j d \frac{q}{p}}|\gamma_j|^q\right)^\frac{p}{q}\right)^\frac{1}{p}\\
        &\qquad \lesssim \left(\sum_{j=\nu}^\infty |\gamma_j|^p \left(\sum_{\ell=j}^\infty 2^{-(\ell-j)\nd\frac{q}{p}}\right)^{\frac{p}{q}} \right)^\frac{1}{p}\\
        &\qquad \lesssim \|\gamma\mid\ell_p\|.
          \end{align*}
Here we used the special nested construction of the cubes $\{Q_{j, m_j}\}_{j\in\no}$.
          
          Concerning $I_1$ we use the monotonicity of $\varphi$, that is, $\varphi(2^{-j}) \geq \varphi(2^{-\nu})$, $j\leq \nu$, to argue that
\begin{align*}
  I_1 & \lesssim    \varphi(2^{-\nu})2^{\nu\frac{d}{p}}\left[\int_{Q_{\nu,k}}\left(
          \sum_{j=0}^\nu \varphi(2^{-j})^{-q}|\gamma_j|^q\chi_{j,m_j}(x)^q \right)^\frac{p}{q}\dint x\right]^\frac{1}{p}\\
& \lesssim 2^{\nu\frac{d}{p}}\left[\int_{Q_{\nu,k}}\left(
          \sum_{j=0}^\nu |\gamma_j|^q\chi_{j,m_j}(x)^q \right)^\frac{p}{q}\dint x\right]^\frac{1}{p}\\
& \lesssim  \left(\sum_{j=0}^\nu |\gamma_j|^p \right)^{\frac1p} \ \lesssim \ \|\gamma| \ell_p\|.        
\end{align*}

    Thus, by \eqref{mf-lif},
    \begin{align*}
        \left|\sum_{j=0}^\infty\lz_{j,m_j}(\gamma)\right|\leq c\, \|f_\gamma\mid C(\rd)\|
        \leq c'\, \|\gamma\mid \ell_p\|\leq c'.
    \end{align*}
    In consequence, since this holds for all $\gamma\in\ell_p$ with $\|\gamma\mid \ell_p\|\leq 1$,
    \begin{align*}
        \left\|\left\{2^{-j s}\varphi(2^{-j})^{-1}\right\}_{j\in\no}\mid \ell_{p'}\right\|
        &=\sup\left\{\left|\sum_{j=0}^\infty\gamma_j2^{-j s}\varphi(2^{-j})^{-1}\right|:
        \|\gamma\mid \ell_p\|\leq 1\right\}\\
        &=\sup\left\{\left|\sum_{j=0}^\infty\lz_{j,m_j}(\nu)\right|:
        \|\gamma\mid \ell_p\|\leq 1\right\}\leq C<\infty.
    \end{align*}    
    \end{proof}

    \begin{proposition}\label{Prop-E-Linf}
      Let $s\in\real$, $0<p<\infty$, $0<q\leq\infty$ and $\varphi\in\Gp$. Assume in addition that $\varphi$ satisfies \eqref{intc} when $q<\infty$. Assume that $\rphi{\varphi}<\infty$.
      \begin{enumerate}[\bfseries\upshape  (i)]
      \item
              Then \eqref{mf-lif} holds if 

              \begin{equation}\label{E-Linf}
      s> \frac{\nd}{\rphi{\varphi}}.
              \end{equation}
              In case of $s=\frac{\nd}{\rphi{\varphi}}$, the embedding \eqref{mf-lif} holds if $0<p=\rphi{\varphi}\leq 1$ and
              \begin{equation}\label{global-cond}
                \lim_{t\to\infty} t^{-\frac{\nd}{\rphi{\varphi}}} \varphi(t) \geq c>0.
                \end{equation}
  \item
    Conversely,  { under the assumption }
\begin{equation}\label{local-cond}
                \lim_{t\to 0} t^{-\frac{\nd}{\rphi{\varphi}}} \varphi(t) \leq C<\infty
\end{equation}
    the embedding \eqref{mf-lif} implies $s\geq \frac{\nd}{\rphi{\varphi}}$ if $0<p\leq 1$, and $s>\frac{\nd}{\rphi{\varphi}}$ if $p> 1$.
  \item
    If  \eqref{global-cond}  and \eqref{local-cond}
are satisfied, then
\begin{equation}
  \MfF(\rd) \eb C(\rd) \qquad \text{if, and only if,}\qquad \begin{cases} s>\frac{\nd}{\rphi{\varphi}}, & \text{or} \\
     s=\frac{\nd}{\rphi{\varphi}}, & \text{and}\quad 0<p=\rphi{\varphi}\leq 1.\end{cases}
\end{equation}
    \end{enumerate}
    \end{proposition}

    \begin{proof}
      Assume first that $s>\frac{\nd}{\rphi{\varphi}}$, then $\varphi(2^{-j})^{-1}  \leq c\ 2^{j \frac{\nd}{\rphi{\varphi}}}$ implies that $\{2^{-js} \varphi(2^{-j})^{-1}\}_{j\in\nat} \in \ell_1$ which by Proposition~\ref{af-f}  yields \eqref{mf-lif}. In case of $s=\frac{\nd}{\rphi{\varphi}}$ and $0<p=\rphi{\varphi}\leq 1$ the sufficiency of \eqref{mf-lif} is verified by the embeddings
      \[
        \mathcal{E}^{\frac{\nd}{\rphi{\varphi}}}_{\varphi,p,q}(\rd) \eb \mathcal{E}^{\frac{\nd}{\rphi{\varphi}}}_{\rphi{\varphi},p,q}(\rd) = F^{\frac{\nd}{\rphi{\varphi}}}_{\rphi{\varphi},q}(\rd) \eb C(\rd),         \]
where the first embedding is covered by Theorem~\ref{E-E-case1} with $s_1=s_2=s$, $q_1=q_2=q$, $p_1=p_2=p$, $\varphi_1=\varphi$ and $\varphi_2(t)=t^{\nd/\rphi{\varphi}}$. Then \eqref{embee00_1} is always covered by the definition of $\rphi{\varphi}$, while \eqref{embee0_1} benefits from \eqref{global-cond}. The second embedding { uses 
$p=\rphi{\varphi}$,} and the last embedding is the classical result  recalled in Theorem~\ref{eb-lif}(i).  This proves part (i).

Part (ii) is obtained by Lemma~\ref{E-C-nec} {and \eqref{local-cond}}, {since then $\varphi(2^{-j}) \sim 2^{-j\frac{\nd}{\rphi{\varphi}}}$, $j\in\nat$, which leads to $\{ 2^{-j(s -\frac{\nd}{\rphi{\varphi}})} \}_{j\in\nat} \in\ell_{p'}$, as explicated in (ii).}

Finally, concerning (iii), the sufficiency is clear by (i). On the other hand, our assumptions \eqref{local-cond} and \eqref{global-cond} imply that $\varphi(t) \sim t^{\nd/\rphi{\varphi}}$, $t>0$, such that $\mathcal{E}^{s}_{\varphi,\rphi{\varphi},q}(\rd) = F^s_{\rphi{\varphi},q}(\rd)$. Moreover, $\mathcal{E}^{s}_{\varphi,\rphi{\varphi},q}(\rd)\eb \MfF(\rd) $ since $\rphi{\varphi}\geq p$. Thus the continuity of $\MfF(\rd) \eb C(\rd)$ implies $\mathcal{E}^{s}_{\varphi,\rphi{\varphi},q}(\rd) = F^s_{\rphi{\varphi},q}(\rd) \eb C(\rd)$ and the rest is covered by Theorem~\ref{eb-lif}(i) again, leading to $s>\frac{\nd}{\rphi{\varphi}}$ or $s=\frac{\nd}{\rphi{\varphi}}$ and $0<p\leq \rphi{\varphi}\leq 1$.\\

{It remains to prove that $p=\rphi{\varphi}$ is necessary for the embedding in case of $s=\frac{\nd}{\rphi{\varphi}}$. We follow the idea in the proof of Proposition 3.8 in \cite{hs14} and construct an $f_0\in \mathcal{E}^{\frac{\nd}{\rphi{\varphi}}}_{\varphi,p,q}(\rd)$, which is not bounded, if $0<p<\rphi{\varphi}\leq 1$. Consider an appropriately  chosen  shrinking sequence  of dyadic cubes $Q_{0,m_0} \varsupsetneq Q_{1,m_1}  \varsupsetneq Q_{2,m_2} \varsupsetneq \dots $ with $x_0\in \bigcap_{j} Q_{j,m_j}$. For any $j\in\no$, let $a_{j,m_j}$ be an atom  for $\mathcal{E}^{\frac{\nd}{\rphi{\varphi}}}_{\varphi,p,q}(\rd)$, according to Definition \ref{defi-atom} and Theorem \ref{atde},  supported in $Q_{j,m_j}$ with $a_{j,m_j}(x_0)=1$, $j\in \no$, and define 
$$
f_0:=\sum_{j=0}^{\infty} \sum_{m\in\zd}\lambda_{j,m}a_{j,m} \quad \text{with} \quad 
\lambda_{j,m}:=\left\{ \begin{array}{ll} 1, & \text{if} \;\; m=m_j, \;\; j\in \no, \\
 0, & \text{otherwise}. \end{array}
 \right.
$$
Then clearly $f_0\not\in C(\rd)$, but $f_0\in \mathcal{E}^{\frac{\nd}{\rphi{\varphi}}}_{\varphi,p,q}(\rd)$. This is a consequence of $\lambda=\{\lambda_{j,m}\}_{j\in\no,m\in\zd}\in {e}^{\frac{\nd}{\rphi{\varphi}}}_{\varphi,p,q}(\rd)$ and can be seen as follows. For convenience, let us assume $q<\infty$, otherwise the modifications of the argument below are obvious.
For  dyadic cubes $P=Q_{\nu,m}$ with $\nu\in \no$, only the case $m=m_{\nu}$ contributes to the supremum in the calculation of $\|\lambda \mid {e}^{\frac{\nd}{\rphi{\varphi}}}_{\varphi,p,q}(\rd)\|$, and we have
\begin{align*}
\varphi(2^{-\nu})\,  2^{\nu\frac{d}{p}}
\left(\int_{Q_{\nu,m_{\nu}}}\left(\sum_{j=0}^{\infty} 2^{j  \frac{\nd}{\rphi{\varphi}}q}
\chi_{j,m_j}(x)\right)^{\frac{p}{q}}\dint x\right)^{\frac{1}{p}} 
& \leq  \varphi(2^{-\nu}) \,2^{\nu\frac{d}{p}}
\left(\sum_{\ell=\nu}^{\infty} 2^{-\ell d}  \left(\sum_{j=0}^{\ell}  2^{j\frac{\nd}{\rphi{\varphi}}q}\right)^{\frac{p}{q}}\right)^{\frac{1}{p}}  \\
&\leq  \varphi(2^{-\nu}) \,2^{\nu\frac{d}{p}}
\left(\sum_{\ell=\nu}^{\infty} 2^{-\ell d}  2^{\ell\frac{\nd}{\rphi{\varphi}} p}   \right)^{\frac{1}{p}}\\   
&\leq  \varphi(2^{-\nu}) \,2^{\nu\frac{d}{p}} 2^{-\nu  d\left(\frac1p-\frac{1}{\rphi{\varphi}}\right)} \leq  \varphi(2^{-\nu})  \,2^{\nu\frac{d}{\rphi{\varphi}}}\leq c,
\end{align*}
where we have used the assumption $p<\rphi{\varphi}$  and \eqref{local-cond}.
Since $\varphi \in \Gp$ and the supremum for the integral is taken when $P=Q_{0,m_0}$, the case of $P=Q_{\nu,m}$ with $-\nu \in \mathbb{N}$ can be estimated from above by the case $\nu=0$ is the above calculations. Therefore we conclude that 
$$
\bigg\| \biggl(\sum_{j=0}^{\infty} 2^{j  \frac{\nd}{\rphi{\varphi}}q}\sum_{m\in\zd} |\lambda_{j,m}|^q \chi_{j,m}(\cdot)  \biggr)^{\frac{1}{q}}  \,\Big|\, \mathcal{M}_{\varphi, p}(\rd) \bigg\|  <\infty,
$$
which completes the proof.
}
\end{proof}

    \begin{remark}
      One can replace $C(\rd)$ in the above results by $L_\infty(\rd)$.
    \end{remark}

    \begin{remark}
        We would like to comment on the `global' condition \eqref{global-cond} in Proposition~\ref{Prop-E-Linf}. It is only needed for the sufficiency result (i) in the limiting case. Moreover, compared with our previous result for the embedding  $\MfB(\rd)\eb C(\rd)$ recalled in Theorem~\ref{eb-lif}(iii), it seems likely that one can remove it in the end. But we were not able to do so at the moment. Moreover, in all the other sharp results recalled in Theorem~\ref{eb-lif} above for convenience, this additional assumption \eqref{global-cond} is always satisfied. Concerning the local condition \eqref{local-cond} we rather suppose that an assumption of that type should appear, as in that respect the (embedding) behaviour of Besov and Triebel-Lizorkin spaces differs occasionally.
    \end{remark}



\bigskip~

{\small
\noindent
Dorothee D. Haroske\\
Institute of Mathematics \\
Friedrich Schiller University Jena\\
07737 Jena\\
Germany\\
{\tt dorothee.haroske@uni-jena.de}\\[4ex]
\noindent
Zhen Liu\\
Institute of Mathematics \\
Friedrich Schiller University Jena\\
07737 Jena\\
Germany\\
{\tt zhen.liu@uni-jena.de}\\[4ex]
\noindent
Susana D. Moura\\
University of Coimbra\\
CMUC, Department of Mathematics\\
EC Santa Cruz\\
3001-501 Coimbra\\
Portugal\\
{\tt smpsd@mat.uc.pt}\\[4ex]
%
Leszek Skrzypczak\\
Faculty of Mathematics \& Computer Science\\
Adam  Mickiewicz University\\
ul. Uniwersytetu Pozna\'nskiego 4\\
61-614 Pozna\'n\\
Poland\\
{\tt lskrzyp@amu.edu.pl}
}

\end{document}